\documentclass[11pt,english]{smfart}
 
\usepackage{ifpdf}
\usepackage{enumerate}
\usepackage{amssymb}
\usepackage{amsmath, amsthm, amsopn, amsfonts}
\ifpdf
\usepackage{graphicx}  \DeclareGraphicsExtensions{.pdf,.mps}

\else
\usepackage{amsfonts}
\usepackage{graphicx}  \DeclareGraphicsExtensions{.ps,.eps}

\fi

\usepackage{graphpap}

\def\Re{\textrm{Re}} 
\def\pasdegrille{\let\grille = \pasgrille}

\def\aat#1#2#3{
\divide \dimen1 by 48 \dimen3=\dimen1 \multiply \dimen1 by #1
\advance \dimen1 by -\dimen3 \divide \dimen1 by 101 \multiply
\dimen1 by 100 \divide \dimen2 by \count11 \multiply \dimen2 by #2
\setbox0=\hbox{#3}\ht0=0pt\dp0=0pt
   \rlap{\kern\dimen1 \vbox to0pt{\kern-\dimen2\box0\vss}}\dimen1= \wd1
\dimen2=\ht1}
\def\pasgrille{
\count12= \dimen1 \divide \count12 by 50 \divide \dimen2 by
\count12 \count11 =\dimen2 \ \divide \dimen1 by 48
\setlength{\unitlength}{\dimen1} \smash{\rlap{\ }} \dimen1= \wd1
\dimen2=\ht1 }
\def\grille{
\count12= \dimen1 \divide \count12 by 50 \divide \dimen2 by
\count12 \count11 =\dimen2 \ \divide \dimen1 by 48
\setlength{\unitlength}{\dimen1}
\smash{\rlap{\graphpaper[1](0,0)(50, \count11)}} \dimen1= \wd1
\dimen2=\ht1 }

\pasdegrille
\def\Re{\textrm{Re}} 
\def\11{{\rm 1~\hspace{-1.4ex}l} }
\def\C{\mathbb C}
\def\R{\mathbb R}
\def\X{\overline{X}}

\def\N{\mathbb N}

\def\W{\mathcal W}

\def\e{e}

\def\eps{\epsilon}
\def\dis{\displaystyle}    

\renewcommand{\theequation}{\arabic{equation}}

\newtheorem{thm}{Theorem}[section]
 
 \newtheorem{lem}[thm]{Lemma}
 
 \newtheorem{prop}[thm]{Proposition}
 
 \theoremstyle{definition}
 
 \theoremstyle{remark}

 \numberwithin{equation}{section}

\newcommand{\rien}[1]{\relax}% DIVERS
\usepackage[colorlinks=true, pdfstartview=FitV, linkcolor=blue, citecolor=red, urlcolor=blue]{hyperref}
\usepackage{color}
\definecolor{gr}{rgb}   {0.,   0.69,   0.23 }
\definecolor{bl}{rgb}   {0.,   0.5,   1. }
\definecolor{mg}{rgb}   {0.85,  0.,    0.85}
\definecolor{gy}{rgb}   {0.8,  0.8,   0.8}
\definecolor{yl}{rgb}   {0.8,  0.7,   0.}
\definecolor{or}{rgb}  {0.7,0.2,0.2}

\renewcommand{\L}{\mathcal{L}} 
\def\smfbyname{by}
\def\refname{References}
\begin{document}
%%\selectlanguage{english}

\thanks{N. Burq and L. Thomann were partially supported by the ANR project "SMOOTH" ANR-22-CE40-0017.}

%    Information for first author
\author{Nicolas Burq}
\address{Laboratoire de Math\'ematiques d'Orsay,
CNRS, Universit\'e Paris--Saclay, B\^atiment 307, F-91405 Orsay Cedex, and Institut Universitaire de France}
\email{nicolas.burq@math.u-psud.fr}

\author{Aur\'elien Poiret}
\address{Lycée Jeanne d'Albret,  6 rue Giraud Teulon, 78100 Saint-Germain-en-Laye}
\email{aurelien.poiret@gmail.com}

\author{ Laurent Thomann }
\address{Université de Lorraine, CNRS, IECL, F-54000 Nancy, France}
\email{laurent.thomann@univ-lorraine.fr}

\title[Bilinear Strichartz estimates and a.s. global solutions for NLS]{Bilinear Strichartz estimates and almost sure global solutions for the nonlinear Schr\"odinger equation}
%\selectlanguage{english}


\begin{abstract}
The purpose of this article is to construct global solutions, in a probabilistic sense, for the nonlinear  Schr\"odinger equation posed on $\R^d$, in a supercritical regime. Firstly, we establish Bourgain type bilinear estimates for the harmonic oscillator     which yields a gain of half a derivative in space for the local theory with randomised initial conditions, for the cubic equation in $\R^3$. Then, thanks to the lens transform, we are able to obtain global in time solutions for the nonlinear Schr\"odinger equation without harmonic potential. Secondly,  we prove a Kato type smoothing estimate for   the  linear  Schr\"odinger equation with harmonic potential. This allows us to consider the  Schr\"odinger equation  with a nonlinearity of odd degree in a supercritical regime, in any dimension~$d\geq 2$.  

\end{abstract} 
\begin{altabstract}
L'objectif de cet article est de construire des solutions globales, en un sens probabiliste, pour l'\'equation  de Schr\"odinger non lin\'eaire pos\'ee sur $\mathbb{R}^d$, en r\'egime surcritique.  Tout d'abord, nous \'etablissons des estimations bilin\'eaires de type Bourgain pour l'oscillateur harmonique. Celles-ci permettent un gain d'une demi-d\'eriv\'ee en l'espace pour la th\'eorie locale avec des conditions initiales al\'eatoires, en ce qui concerne  l'\'equation cubique dans $\mathbb{R}^3$.
Puis,  gr\^ace \`a la transform\'ee de lentille, nous sommes en mesure d'obtenir des solutions globales en temps pour l'\'equation de Schr\"odinger non lin\'eaire sans potentiel harmonique. Dans un second temps, nous prouvons un effet r\'egularisant de type Kato pour l'\'equation de Schr\"odinger lin\'eaire avec potentiel harmonique. Ceci nous permet de consid\'erer l'\'equation de Schr\"odinger avec une non-lin\'earit\'e de degr\'e impair dans un r\'egime surcritique, dans toute dimension~$d\geq 2$.  
\end{altabstract}

\maketitle

\newpage
 
\tableofcontents

 %%%%%%%%%%%%%%%%%%%%%%%%%%%%%%%%%%%%%%%%%%%%%%%%%%%%%%%%%%%%%%%%%%%%%%%%%%%%%%

\section{Introduction and results} 
\subsection{Introduction}
In this article, by the means of random initial conditions, we  construct global solutions to the nonlinear Schr\"odinger equation in a supercritical regime. Namely, in the following we will consider the equation
 \begin{equation} \label{1-sch}  
  \left\{
      \begin{aligned}
    &    i \frac{ \partial U  }{ \partial t } +  \Delta  U    =   \kappa | U |^{p-1} U, \qquad (t,x)\in \R \times \R^d \\
     &  U(0,x) =u_0(x),
      \end{aligned}
    \right.
\end{equation}
where $ \kappa \in \lbrace -1,1 \rbrace $, where $p\geq 3$ is an odd integer, and $d\geq 1$.  \medskip

 The starting points of this  article were two unpublished papers  from the PhD thesis of the second author (in French), see \cite{poiret1,poiret2}. The first result of the present paper concerns the cubic Schr\"odinger equation in dimension $d=3$. In this case the well-posedness is obtained using regularizing properties of random series combined with bilinear estimates for the harmonic oscillator. Then, in the second part of the article we prove global existence results for the nonlinear  Schr\"odinger equation in dimension $d\geq 2$, with a nonlinearity of degree $p \geq 5$. This result is also obtained by taking benefit from stochastic properties of  random series, but combined with the Kato smoothing effect which is established here for the Schr\"odinger equation with harmonic potential. \medskip
 
  Since the pioneering works of Bourgain  \cite{Bourgain1, Bourgain2} and the works of   Tzvetkov~\cite{Tzv, Tzv2} on the Schr\"odinger equation,  the papers of Burq-Tzvetkov \cite{burq4, burq5,burq7} on the wave equation, there have been many contributions to the study of dispersive equations with random initial conditions. Among them, there are results on  the Schr\"odinger equation  by Colliander-Oh~\cite{CoOh}, Benyi-Oh-Pocovnicu~\cite{BOP-1, BOP-survey}, Bourgain-Bulut~\cite{Bourgain-Bulut2, Bourgain-Bulut3}, Nahmod-Oh-Rey Bellet-Staffilani~\cite{NaOhRBSt}, Nahmod-Staffilani~\cite{ NaSt}, Dodson-L\"uhrman-Mendelson~\cite{DoLuMen}, Kilip-Murphy-Visan~\cite{KMV}, Oh-Tzvetkov-Zhang~\cite{4NLS}, Deng-Nahmod-Yue~\cite{DNY1,DeNaYu},  on the wave equation by Oh-Pocovnicu-Tzvetkov~\cite{OhPoTzv}, Burq-Lebeau~\cite{bBL}, de Suzzoni~\cite{suzzoni1, suzzoni2}, Bourgain-Bulut~\cite{Bourgain-Bulut3}, Bringmann~\cite{Br1, Br2, Br3}, Sun-Tzvetkov~\cite{bSuTz}, on  the Benjamin-Ono equation  by Tzvetkov-Visciglia~\cite{TzvVi1, TzvVi2}, Deng-Tzvetkov-Visciglia~\cite{DeTzvVi}, Deng~\cite{De2}, and many others. \medskip

%We also refer to \cite{Br3, De, De2} for more results on dispersive equations with random initial conditions.\medskip

 In some of the more recents works on the topic, some of the material from  \cite{poiret1,poiret2} was applied, extended or generalised, see {\it e.g.} \cite{PRT2, PTV}. Therefore it appeared relevant to rework the material from these papers  and to gather them in a reader accessible version, which is the aim of the present article. \medskip

 The strategy of the proof, already used in \cite{thomann2, BTT, De} and more recently in \cite{BTh, Latocca}, is the following. We first study a nonlinear Schr\"odinger equation with harmonic potential, for which we are able to prove local existence results. Then thanks to the lens transform, this latter equation is conjugated to the equation~\eqref{1-sch} for which we are able to deduce global existence results. \medskip

For $d\geq 1$, we define the harmonic oscillator by
$$ H = -\Delta_{^d} + |x|^2=\sum_{j=1}^d\big(- \frac{\partial^2}{dx^2_j}+x_j^2\big), $$ 
  and we denote  by $\{h_{n},\;n\geq 0\}$ an orthonormal basis of $L^{2}(\R^{d})$ of eigenvectors  of~$H$ (the Hermite functions).   The eigenvalues of $H$ are the $\big\{2(\ell_{1}+\dots+\ell_{d})+d,\; \ell\in \N^{d}\big\}$, and we can order them in a  non decreasing sequence $\{\lambda^2_n, \;n \geq0\}$, repeated according to their multiplicities, and so that    
\begin{equation*}
H h_n = \lambda_n^2 h_n , \quad \quad \forall n \in \N.
\end{equation*}
 It is then easy to observe that $\lambda_n \sim n^{1/(2d)}$, when $n \longrightarrow +\infty$. \medskip

In the following, $ H^s(\R^d)$ and  $ W^{s,p}(\R^d)$ denote the  usual Sobolev spaces. We also   define the harmonic Sobolev spaces, associated to the harmonic oscillator.  The harmonic Sobolev space~$\mathcal{H}^s( \R^d ) $  is defined as the closure of the Schwartz space for the norm
\begin{equation*}
\| u \| _{ \mathcal{H} ^s( \R^d ) } = \| H^{s/2} u \|_ {L^2( \R^d ) }.
\end{equation*}
 Similarly, the space $ \W^{s,p}( \R^d ) $ is defined as the closure of the Schwartz space for the norm
\begin{equation*}
\| u \| _{ \W ^{s,p}( \R^d ) } = \| H^{s/2} u \|_ {L^p( \R^d ) }.
\end{equation*}
 In fact, this latter  norm is a weighted Sobolev norm, because from \cite{sobolev} :
for all ${ 1 < p < + \infty} $ and $ s \geq 0 $, there exists a constant $ C > 0 $ such that 
\begin{equation}\label{1-comparaison} 
\frac{1}{C} \| u \| _{ \W^{s,p}( \R^d ) } \leq \| (-\Delta)^{s/2} u \|_{L^p(\R^d)} + \| \langle x\rangle ^s u \|_{L^p(\R^d ) } \leq C \| u \| _{ \W^{s,p}( \R ^d ) }.
\end{equation}
Since the family  $\{h_{n},\;n\geq 0\}$ forms a Hilbertian basis of $L^2(\R^d)$,  any $ u \in \mathcal{H}^\sigma(\R^d) $ can be written 
\begin{equation*}
u = \sum_{n \in \N } c_n h_n \quad \mbox{where} \quad  \|u\|^2_{\mathcal{H}^\sigma(\R^d)}=\sum_{n \in \N} \lambda_n^{2\sigma} |c_n|^2 < + \infty.
\end{equation*}

Next,  assume that $ (\Omega , \mathcal{T} , P) $ is a given   probability space and that $ (g_n ( \omega ) )_{n \in \N} $ is an identically distributed sequence of independent  complex Gaussian random variables  
$$ g_n \sim \mathcal{N}_\C(0,1),$$
namely,  the density of $ g_n $ is given by $ \frac{1}{\pi} e^{-|z|^2}dL$ where $ dL $ denotes the Lebesgue measure on $ \C $.  Then, for  $ u_0 \in \mathcal{H}^\sigma(\R^d) $, which can be expended as
$$ u_0(x) = \displaystyle{ \sum_{n \in \N} } c_n  h_n(x),$$
we can consider  the application $ \omega \longmapsto u_0( \omega,.) $, from $ ( \Omega , \mathcal{T} ) $ to $ \mathcal{H}^\sigma ( \R^d ) $ that we equip with its Borelian $\sigma$-algebra, defined by 
$$ u_0^{\omega}(x) = \displaystyle{ \sum_{n \in \N} } c_n g_n( \omega ) h_n(x) .$$
 We can check that the application $ \omega \longmapsto u_0^{\omega} $ is in $ L^2 ( \Omega , \mathcal{H}^\sigma( \R^d ) ) $ and we define the probability measure $ \mu $ as the distribution of this random variable. By definition, we have the following equality:
\begin{equation*}
P\big( \omega \in \Omega \;: \Psi ( u_0^{\omega} ) \in A \big) = \mu \big( f \in \mathcal{H}^\sigma ( \R^3) \;: \Psi (f) \in A \big),
\end{equation*}
for any measurable   $\Psi: (\mathcal{H}^\sigma(\R^d) , \mathcal{B}( \mathcal{H}^\sigma(\R^d) ) )\rightarrow (\R, \mathcal{B}(\R) ) $ and any set $ A \in \mathcal{B}(\R) $. \medskip
 
 Let us first state a result which shows that the randomisation does not improve the initial condition in the Sobolev scale, and that it does not improve its spatial localisation.
\begin{thm} \label{2sobolev2} 
For all  $ s \geq 0 $, if  $u_0  \notin \mathcal{H}^s(\R^d)$ then  
$$u^\omega_0 \notin H^s(\R^d)\quad  \text{and}  \quad \langle x\rangle ^s u^\omega_0 \notin L^2(\R^d),    \quad \omega- \text{a.s.}$$
\end{thm}

By \eqref{1-comparaison} it follows that one therefore has $u^\omega_0  \notin \mathcal{H}^s(\R^d)$, $\omega- \text{a.s.}$

\subsection{Almost sure global existence for the cubic Schr\"odinger equation in dimension 3}

Our first results concern the cubic equation 
\begin{equation} \label{1-schrodinger} 
  \left\{
      \begin{aligned}
      &  i \frac{ \partial U  }{ \partial t } +  \Delta  U    =   \kappa | U |^{2} U, \qquad (t,x)\in \R \times \R^3 \\
    &   U(0,x) =u_0(x),
      \end{aligned}
    \right.
\end{equation}
where $ \kappa \in \lbrace -1,1 \rbrace $. Let us first recall some deterministic results concerning \eqref{1-schrodinger}. There exists $T>0$ and a space $ X ^ s_T $ continuously embedded into $ C ^ 0 ([-T, T] ; H ^ s (\R ^ 3)) $ such that: \medskip

$\bullet$  If $ s>1/2 $ (subcritical case) then for all $ R >0$, there exists $ T_R> 0 $ such that if $ \|u_0\|_{H^s(\R^3)} \leq R $, then there exists a unique local solution $ u \in X^s_T $ to the equation~\eqref{1-schrodinger}, see \cite{Ginibre-Velo,ca-wei}. Moreover, the mapping $ u_0 \in B_{H^s}(0,R) \rightarrow u \in X^s_{ T } $ is continuous. This means that the problem is locally well-posed. If $T$ can be chosen equal to $ + \infty $, we say that the problem is globally well-posed.

In the case where the problem is globally well-posed, it is natural to study the behavior of the solution in $ + \infty $ : if for all $ u_0 \in H^s(\R^3) $,  there exists  $ u_+ \in H^s(\R^3) $ such that $ \underset{t \rightarrow + \infty }{\lim} \| u(t) - e^{it \Delta} u_+ \|_{ H^s( \R^3 ) } = 0 $, which is known as scattering. According to \cite{colliander3} it is the case, for $\kappa=1$, as soon as $ s > \frac{4}{5} $.  

$\bullet$ If $ s=1/2$ (critical case), one can prove the local existence of a unique solution for each initial data as in the subcritical case, but the existence time of the solution depends on $u_0$ and not   only on $ \|u_0\|_{H^s(\R^3)} $. Therefore    the globalization problem is a complex problem. See \cite{ca-wei}. \medskip

$\bullet$  If $ s <1/2$ (supercritical case), according to \cite{CCT}, we know that there exists a sequence of real numbers $ t_n \in \R $ which tends to 0 and a sequence of functions $ u_n \in H^s(\R^3) $ such that $ \underset{n \rightarrow + \infty }{\lim} \|u_n\|_{H^s(\R^3)} = 0 $ and $ \underset{n \rightarrow + \infty }{\lim} \|u_n(t_n)\|_{H^s(\R^3)} = + \infty $, where $ u_n(t) $ denotes a solution of the equation~\eqref{1-schrodinger} with initial data $ u_n $. Therefore, the flow of the equation \eqref{1-schrodinger} is not continuous in\;0, which implies that the problem is not well posed and the usual methods do not allow to study  the equation in this situation.  \medskip

 We assume in this section that the eigenfunctions of the harmonic oscillator are given by the tensor eigenfunctions, {\it i.e.} for $n\geq 0$, there exist $(n_1,n_2,n_3) \in \N^3$ such that for $x=(x_1,x_2,x_3) \in \R^3$
\begin{equation} \label{1-tensor}
 h_n (x) = e_{n_1} (x_1)e_{n_2} (x_2) e_{n_3} (x_3) ,
\end{equation}
where  $ ( - \partial^2_{x_j} +x_j^2 ) e_{n_k} = (2n_k+1) e_{n_k} $ and $\lambda_n ^2=2(n_1+n_2+n_3)+3 $. Here  $ ( e_n ) _{ n \in \N } $ denotes the basis of the eigenfunctions of the harmonic oscillator in dimension $ 1 $ and $ (\mu_n^2)_{n \in \N} $ the associated eigenfunction sequence, so that we have  $ \mu _n^2 = 2 n+1 $. Indeed,~$e_{n}$ are given by the formula 
\begin{equation*} 
e_{n}(x)=(-1)^{n}c_{n}\,\e^{x^{2}/2}\frac{{d}^{n}}{{d}x^{n}}
\big(\,\e^{-x^{2}}\,\big),\;\;\text{with}\;\;\;\frac1{c_{n}}
=\big(n\,!\big)^{\frac12}\,2^{\frac{n}2}\,\pi^{\frac14}.
\end{equation*}

Our main result then reads : 

 \begin{thm} \label{1-thm1}
Let $ \sigma \in ]0, \frac{1}{2} [ $ with $ u_0 \in \mathcal{H} ^{\sigma} ( \mathbb{R}^3 ) $ and $ s \in ] \frac{1}{2} , \frac{1}{2} +\sigma [ $, then there exists a set $ \Omega ' \subset \Omega $ satisfying the following conditions:
\begin{enumerate}[$(i)$]
\item $ P( \Omega ' ) > 0$.
 \item For any element $ \omega \in \Omega'$, there exists a unique global solution $ U$ to the equation~\eqref{1-schrodinger} with initial data $ u_0^{\omega} $ such that 
 $$U(t)- e^{it\Delta} u_0^{\omega} \in \mathcal{C}\big(\R; \mathcal{H}^s(\R^3)\big).$$
 \end{enumerate}
Moreover, for any element $ \omega \in \Omega'$, there exist $ L_+, L_ - \in \mathcal{H}^s(\mathbb{R}^3) $ such that 
\begin{equation}\label{scat11}
\begin{aligned}
&\lim_{t \rightarrow + \infty } \big\| U(t) - e^{it\Delta} \big( u_0^{\omega} +L_+\big) \big\|_{ H^s(\mathbb{R}^3) } = 0, \\
&\lim_{t \rightarrow -\infty } \big\| U(t) -  e^{it\Delta} \big( u_0^{\omega} +L_-\big) \big\|_{ H^s(\mathbb{R}^3) } = 0,
\end{aligned}
\end{equation}
and 
\begin{equation}\label{scat12}
\begin{aligned}
&\lim_{t \rightarrow + \infty } \big\| e^{-it\Delta}U(t) -  (u_0^{\omega} + L_+) \big\|_{ \mathcal{H}^s(\mathbb{R}^3) } = 0, \\
&\lim_{t \rightarrow -\infty } \big\| e^{-it\Delta}U(t) - (u_0^{\omega} + L_-) \big\|_{ \mathcal{H}^s(\mathbb{R}^3) } =0.
\end{aligned}
\end{equation}
Finally, when $\eta \longrightarrow 0$ we have 
\begin{equation} \label{1-proba1}
 \mu \big( u_0 \in \mathcal{H}^\sigma {(\mathbb{R}^3)}  \ :\ \mbox{we have global existence and scattering } \big| \; \|u_0\|_{\mathcal{H}^\sigma (\mathbb{R}^3)}   \leq \eta \big) \longrightarrow 1.
\end{equation}
\end{thm}

 Notice that since $e^{\pm i t \Delta}$ does not act on $\mathcal{H}^{\sigma}(\R^3)$, the  properties \eqref{scat11} and~\eqref{scat12}  are different. It is interesting to notice that Theorem~\ref{1-thm1} is not a small data result in the critical $H^{1/2}(\R^3)$ space. More precisely, for $0<\sigma <1/2$, let $u_0 \in \mathcal{H}^\sigma(\R^3) \backslash H^{1/2}(\R^3)$. Then for $\omega \in \Omega'$, $u_0^{\omega}$ satisfies the conclusion of Theorem~\ref{1-thm1} and by Theorem \ref{2sobolev2} we have $u_0^{\omega} \notin H^{1/2}(\R^3)$. Similarly,  for $K\geq 1$ we can define $ \dis [u_0]_K = \sum_{\lambda_n < K } c_n h_n \in \bigcap_{s \geq 0} \mathcal{H}^{s}(\R^3)$. Then Theorem~\ref{1-thm1} applies to this initial condition and we have $\big\| [u_0^{\omega}]_K\big\|_{H^{1/2}(\R^3)} \gg 1$ for ${K \gg 1}$ (we refer to Proposition~\ref{prop54} for a precise statement). This is the counterpart  with some global existence results obtained for the Navier-Stokes equation, see \cite{galla1, galla2, galla3}.

 We can obtain a more quantitative statement in the case of small initial conditions $u_0 \in \mathcal{H}^{\sigma}(\R^3)$.  For ${0 <\alpha <1}$, there exists  a  set $ \Omega_{\alpha} \subset \Omega'$ satisfying $ P( \Omega_{ \alpha} ) \geq 1 - \alpha $, under the condition 
 \begin{equation}\label{quanti}
  \alpha = \dis C_1 \exp\big( -\frac{C_2}{  \|u_0\|^2_{ \mathcal{H}^\sigma(\mathbb{R}^3) } } \big),
  \end{equation} for some universal constants $C_1,C_2>0$.

It is likely that our approach can be extended to any dimension $d\geq 2$. In this case we also gain a half-derivative compared to the deterministic problem. We will not give more details. \medskip

Our result can be extended to more general i.i.d. random variables $(g_n)_{n \in \N}$. Actually, we need the $g_n$ to be centered, that they admit moment at any order and satisfy multilinear chaos estimates as in Proposition~\ref{prop-chaos}. \medskip

 To establish our results, we use the ideas of N. Burq  and N. Tzvetkov developed in \cite{burq4,burq5}   by randomizing the    initial data. This allows us to gain derivatives in $ L^p $ spaces,  for $ p > 2 $.\medskip
 
  The proof essentially relies on two intermediate results:
  \begin{enumerate}[$(i)$]
\item The lens transformation (introduced in \cite{niederer,carles} and used in \cite{BTT,tao2}) which allows us to reduce to proving a local existence theorem on $ ] - \frac{\pi}{4} ; \frac{\pi}{4} [ $ for the Schrödinger equation with harmonic potential (see Section \ref{1-214}).
 \item The existence of a bilinear estimate for the harmonic oscillator (see Section \ref{Sect3}) which allows to gain the half-derivative on the first order terms in $ u_0 $. This estimate is the analogue of the bilinear estimate for the usual Laplacian proved by Bourgain in \cite{bilinear}.
   \end{enumerate}

The results of Theorem \ref{1-thm1} extend the results of \cite{BTT, De} obtained in space dimension $d=1$ and $d=2$. In \cite{PRT2} similar results where obtained in space dimension $d=3$, but the randomisation of the initial condition was different, which implied scattering results in stronger harmonic Sobolev norms.  In the papers \cite{Br1, Br2, BOP,Camps} some scattering results are obtained using a randomisation based on the Littlewood-Paley decomposition of the initial condition.

The existence of a bilinear estimate for the harmonic oscillator being of interest in itself, we state it below in any dimension ${d\geq 2}$. Let us first define the dyadic localisation operators. Let $ \eta \in C_0^\infty( \R ) $ such that $ \eta_{[0,1]} = 1 $ and $ \eta_{[2, + \infty [ } =0$. Set $\psi(x)= \eta(x)-\eta(4x)$.  For dyadic numbers $ N= 2^j $, we define  the following sequence of operators:
\begin{equation}\label{dya}
\Delta_N (u ) = \left\{
    \begin{aligned}
    &      \psi ( \frac{H}{N^2} )  u & \ \mbox{for} \ N \geq 1,   \\
     &    \ \ 0  & \ \mbox{else}.
    \end{aligned}
\right.
\end{equation}
 Notice that if $ \lambda_n \notin [ \frac{N}{2} , \sqrt{2} N ] $ then $ \Delta_N ( h_n ) = 0 $ and that $ \underset{N \; dyadic}{\sum} \Delta_N (u ) = u $.

  \begin{thm} \label{1-bilis} Let $d\geq 2$. Then for any $ \delta \in ]0 , \frac{1}{2} ] $, there exists a constant $ C > 0 $ such that for all dyadic numbers $ N,M \geq 1$ and $u, v  \in \mathcal{S}'(\R^d)$,
 \begin{multline}\label{bilinear}
   \big\| e^{it H} \Delta_N ( v ) \, e^{it  H} \Delta_M (u) \big\|_{L^2 ( [-1 , 1 ] ; L^2 ( \R^d ))  }  \leq \\
 \leq    C  \min  ( N,M)^{ \frac{d-2}{2} }   \left( \frac{ \min ( N,M)  }{ \max ( N,M )}  \right) ^{1/2-\delta }    \| \Delta_N(v) \|_{L^2(\R^d)} \| \Delta_M(u) \|_{L^2(\R^d)}.
\end{multline}
\end{thm}

The first bilinear estimate has been obtained by Bourgain \cite{bilinear} for the Schr\"odinger group on $\R^d$. For bilinear estimates on compact manifolds and application to the well-posedness of nonlinear Schr\"odinger equation, we refer to \cite{burq1}.

We refer to   the work \cite{PTV} of F. Planchon, N. Tzvetkov, and N. Visciglia    for an alternative proof of Theorem~\ref{1-bilis} in the case $d=2$, with the improvement $\delta=0$. This bilinear estimate has   been used in   \cite{PTV} in order to obtain bounds on the growth of norms for the nonlinear Schr\"odinger equation with harmonic potential.   
\subsection{Almost sure global existence for the nonlinear Schr\"odinger equation in any dimension $d\geq 2$}

We now consider  the following Schr\"odinger equation with polynomial nonlinearity
\begin{equation} \label{2schrodinger}  
  \left\{
      \begin{aligned}
       & i \frac{ \partial U  }{ \partial t } +  \Delta  U   =   \kappa | U |^{p-1}U, \qquad (t,x)\in \R \times \R^d \\
        &  U(0,x)=u_0(x),
      \end{aligned}
    \right.
\end{equation}
where $ \kappa \in \lbrace -1,1 \rbrace $ and $ p \geq 5$ denotes an odd integer, and $d\geq 2$. \medskip

 We   prove the following results:
\begin{thm} \label{2thm1}
Let $ u_0 \in \mathcal{H}^{  \frac{d-1}{2} } (  \R ^d ) $ then there exist $ s \in ]  \frac{d}{2} - \frac{2}{p-1}  , \frac{d}{2} [ $ and a set of $ \Omega ' \subset \Omega  $ such that the following conditions are met: 
\begin{enumerate}[$(i)$]
\item  $ P (  \Omega '  ) > 0 $.
\item  For all element $ \omega \in \Omega'$, there is a unique global solution $ U $ to the equation~\eqref{2schrodinger} in the space $ e^{it\Delta} u_0^{\omega} + X ^s  $ with initial data $ u_0^{\omega} $.
\item  For all element  $ \omega \in \Omega ' $, there exists $ L_+ , L_- \in \mathcal{H}^s(\R^d) $ such that 
\begin{equation}\label{2scat11}
\begin{aligned}
&\lim_{t \rightarrow + \infty } \big\| U(t) - e^{it\Delta} \big(u_0^{\omega} +L_+\big) \big\|_{ H^s(\mathbb{R}^d) } = 0, \\
&\lim_{t \rightarrow -\infty } \big\| U(t) -  e^{it\Delta} \big(u_0^{\omega} +L_-\big) \big\|_{ H^s(\mathbb{R}^d) } = 0,
\end{aligned}
\end{equation}
and 
\begin{equation}\label{2scat12}
\begin{aligned}
&\lim_{t \rightarrow + \infty } \big\| e^{-it\Delta}U(t) -  (u_0^{\omega} + L_+) \big\|_{ \mathcal{H}^s(\mathbb{R}^d) } = 0, \\
&\lim_{t \rightarrow -\infty } \big\| e^{-it\Delta}U(t) - (u_0^{\omega} + L_-) \big\|_{ \mathcal{H}^s(\mathbb{R}^d) } =0.
\end{aligned}
\end{equation}
\end{enumerate}

Finally, when $\eta \longrightarrow 0$ we have 
\begin{equation}\label{condi}
 \mu \big( u_0 \in \mathcal{H}^\sigma {(\mathbb{R}^3)}  \ :\ \mbox{we have global existence and scattering } \big| \; \|u_0\|_{\mathcal{H}^\sigma (\mathbb{R}^3)}   \leq \eta \big) \longrightarrow 1.
\end{equation}
\end{thm}

The proof of this result relies on the Kato smoothing effect for the linear Schr\"odinger equation with harmonic potential. In dimension 1, the smoothing effect was used in \cite{BTT} in order to obtain local existence results. We observe here that this approach can be extended in high dimension, in the case of a nonlinearity of degree $ p \geq 5 $ and an initial condition $ u_0 \in \mathcal{H}^{(d-1)/2}(\R^d)$.  \medskip

We refer to \cite{DoLuMen, KMV} for almost sure scattering results for NLS. \medskip

Let us state the smoothing effect for the harmonic oscillator, which has his own interest.

 \begin{thm}\label{thm15}
Let  $ \epsilon \in ] 0 , \frac{1}{2} [ $ then there exists a constant $ C > 0 $ such that for all $ u_0 \in L^2(\R^d) $,
\begin{equation} \label{2effectregularisant1} 
\Big\| \frac{1}{\langle x\rangle ^{1/2-\epsilon} }\sqrt H ^{1/2- 2 \epsilon} e^{itH} u_0 \Big\|_{L^2([- \pi,\pi] \times \R^d)} \leq C \|u_0\|_{L^2(\R^d)},
\end{equation}
and for all $ u_0 \in  \mathcal{H}^{   \frac{d-1}{2} } (\R^d) $,
\begin{equation} \label{2effectregularisant2} \Big\| \frac{1}{\langle x\rangle ^{1/2-\epsilon} } \sqrt{-\Delta} ^{d/2- 2 \epsilon} e^{itH} u_0 \Big\|_{L^2([- \pi,\pi]\times \R^d)} \leq C \|u_0\|_{\mathcal{H}^{   \frac{d-1}{2} } (\R^d)}.
\end{equation}
\end{thm}

The first results on the local smoothing effect were obtained in  \cite{kato,Sjolin,vega, Constantin-Saut}. Here the proof follows the  Doi method \cite{Doi1,Doi2}. See also \cite{yajima1,yajima2, thomann3} for results on the smoothing effect for the Schr\"odinger equation with confining potentials.

The result of Theorem \ref{thm15} is an extension to the well-known smoothing effect for the linear Schr\"odinger group
\begin{equation*} 
\Big\| \frac{1}{\langle x\rangle ^{1/2-\epsilon} }\sqrt{-\Delta}  ^{1/2- 2 \epsilon} e^{it\Delta} u_0 \Big\|_{L^2([- \pi,\pi] \times \R^d)} \leq C \|u_0\|_{L^2(\R^d)},
\end{equation*}
and for all $ u_0 \in  H^{   \frac{d-1}{2} } (\R^d) $,
\begin{equation*}
 \Big\| \frac{1}{\langle x\rangle ^{1/2-\epsilon} } \sqrt{-\Delta} ^{d/2-2 \epsilon} e^{it\Delta} u_0 \Big\|_{L^2([- \pi,\pi]\times \R^d)} \leq C \|u_0\|_{H^{   \frac{d-1}{2} } (\R^d)}.
\end{equation*}

\subsection{Notations}
  In this paper $c,C>0$ denote constants the value of which may change
from line to line. These constants will always be universal, or uniformly bounded with respect to the other parameters.  We denote by  $H=-\Delta+|x|^{2}$ the harmonic oscillator on $\R^d$, and   for $\sigma\in \R$ we define the   Sobolev space~${\mathcal H}^{\sigma}(\R^d)$  by the norm  $\|u\|_{{\mathcal H}^{\sigma}(\R^d)}=\|H^{\sigma/2}u\|_{L^{2}(\R^d)}$.  More generally, we define the spaces~$\W^{\sigma,p}(\R^d)$ by the norm $\|u\|_{\W^{\sigma,p}(\R^d)}=\|H^{\sigma/2}u\|_{L^{p}(\R^d)}$. The Fourier transform is defined by $\mathcal{F}f(\xi)=\int_{\R^d} e^{-ix\xi}f(x)dx$, for $f \in \mathcal{S}(\R^d)$.  We denote by $\langle x\rangle =\sqrt{1+|x|^2}$.

\subsection{Organization of the rest of the article}

 In Section \ref{Sect2}, we present some harmonic analysis results, including study of the lens transform, Strichartz estimates for the Schr\"odinger group with harmonic potential, properties of the Bourgain space, and Hermite functions estimates.   Section \ref{Sect3} is devoted to the proof of the bilinear estimates for the harmonic oscillator. Then, in Section \ref{Sect4}, we prove the local smoothing effect. In Section \ref{Sect5} we  prove that the randomized initial  data does not allow to gain derivatives in the $L^2$ scale and that it does not enjoy better spatial localisation properties (Theorem~\ref{2sobolev2}). Section~\ref{Sect6} is devoted to the fixed point argument for the cubic NLS in dimension 3. In Section \ref{Sect7} we study the regularity of the randomised initial conditions which allows to complete the proof of Theorem~\ref{1-thm1}. In Section~\ref{Sect8} and in Section~\ref{Sect9}, using a fixed point argument, we prove the results on the NLS in dimension $d\geq 2$.

 %%%%%%%%%%%%%%%%%%%%%%%%%%%%%%%%%%%%%%%%%%%%%%%%%%%%%%%%%%%%%%%%%%%%%%%%%%%%%%

\section{Preliminary results}\label{Sect2}
In this section, except in the fourth part, the dimension of space is assumed to be any $d\geq 1$.

\subsection{The lens transformation} \label{1-214}
As in \cite{thomann2,BTT}, we use the lens transform which allows to work with the Schr\"odinger equation with harmonic potential.  More precisely, suppose that $U(s,y)$ is a solution of the problem \eqref{1-sch}. Then the function  $u(t,x)$ defined  for $|t|<\frac{\pi}{4}$ and $x\in\R$ by
\begin{multline}\label{lens}
 u(t,x)=  \mathcal{L} (U)(t,x)
 :=\Big( \frac{1}{ \cos 2t } \Big) ^{d/2} U\big(\frac{\tan(2t)} 2, \frac{x}{\cos(2t)}\big)e^{-\frac{i|x|^2{\rm tan}(2t)}{2}}\\
 :=\mathcal{L}_t (U\mid_{s= \frac{\tan(2t)} 2})(x)
\end{multline}
where
\begin{equation}\label{lensbis}
 \mathcal{L}_t (G)(x)= \Big( \frac{1}{ \cos (2t) } \Big) ^{d/2}  G\big(\frac{x}{\cos(2t)}\big)
e^{-\frac{i|x|^2{\rm tan}(2t)}{2}}\,,
\end{equation}
 solves the problem
 \begin{equation}\label{C}   
  \left\{
      \begin{aligned}
         &i\partial_t u-Hu=\kappa \cos^{ \frac{d}{2}(p-1) -2 }(2t)|u|^{p-1}u,\quad |t|<\frac{\pi}{4},\, x\in\R^d,
       \\  &  u(0,\cdot)  =U_0,
      \end{aligned}
    \right.
\end{equation}
where $H=-\partial^2_x+|x|^{2}$. Similarly, if $U = e^{is \Delta_y} U_0$ is a solution of the linear Schr\"odinger equation, then
\begin{equation}\label{lens2}
u=e^{-itH} U_0=\mathcal{L}(U)
\end{equation}
 is the solution of the linear harmonic Schr\"odinger equation with the same initial data.  In other words, if we denote by $\Psi(s,s') $ the map which sends the data at time~$s'$ to the solution at time $s$ of~\eqref{C}, and by $\Phi(t,t') $ the map which sends the data at time $t'$ to the solution at time $t$ of~\eqref{2schrodinger}, the family  $  (\mathcal{L}_t)_{|t|< \frac \pi 4}$ conjugates the linear and the nonlinear flows: with $t(s) =\frac{\arctan(2s)} 2$, $s(t) = \frac{ \tan (2t)} 2$,
\begin{equation*} 
  \mathcal{L}_{t(s)} \circ e^{i(s-s')\Delta_y}   =e^{-i (t(s) -t(s'))H}\circ  \mathcal{L} _{t(s')}.
\end{equation*}
and 
\begin{equation*}
\mathcal{L}_{t(s)} \circ  \Psi(s,s')  = \Phi(t(s), t(s')) \circ \mathcal{L} _{t(s')}.
\end{equation*}
 
 By \eqref{lens2} we have
\begin{equation}\label{sym1}
e^{-itH} u ( t ,x ) = \Big( \frac{1}{ \cos (2t) } \Big) ^{d/2} e^{it\Delta} u \Big( \frac{\tan (2 t)}{2}  , \frac{x}{ \cos (2 t) } \Big) e^{ - \frac{i |x|^2 \tan (2 t) }{2} }
\end{equation}
and equivalently
\begin{equation}\label{sym2}
e^{it\Delta} u(t,x)= \Big( \frac{1}{\sqrt{1+4t^2}} \Big) ^{d/2} e^{-itH} u \Big( \frac{\arctan(2t)}{2}  , \frac{x}{\sqrt{1+4t^2} } \Big) e^{ \frac{i|x|^2t}{1+4t^2} }.
\end{equation}

%%%%

\subsection{The Strichartz estimates for the harmonic oscillator}

In the following, we will say that a pair $ (q , r) \in [2, + \infty ]^2 $ is admissible if and only if
\begin{equation*}
 ( q, r, d ) \neq ( 2, + \infty , 2 ) \quad \mbox{ and }\quad  \frac{2}{q} = \frac{d}{2} - \frac{d}{r}.
\end{equation*}
We define, for $ s \in \R $ and $ T > 0$ the Strichartz spaces for the Schr\"odinger equation 
\begin{align*}
 &  X^s = \underset{ (q,r) \ admissible }{ \bigcap } L^q( \R ; W^{s,r}(\R^d)),
\\ & X_T^s = \underset{ (q,r) \ admissible }{ \bigcap } L^q( [-T,T] ;  W^{s,r}(\R^d)).
\end{align*}
Similarly, we define the Strichartz spaces for the harmonic oscillator
\begin{align*}
 &  \X^s = \underset{ (q,r)  \ admissible }{ \bigcap } L^q\big( [-\frac{\pi}4,\frac{\pi}4] ; \W^{s,r}(\R^d)\big),
\\ & \X_T^s = \underset{ (q,r)  \ admissible }{ \bigcap } L^q( [-T,T] ; \W^{s,r}(\R^d)).
\end{align*}
The  Strichartz estimates for the usual Schr\"odinger equation can be found in \cite{tao}. Let us state the Strichartz inequalities for the harmonic oscillator.

\begin{prop} \label{1-Strichartz} For any time $ T \geq 0 $, there exists a constant $ C_T > 0 $ such that for any function $ u \in \mathcal{H}^s(\R^d) $ and any function $ F \in L^{q'}([-T,T] ;  \W^{s,r'} (\R^d)) $ with $ ( q, r ) $ admissible
\begin{equation}\label{1-11}
 \| e^{-itH} u \|_{\X_T^s} \leq C_T \| u \|_{ \mathcal{H}^s(\R^d) }, 
\end{equation}
\begin{equation}\label{1-12}
 \bigg| \int _0^t e^{-i(t-s)H} F(s) ds \bigg| _{ \X^s_T} \leq C \| F \|_{ L^{q'}( [-T,T]; \W^{s,r'} (\R^d))}.
\end{equation}
\end{prop}

\begin{proof} The Strichartz estimates follow from the study of the kernel of $e^{itH}$, and we refer to \cite[Section~5]{BTT} for a proof. Alternatively, let us show here how  \eqref{1-11} also follows from the Strichartz estimates for the Schr\"odinger flow with the free Laplacian. By replacing   $ u $ by $ e^{iTH} u $, we see that it is enough to prove the estimate for some $ T> 0 $, for example for $ T = \frac{\pi}{4} $. Similarly, by replacing~$u$ by $H^{\frac{s}{2}}u$, we can restrain the proof to the case where $s= 0$. Thus, if the couple $ (q,r) $ is admissible, by \eqref{sym1} we obtain  
\begin{multline*}
 \| e^{-itH} u \|_{L^q( ] - \frac{\pi}{4} , \frac{\pi}{4} [ ; L^r (\R^d))} =\\
\begin{aligned}
&= \Big\| \Big( \frac{1}{ \cos (2t) } \Big) ^{d/2} e^{it\Delta} u \Big( \frac{\tan (2 t)}{2}  , \frac{x}{ \cos (2 t) } \Big) e^{ - \frac{i |x|^2 \tan (2 t) }{2} }    \Big\|_{L^q( ]  - \frac{\pi}{4} , \frac{\pi}{4} [ ; L^r (\R^d))}\\
&= \| e^{it\Delta} u \|_{L^q( \R ; L^r(\R^d))}.
\end{aligned}
\end{multline*}
Then, we can use the Strichartz estimates for the Laplacian to conclude.  

The proof of \eqref{1-12}, which relies on a duality argument,  is the same as the proof of the Strichartz estimates for the Laplacian that can be found in \cite[Section 2.3]{tao}.  
\end{proof}

  \begin{prop}\label{tvb1}
If $ s \geq 0$ then there exists a constant $ C > 0 $ such that for all $ p \in [1,+\infty ] $ and $ q \in [1,+\infty ] $ satisfying $ \frac{2}{p} + \frac{d}{q} - \frac{d}{2} \leq 0 $, we have for any $ u \in L^p( [-\frac{\pi}4,\frac{\pi}4] ; \W^{s,q} ( \R^d )) $,
\begin{equation*}
\| U \|_{ L^p( \R ; W^{s,q} ( \R^d ) ) } \leq C \| u \| _{L^p(  [-\frac{\pi}4,\frac{\pi}4] ; \W^{s,q} ( \R^d ) ) }.
\end{equation*}
As a consequence, let $ s \geq 0 $ and $ u \in \X^{s} $ then $ U \in X^{s} $ and there exists a constant $ C> 0 $ such that for all $ u \in \X^{s} $,
\begin{equation*}
\|U\|_{X^{s}}  \leq C \|u\|_{\X^{s}}.
\end{equation*}
\end{prop}

\begin{proof} By interpolation, it is sufficient to prove the result for $ s =n \in \N $. Let $ \alpha \in \N^d$ with $ | \alpha | \leq n $, then thanks to the Leibniz formula, we obtain
\begin{multline*}
 \partial ^\alpha_x U(t,x)  =  \partial ^\alpha_x \Big( \Big( \frac{1}{\sqrt{1+4t^2}}  \Big) ^{d/2}  u \big( \frac{\arctan(2t)}{2}  , \frac{x}{\sqrt{1+4t^2} } \big)    e^{ \frac{i|x|^2t}{1+4t^2}  } \Big)\\
 =\Big( \frac{1}{\sqrt{1+4t^2}}  \Big) ^{d/2}  \sum_{ 0 \leq \beta \leq \alpha} \binom{\alpha}{\beta} \partial ^\beta_x \Big( u \big( \frac{\arctan(2t)}{2} , \frac{x}{\sqrt{1+4t^2} } \big) \Big)  \partial ^{\alpha-\beta}_x \Big( e^{ \frac{i|x|^2t}{1+4t^2}  }\Big).
\end{multline*}
Then, as 
\begin{equation*}
 \big|  \partial ^{\alpha-\beta}_x ( e^{ \frac{i|x|^2t}{1+4t^2}  } ) \big| \leq  C_{\alpha,\beta} \Big( 1+ |  \frac{x}{\sqrt{1+4t^2}  } |^{|\alpha-\beta|}\Big)  ,
\end{equation*}
we establish
\begin{multline*}
 |  \partial ^\alpha_x U(t,x) | \leq \\
    \sum_{ 0 \leq \beta \leq \alpha} C_{\alpha,\beta}  \Big( \frac{1}{\sqrt{1+4t^2}}  \Big) ^{d/2 + |\beta|}  
  | \partial ^\beta_x  u | \big( \frac{\arctan(2t)}{2}  , \frac{x}{\sqrt{1+4t^2} } \big)   \Big( 1+ |  \frac{x}{\sqrt{1+4t^2}  } |^{|\alpha-\beta|}\Big)  .
\end{multline*}

Therefore, 
\begin{multline*}
\| \partial ^\alpha_x  U \|_{ L^p( \R ; L^{q} ( \R^d )  )  }  \leq \\
\begin{aligned}
   \leq& \sum_{ 0 \leq \beta \leq \alpha} C_{\alpha,\beta} \Big\| \Big( \frac{1}{\sqrt{1+4t^2}}  \Big) ^{d/2+d/q + |\beta|}   \big\| u  \big( \frac{\arctan(2t) }{2} , x  \big) \big\|_{  \W^{|\alpha|,q}(\R^d) } \Big\|_{L^p(\R)} \\
 \leq  & \sum_{ 0 \leq \beta \leq \alpha} C_{\alpha,\beta}  \Big\| \Big( \frac{1}{\sqrt{1+ \tan^2 2t }}  \Big) ^{d/2+d/q-2/p + |\beta|}   u (t,x)  \Big\|_{ L^p( [-\frac{\pi}4,\frac{\pi}4 ] ;  \W^{|\alpha|,q}(\R^d) ) } .
\end{aligned}
\end{multline*}
 
To conclude, it is enough to notice that $ \frac{d}{2} + \frac{d}{q}- \frac{2}{p} + |\beta| \geq 0 $ for $ \beta \in \N^d $. 
\end{proof}

%%%%

\subsection{Some properties of Bourgain spaces}
We now define the Bourgain spaces and  recall some of their different properties.

We define the space $ \X^{s,b} = \X^{s,b}( \R \times\R^d ) $ as the closure of $ C_0^\infty ( \R \times\R^d ) $ for the norm
\begin{eqnarray*}
\| u\|^2_{ \X^{s,b} } & =& \sum_{n \in \N } \| \langle  t + \lambda_n ^2 \rangle ^b \lambda_n^s \widehat{P_n u } (t) \|^2_{ L_t^2( \R ; L^2_x(\R^d) ) } \\
& =& \sum_{n \in \N } \| H^{s/2} e^{itH} P_n u(t,.) \|^2_{ L_x^2(\R^d ; H_t^b( \R )) },
\end{eqnarray*}
where $ \widehat{P_n u } (t) $ denotes the Fourier transform, with respect to the time variable, of \linebreak $ P_n u := \langle  u , h_n \rangle _{L^2(\R^d) \times L^2(\R^d)}   h_n $.

In \cite[Corollary 2.10]{tao},  we can find the following proposition: 

  \begin{prop} \label{1-bourgain1}
For any   real $ b > \frac{1}{2} $, there exists a constant $ C > 0 $ such that for any real $ s \in \R $ and for any admissible pair $ (q,r )$, if $ u \in \X^{s,b} $ then $ u \in L^q(\R ; \W^{s,r} (\R^d)) $ and
\begin{equation*}
\|u\|_{L^q(\R ; \W^{s,r} (\R^d))} \leq C \|u\|_{ \X^{s,b} }.
\end{equation*}
\end{prop}

Thanks to \cite[Lemma 2.4]{burq1}, we can establish the following proposition:
 
  \begin{prop} \label{1-bourgain2}
For any $ \theta \in [0,1] $, if $ b > \frac{1-\theta}{2} $ then there exists a constant $ C > 0 $ such that for any $s\in \R$ and any function $ u \in \X^{s,b} $,
\begin{equation*}
\|u\|_{ L^{ \frac{2}{\theta} } (\R; \mathcal{H}^s (\R^d))} \leq C \|u\|_{\X^{s,b}}.
\end{equation*}
\end{prop}

\begin{proof} It is enough to consider the case $s=0$. Using the inverse Fourier transform, we have 
\begin{equation*}
P_n u(t) = \frac{1}{2 \pi} \int_\R   e^{it \tau }  \widehat{P_n u } ( \tau ) d \tau= \frac{1}{2 \pi} \int_\R \frac{\langle  \tau + \lambda_n ^2\rangle ^b}{\langle \tau+ \lambda_n^2\rangle ^b}  e^{it \tau }  \widehat{P_n u } ( \tau ) d \tau. 
\end{equation*}
Then, for $ b > \frac{1}{2} $, we deduce from the  Cauchy-Schwarz inequality that
\begin{equation*}
|P_n u (t)| \leq C \Big( \int_\R \langle \tau+\lambda_n^2\rangle  ^{2b} | \widehat{P_n u }(\tau) |^2 d \tau \Big) ^{1/2}.
\end{equation*}
Thus, by squaring and summing for $ n \in \N $, we obtain that for any  $ b > \frac{1}{2} $ and function $ u \in \X^{0,b} $,
\begin{equation*}
\|u\|_{L^\infty(\R; L^2(\R^d))} \leq C \|u\|_{ \X^{0,b} }.
\end{equation*}
On the other hand,
\begin{equation*}
\|u\|_{ L^2(\R ;L^2(\R^d))} = \|u\|_{ \X^{0,0} }
\end{equation*}
thus the result follows by interpolation. 
\end{proof}

 \begin{prop} \label{1-bourgain3} 
 For any constant $ 0<  \delta <1 $, there exist two constants $ b' < \frac{1}{2} $ and $ C > 0 $ such that for any  $ s \in \R $ and any function $ u \in L^{1+\delta}(\R; \mathcal{H}^s (\R^d)) $,
\begin{equation*}
\|u\|_{\X^{s,-b'}} \leq C \|u\|_{L^{1+\delta}(\R; \mathcal{H}^s (\R^d))}.
\end{equation*}
\end{prop}
\begin{proof} From Proposition~\ref{1-bourgain2}, we have by duality that for any $ \theta \in [0,1] $ and $ b > \frac{1-\theta}{2} $, there exists $ C> 0 $ such that for any function $ u \in L^{ \frac{2}{2-\theta} } ( \R ; L^2( \R^d)) $
\begin{equation*}
\|u\|_{ \X^{0,-b} } \leq C \|u\|_{ L^{ \frac{2}{2-\theta} } ( \R ; L^2( \R^d))}.
\end{equation*}
Then we choose $ \theta = \frac{2 \delta}{1+ \delta} $ and $ b= \frac{1-\theta+\delta}{2} < \frac{1}{2} $ to obtain the result. 
\end{proof}

In \cite[Lemma 2.11]{tao}, we can obtain the following proposition:

 \begin{prop} \label{1-bourgain4}
If $ \psi \in C^\infty_0 ( \R ) $ then for any $ b \geq 0 $, there exists a constant $ C > 0 $ such that for any $ s \in \R $ and any function $ u \in \X^{s,b} $,
\begin{equation*}
\| \psi(t) u \|_{\X^{s,b}} \leq C \|u\|_{\X^{s,b}}.
\end{equation*}
\end{prop}
Finally, we give a last proposition whose proof can be found in \cite[Lemma 3.2]{G} (taking $b'=1-b$ and $T=1$).

 \begin{prop} \label{1-bourgain5}
Let $\psi \in C^\infty_0 ( \R ) $ then for any $ 1 \geq b > \frac{1}{2} $, there exists a constant $ C > 0 $ such that for any $ s \in \R $ and any function $ F \in \X^{s,b-1}$,
\begin{equation*}
 \Big\| \psi(t) \int_0^t e^{-i(t-s) H } F(s) ds \Big\|_{\X^{s,b}} \leq C \|F\|_{\X^{s,b-1}}.
\end{equation*}
\end{prop}
From now on, we set $ T = \frac{\pi}{4} $, then we define the new Bourgain space which will interest us: 
we define the space $ \X^{s,b}_T = \X^{s,b}([ - T ; T ] \times\R^d ) $ as the subset of $ \X^{s,b} $ for which the following norm
\begin{equation*} 
\| u \| _{  \X^{s,b}_T  }  = \inf_{  w \in \X^{s,b} }  \left\{  \| w \| _{\X^{s,b} }  \quad  \text{with} \quad   w |_{[-T,T]} = u \right\}
\end{equation*}
is finite. Proposition~\ref{1-bourgain2} allows us to obtain the following result:

 \begin{prop} \label{1-bourgain7}
Let $ b > \frac{1}{2} $ and $ s \in \R $, then $ \X^{s,b}_T \hookrightarrow C^0( [-T,T]; \mathcal{H}^s(\R^d) ) $.
\end{prop}

And finally, thanks to Proposition~\ref{tvb1} and Proposition~\ref{1-bourgain1}, we obtain the following result:

  \begin{prop} \label{tvb} Let $ b > \frac{1}{2}$, $ s \geq 0 $ and $ u \in \X^{s,b}_{\pi/4} $ then $ U \in X^{s} $ and there exists a constant $ c> 0 $ such that for any $ u \in \X^{s,b}_{\pi/4} $,
\begin{equation*}
\|U\|_{X^{s}}  \leq c \|u\|_{\X^{s,b}_{\pi/4} }.
\end{equation*}
\end{prop}

%%%%%%

\subsection{Basic properties of the eigenfunctions of the harmonic oscillator in dimension $d=3$}
In this fourth, we give some estimates of the tensor eigenfunctions of the harmonic oscillator.

  \begin{prop} \label{1-dispersive}
For any $ \delta > 0 $, there exists a constant $ C_\delta > 0 $ such that for all $ n,m,k \in \N$,
\begin{eqnarray}
  \| h_n \|_{L^4(\R^3)} & \leq & C \lambda_n^{-1/4} (\log \lambda_n ) ^3, \label{1-propre1} \\
   \| h_n \|_{L^\infty(\R^3)}  &\leq &C \lambda_n^{-1/6} , \label{1-propre2}
 % \| h_n h_m\|_{L^2(\R^3)} & \leq & C_\delta  \max ( \lambda_n, \lambda_m)^{-1/2+\delta}, \label{1-propre3} \\
 % \| h_{n} h_{m} h_{k} \|_{L^2(\R^3)}&  \leq & C_\delta  \max ( \lambda_n, \lambda_m, \lambda_k )^{-1/2+\delta} .
\end{eqnarray}
\end{prop}

\begin{proof}Denote by $\mu_n=\sqrt{2n+1}$, so that $(-\partial^2_x+x^2) e_n= \mu^2_n e_n$. In dimension 1, the following   estimates are known by the work \cite{koch}, namely
$$\| e_n \|_{L^4(\R)}  \leq C \mu_n^{-1/4} \log \mu_n  , \quad \| e_n \|_{L^\infty(\R)}  \leq C \mu_n^{-1/6}. $$
Let us prove \eqref{1-propre1}.  In dimension 3, as $ \lambda_n ^2 = \mu_{n_1}^2 + \mu_{n_2}^2 + \mu_{n_3}^2 $ then there exists $ j \in \big\{ 1,2,3 \big\} $ such that $ \mu_{n_j } ^2 \geq \frac{\lambda_n^2 }{3} $. Thus, we obtain 
\begin{eqnarray*}
\| h_n \|_{L^4(\R^3)}  & =& \| e_{n_1} \|_{L^4(\R)} \| e_{n_2} \|_{L^4(\R)} \| e_{n_3} \|_{L^4(\R)} \\
 & \leq & C \mu_{n_1}^{-1/4} \mu_{n_2}^{-1/4}  \mu_{n_3}^{-1/4} (\log \mu_{n_1}) (\log \mu_{n_2})(\log \mu_{n_3}) \\
 & \leq & C  \lambda_n^{-1/4} ( \log \lambda_n )^3,
\end{eqnarray*}
because $ \mu_{n_1} \leq \lambda_n$, $\mu_{n_2} \leq \lambda_n $ and $ \mu_{n_3} \leq \lambda_n $. The  estimate~\eqref{1-propre2} is obtained in a similar manner. 
 \end{proof}

 \begin{lem} \label{1-inegalite sobolev}
 For any $ s \in \R $, there exists a constant $ C> 0 $ such that for all $ f,g,h \in \mathcal{H}^s(\R)$,
\begin{multline*}
\| f(x_1) g(x_2) h(x_3)\|_{\mathcal{H}^s(\R^3)} \leq \\ 
\leq C   (  \| f\|_{\mathcal{H}^s (\R)}\| g\|_{L^2(\R)}  \| h\|_{L^2(\R)} + \| f \|_{L^2(\R)} \| g\|_{\mathcal{H}^s (\R)}  \| h\|_{L^2(\R)}  +\\
\hspace{5cm}+ \| f \|_{L^2(\R)} \| g\|_{L^2(\R)} \| h\|_{\mathcal{H}^s (\R)}  ).
\end{multline*}
\end{lem}

\begin{proof}
We write 
\begin{equation*}
f(x_1)= \sum_{n=0}^{+\infty} a_n e_n(x_1), \quad g(x_2)= \sum_{m=0}^{+\infty} b_m e_m(x_2), \quad h(x_3)= \sum_{k=0}^{+\infty} c_k e_k(x_3).
\end{equation*}
Then, since the family $ (e_n(x_1)e_m(x_2)e_k(x_3))_{(n,m,k)\in \N^3} $ forms a Hilbertian basis of $ L ^2( \R ^3 ) $ composed of eigenfunctions for $ H $ associated to the eigenvalue $ \mu_n^2 + \mu_m^2+ \mu_k^2 $, we obtain
\begin{multline*}
\| f(x_1) g(x_2) h(x_3)\|^2_{\mathcal{H}^s(\R^3)} = \sum_{n,m,k} |a_n|^2|b_m|^2|c_k|^2 ( \mu_n^2 + \mu_m^2+\mu_k^2 )^s  \leq \\
\begin{aligned}
& \leq  C \sum_{n,m,k \geq 0} |a_n|^2|b_m|^2|c_k|^2 \mu_n^{2s} +  C \sum_{n,m,k\geq 0} |a_n|^2|b_m|^2|c_k|^2 \mu_m^{2s} + \\
& \hspace{7cm}+ C \sum_{n,m,k\geq 0} |a_n|^2|b_m|^2|c_k|^2 \mu_k^{2s}   \\
& \leq C   ( \| f\|^2_{\mathcal{H}^s (\R)}\| g\|^2_{L^2(\R)}  \| h\|^2_{L^2(\R)} + \| f \|^2_{L^2(\R)} \| g\|^2_{\mathcal{H}^s (\R)}  \| h\|^2_{L^2(\R)} +\\
&\hspace{7cm}+C  \| f \|^2_{L^2(\R)} \| g\|^2_{L^2(\R)} \| h\|^2_{\mathcal{H}^s (\R)}  ),
\end{aligned}
\end{multline*}
which was the claim.
 \end{proof}
 
\begin{prop} \label{1-proper6} For all $ \delta> 0 $ and all $ s \in [0, 1] $, there exists a constant $ C> 0 $ such that for all $ n, m , k \in \N$,
\begin {eqnarray}
  \| h_n h_m \| _ {\mathcal{H} ^ s (\R ^ 3)} &\leq &C \max (\lambda_n, \lambda_m)^ {s-1/2 + \delta}, \label{1-propre3} \\
   \| h_n h_m h_k \| _ {\mathcal{H} ^ s (\R ^ 3)} & \leq & C \max (\lambda_n, \lambda_m, \lambda_k) ^ {s-1/2 + \delta}. \label{1-propre4}
\end {eqnarray}
\end{prop}

\begin{proof} 
To begin with, recall the  following bilinear estimate which is proven in~\cite[Lemma A.8]{BTT}: for all $0 \leq \theta \leq 1$, we have
\begin{equation}\label{prod}
 \| e_n e_m\|_{\mathcal{H}^\theta(\R)}  \leq C  \max ( \mu_n, \mu_m)^{-1/2+\theta} \min\big( \log \mu_n, \log\mu_m   \big)^{1/2},
\end{equation}
with the notation $\mu_n=\sqrt{2n+1}$. We first prove \eqref{1-propre3} in the case $s=0$. We can suppose that $ \max( \lambda_n,\lambda_m) = \lambda_n $ and $ \max( \mu_{n_1},\mu_{n_2}, \mu_{n_3}) = \mu_{n_1} $, then 
$\mu^2_{n_1} \geq \frac{\lambda_n^2}{3}  \geq \frac{\lambda_m^2}{3} \geq \frac{\mu_{m_1}^2}{3}$. Thus, thanks to~\eqref{prod}, we obtain that for all $\delta >0$
\begin{eqnarray*}
\| h_n h_m \|_{L^2(\R^3)} & = &\| e_{n_1} e_{m_1} \|_{L^2(\R)} \| e_{n_2} e_{m_2} \|_{L^2(\R)} \| e_{n_3} e_{m_3} \|_{L^2(\R)}  \\
 & \leq & \| e_{n_1} e_{m_1} \|_{L^2(\R)} \| e_{n_2}\|_{L^4(\R)} \| e_{m_2} \|_{L^4(\R)} \| e_{n_3} \|_{L^4(\R)} \|  e_{m_3} \|_{L^4(\R)} \\
 & \leq & C_\delta \mu_{n_1}^{-1/2+\delta} \\
 & \leq & C_\delta \lambda_n^{-1/2+\delta}.
\end{eqnarray*}
We now consider the case $s=1$.   Then the general case $0\leq s\leq 1$ will  follow by interpolation. Using Lemma~\ref{1-inegalite sobolev}, \eqref {1-propre1} and \eqref{prod} with $ \theta = 1 $, we obtain
\begin{multline*}
\| h_n h_m \| _ {\mathcal{H} ^ 1 (\R ^ 3)} \leq \\
\begin{aligned}
&\leq C (\| e_ {n_1} e_ {m_1} \| _ {\mathcal{H} ^ 1 (\R)} + \| e_ {n_2} e_ {m_2} \| _ {\mathcal{H} ^ 1 (\R)} + \| e_ {n_3} e_ {m_3} \| _ {\mathcal{H} ^ 1 (\R)} ) \\
 & \leq  C (\max (\mu_ {n_1}, \mu_ {m_1}) ^ {1/2 + \delta} + \max (\mu_ {n_2}, \mu_ {m_2}) ^ {1 / 2 + \delta} + \max (\mu_ {n_3}, \mu_ {m_3}) ^ {1/2 + \delta}) \\
 & \leq  C \max (\lambda_n, \lambda_m)^ {1/2 + \delta},
\end {aligned}
\end{multline*}
which was the claim. 

We turn to the proof of~\eqref{1-propre4} with $s=0$.  Suppose that $ \max ( \lambda_n, \lambda_m, \lambda_k ) = \lambda_n $. Then, by~\eqref{1-propre2} and~\eqref{1-propre3}, we obtain
\begin{eqnarray*}
\| h_{n} h_{m} h_{k}  \|_{L^2(\R^3)} & \leq &\| h_{n} h_{m} \|_{L^2(\R^3)} \| h_{k} \|_{L^\infty(\R^3)} \\
 &\leq &C_\delta \lambda_n ^{-1/2+\delta} \\
 &\leq &C_\delta \max ( \lambda_n, \lambda_m, \lambda_k )^{-1/2+\delta},
\end{eqnarray*}
hence the result. Next, assume that $s=1$. Thanks to the previous inequality,
\begin {eqnarray*}
\| h_n h_m h_k \| _ {\mathcal{H} ^ 1 (\R ^ 3)} & \leq & \| h_n h_m \| _ {\mathcal{H} ^ 1 (\R ^ 3)} \| h_k \| _ {L ^ \infty (\R ^ 3)} + \| h_n h_m \| _ {L ^ 2 (\R ^ 3)} \| h_k \| _ {W ^ {1, \infty} (\R ^ 3)}\\
 & \leq & C \max (\lambda_n, \lambda_m)^ {1/2 + \delta} \lambda_k ^ {- 1/6} + C \max (\lambda_n, \lambda_m)^ {- 1/2 + \delta} \lambda_k ^ {5/6} \\
 & \leq & C \max (\lambda_n, \lambda_m, \lambda_k) ^ {1/2 + \delta}.
\end {eqnarray*}
The general case $0\leq s\leq 1$ follows by interpolation.
\end{proof}

 \begin{lem} \label{1-fast} 
Let $ \delta > 0 $, $ \ell \geq 4 $ and $ N \geq 1 $, then there exists a constant $ C_N > 0 $ such that if we assume $ \lambda_{n_1} \geq \lambda_{n_2}^{1+\delta}$ and $ \lambda_{n_2} \geq \lambda_{n_3} \geq \cdots \geq \lambda_{n_{\ell}}$, then 
\begin{equation*}
  \bigg| \int _{\R^3}  \prod_{j=1}^{\ell} h_{n_j}(x) dx \bigg| \leq C_N \lambda_{n_1}^{-N}.
\end{equation*}
\end{lem}

\begin{proof} Using~\eqref{1-propre1} and \eqref{1-propre2}, we obtain for all $k\geq 1$
\begin{eqnarray*}
\bigg| \int _{\R^3}  \prod_{j=1}^{\ell}  h_{n_j}(x) dx  \bigg| & \leq & \lambda_{n_1}^{-2k} \| h_{n_1} H^k (  \prod_{j=2}^{\ell} h_{n_j})  \|_{L^1(\R^3)} \\
 & \leq &  \lambda_{n_1}^{-2k } \| h_{n_1}  \|_{L^2(\R^3)}     \|  \prod_{j=2}^{\ell} h_{n_j} \|_{H^{2k}(\R^3)} \\
 & \leq & C_k  \lambda_{n_1}^{-2k } \lambda_{n_2}^{2k}  \prod_{j=2}^{\ell} \| h_{n_j} \|_{L^{2(\ell-1)}(\R^3)}  \\
 & \leq &C_k  \left( \frac{\lambda_{n_2}}{\lambda_{n_1}} \right) ^{2k}  \\
 & \leq & C_k  \lambda_{n_1}^{ -\frac{2k\delta}{1+\delta}   },
\end{eqnarray*}
which was the claim.
\end{proof}

  %%%%%%%%%%%%%%%%%%%%%%%%%%%%%%%%%%%%%%%%%%%%%%%%%%%%%%%%%%%%%%%%%%%%%%%%%%%%%%

\section{The bilinear estimate for the harmonic oscillator}\label{Sect3}
The aim of this section is to prove the bilinear estimate of Theorem~\ref{1-bilis}. We notice that   the result will follow from:

 \begin{thm} \label{1-bili}
Let $d\geq 2$. For any $ \delta \in ]0 , \frac{1}{2} ] $, there exists a constant $ C > 0 $ and  $ \epsilon > 0 $ such that for all dyadic numbers $ N,M \geq 1$ and $u, v  \in \mathcal{S}'(\R^d)$,
 \begin{multline*}
  \| e^{it H} \Delta_N ( v ) \, e^{it  H} \Delta_M (u) \|_{L^2 ( [-\epsilon , \epsilon ] ; L^2 ( \R^d ))  }  \leq \\
  \leq   C   \min ( N,M)^{ \frac{d-2}{2} }   \left( \frac{ \min ( N,M)  }{ \max ( N,M )} \right) ^{1/2-\delta }  \| \Delta_N(v) \|_{L^2(\R^d)} \| \Delta_M(u) \|_{L^2(\R^d)}.
\end{multline*}
\end{thm}
Indeed, we can replace $ u $ by $ e^{i\epsilon H}u $ and $ v $ by $ e^{i\epsilon H} v $ to obtain 
\begin{multline*}
 \| e^{i(t+\epsilon) H} \Delta_N ( v ) \, e^{i(t+\epsilon)  H} \Delta_M (u) \|_{L^2 ( [-\epsilon , \epsilon ] ; L^2 ( \R^d ))  }  \leq \\
 \leq   C  \min ( N,M)^{ \frac{d-2}{2} }   \left( \frac{ \min ( N,M)  }{ \max ( N,M )} \right)^{1/2-\delta }  \| \Delta_N(v) \|_{L^2(\R^d)} \| \Delta_M(u) \|_{L^2(\R^d)}.
\end{multline*}
Then, we use the change of variables $ t \longleftrightarrow t + \epsilon $ and Theorem \ref{1-bili} to obtain that
\begin{multline*}
 \| e^{it H} \Delta_N ( v ) \,  e^{i t  H} \Delta_M (u) \|_{L^2 ( [-\epsilon , 2\epsilon ] ; L^2 ( \R^d ))  } \leq  \\ 
 \leq   C  \min ( N,M)^{ \frac{d-2}{2} }  \left( \frac{ \min ( N,M)  }{ \max ( N,M )} \right) ^{1/2-\delta }  \| \Delta_N(v) \|_{L^2(\R^d)} \| \Delta_M(u) \|_{L^2(\R^d)}.
\end{multline*}
We can thus iterate the procedure $ 2 E( \frac{1}{\epsilon} ) $ times to establish Theorem \ref{1-bilis} and we are thus  reduced  to show Theorem \ref{1-bili}. \medskip

Recall that $ \eta \in C_0^\infty( \R ) $ is such that $ \eta_{[0,1]} = 1 $ and $ \eta_{[2, + \infty [ } = 0 $, and was used in the definition \eqref{dya} of $\Delta_N$. Let $ r \ll 1 $ and $ \phi \in C_0^\infty(\R) $ satisfy
\begin{equation*}
\phi(x)  = \left\{
    \begin{aligned}
    &    1  \;\;  \mbox{ for } \; x \in [1/4;2], \\
   &      0 \;\; \mbox{ for } \;  x \in [0,1/4-r] \cup [2+r , + \infty [,
    \end{aligned}
\right.
 \end{equation*}
and set $ \Delta'_N = \phi( \frac{H}{N^2} )$.  Then using that $ \phi (x) ( \eta(x) - \eta(4x) ) = \eta(x) - \eta(4x) $ for all $ x \in \R $, we notice that for all $N \geq 1$
\begin{equation}\label{1-ouf}
\Delta'_N \circ  \Delta_N = \Delta_N.
\end{equation}
Let us then observe that  to prove Theorem \ref{1-bili}, it is sufficient to show
\begin{multline}\label{bouf}
   \| e^{it H} \Delta'_N(v) \, e^{it  H} \Delta'_M(u) \|_{L^2 ( [-\epsilon , \epsilon ] ; L^2 ( \R^d ))  }\leq  \\
 \leq    C  \min ( N,M)^{  \frac{d-2}{2}  }  \left( \frac{  \min ( N,M)  }{ \max ( N,M )} \right) ^{1/2-\delta }   \| v \|_{L^2(\R^d)} \|u \|_{L^2(\R^d)}.
\end{multline}
Indeed, if the previous inequality is satisfied, we can apply it to $ v $ replaced by $ \Delta_N (v) $ and $ u $ replaced by $ \Delta_M (u) $ then we can use \eqref{1-ouf} to obtain Theorem \ref{1-bili}. \medskip

 \textbf{$\bullet$ Case $ M \sim N $ with $ M \geq N $.} For $ d=2$, we can use the Strichartz inequalities (Proposition~\ref{1-Strichartz}) to obtain that
\begin{multline*}
 \| e^{it H} \Delta'_N(v) \, e^{it  H} \Delta'_M(u) \|_{L^2 ( [-\epsilon , \epsilon ] ; L^2 ( \R^d ))  }    \leq \\
\begin{aligned}
& \leq  \| e^{it H} \Delta'_N(v)\|_{L^4{( [-\epsilon , \epsilon ] ; L^4(\R^d))}}  \| e^{it H} \Delta'_N(v)\|_{L^4{( [-\epsilon , \epsilon ] ; L^4(\R^d))} }\\
   &\leq   \| e^{it H} \Delta'_N(v)\|_{L^4{( [-\pi , \pi ] ; L^4(\R^d))} }  \| e^{it H} \Delta'_N(v)\|_{L^4{( [-\pi , \pi ] ; L^4(\R^d))} }\\
 &    \leq     C \| \Delta'_N(v) \|_{L^2(\R^d)} \| \Delta'_N(v)\|_{L^2(\R^d )}\\
& \leq   C  \| v \|_{L^2(\R^d)}   \|u \|_{L^2(\R^d)}.
\end{aligned}
\end{multline*}
For $ d \geq 3 $, using again  Proposition~\ref{1-Strichartz} and the Sobolev  embeddings, we establish that
\begin{multline*}
 \| e^{it H} \Delta'_N(v) \, e^{it  H} \Delta'_M(u) \|_{L^2 ( [-\epsilon , \epsilon ] ; L^2 ( \R^d ))  }  \leq \\
\begin{aligned}
& \leq  \| e^{it H} \Delta'_N(v)\|_{L^\infty{( [-\epsilon , \epsilon ] ; L^d(\R^d))}}  \| e^{it H} \Delta'_N(v)\|_{L^2{( [-\epsilon , \epsilon ] ; L^{  \frac{2d}{d-2} }(\R^d))} }
\\ 
& \leq \| e^{it H} \Delta'_N(v)\|_{L^\infty{( [-\pi , \pi ] ; \W^{ \frac{d-2}{2}   ,2}(\R^d))} } \| e^{it H} \Delta'_N(v)\|_{L^2{( [-\pi , \pi ] ; L^\frac{2d}{d-2}  (\R^d))} }
\\  & \leq  C \| \Delta'_N(v) \|_{\mathcal{H}^{ \frac{d-2}{2}  }(\R^d)} \| \Delta'_N(v)\|_{L^2(\R^d )}
\\ &\leq  C  N^{ \frac{d-2}{2}  } \| v \|_{L^2(\R^d)}   \|u \|_{L^2(\R^d)}.
\end{aligned}
\end{multline*}
Therefore \eqref{bouf} is proved in the case $M \sim N$. 
\medskip

 \textbf{$\bullet$ Case $ M \geq  10N $.} We now have to prove that 
\begin{equation} \label{1-biblitris}
\| e^{it H} v_N  \, e^{it  H} u_M \|_{L^2 ( [-\epsilon , \epsilon ] ; L^2 ( \R^d ))  } \leq C N^{ \frac{d-2}{2} } \left( \frac{  N  }{ M }  \right)^{1/2-\delta } \| v \|_{L^2(\R^d)} \|u \|_{L^2(\R^d)}.
\end{equation}  
We can write 
\begin{equation*}
u_M = \chi \Big( \frac{4|x|^2}{M^2} \Big) u_M  + (1-\chi ) \Big( \frac{4|x|^2}{M^2} \Big) u_M
\end{equation*}
with $ \chi \in C^\infty_0 (\R) $, $0 \leq \chi \leq 1$,  satisfying
\begin{equation*}
\chi(x) = \left\{
    \begin{aligned}
   &     1 \ & \mbox{if } \ &|x| \leq {15}/{32},  \\
   &     0 \ & \mbox{if } \ &|x| \geq   {1}/{2} .
    \end{aligned}
\right.
\end{equation*}
By the triangle inequality, we have to  estimate the following two terms:
\begin{equation} \label{1-term1}
 \big\| \Big(e^{it H} \chi \Big( \frac{4|x|^2}{M^2} \Big) u_M  \Big)   \big( e^{it  H} v_N\big) \big\|_{L^2( [-\epsilon, \epsilon] ;    L^2(\R^d))  } 
\end{equation}
and
\begin{equation} \label{1-term2}
 \big\| \Big(e^{it H} (1-\chi) \Big( \frac{4|x|^2}{M^2} \Big) u_M \Big) \big(  e^{it H } v_N \big) \big\|_{ L^2( [-\epsilon, \epsilon] ;    L^2(\R^d))  }. 
\end{equation}

\subsection{Estimate of the term (3.4)}

We have the following result.

  \begin{prop} \label{1-steffibis}
For all $ \delta > 0 $, there exists a constant $ C_\delta >0$ such that for all $ u \in H^{ -1/2 +  \delta }(\R^d) $ and $ v \in \mathcal{H}^{ \frac{d-1}{2} - \delta }(\R^d)$, 
\begin{equation*}
 \| e^{it H } u \, e^{it  H } v \|_{L^2 ( [ - \frac{\pi}{8}, \frac{\pi}{8} ] ;  L^2 ( \R^d ))  } \leq C_{\delta} \| u \|_{H^{ -1/2 +  \delta }(\R^d)} \|v \|_{\mathcal{H}^{ \frac{d-1}{2} - \delta }(\R^d)} .
 \end{equation*}
\end{prop}
\begin{proof} From \cite[Theorem 2.4]{Staffi} (coming from \cite{bilinear}) we have that for all $ \delta > 0 $, there exists a constant $ C_\delta >0$ such that for all $ {u \in H^{ -1/2 +  \delta }(\R^d)} $ and $ v \in H^{ \frac{d-1}{2} - \delta }(\R^d)$, 
\begin{equation*}
 \| e^{it \Delta} u  \, e^{it  \Delta} v \|_{L^2 ( \R ; L^2 ( \R^d ))  } \leq C_\delta \| u \|_{H^{ -1/2 +  \delta }(\R^d)} \|v \|_{H^{ \frac{d-1}{2} - \delta }(\R^d)} .
\end{equation*}
Using the lens transformation (see Section~\ref{1-214}) and the previous inequality, we obtain that
\begin{multline*}
\| e^{it H} u \, e^{it  H} v \|^2_{L^2 ( [-\frac{\pi}{8} , \frac{\pi}{8} ] ; L^2 ( \R^d ))  } = \\
\begin{aligned}
&= \| e^{-it H} u \, e^{-it  H} v \|^2_{L^2 ( [-\frac{\pi}{8} , \frac{\pi}{8} ] ; L^2 ( \R^d ))  } \\
& =  \int_ { ]-\frac{\pi}{8} , \frac{\pi}{8} [ }\int _{ \R ^d } \frac{1}{ | \cos(2t) |^{2d}  }  | e^{it\Delta} u \, e^{it\Delta}v |^2 \left( \frac{ \tan(2t)}{2} , \frac{x}{ \cos(2t) } \right) dx dt   \\  
&= \int_ { ]-\frac{\pi}{8} , \frac{\pi}{8} [ }\int _{ \R ^d } \frac{1}{ | \cos(2t) |^d  }   | e^{it\Delta} u \, e^{it\Delta}v |^2 \left( \frac{ \tan(2t)}{2} , x \right) dx dt \\
& = \int_ {-1/2}^{1/2}  \int _{ \R ^d } (1+(2t)^2)^{d/2-1}  | e^{it\Delta} u \,e^{it\Delta}v |^2 ( t , x ) dx dt  
\\ & \leq   C   \int_\R  \int _{ \R ^d }  | e^{it\Delta} u \, e^{it\Delta}v |^2 ( t , x ) dx dt  
 \\  &\leq   C   \| u \|_{H^{-1/2 +  \delta }(\R^d)}^2  \|v \| ^2 _{H^{ \frac{d-1}{2} - \delta }(\R^d)} 
 \\ &\leq  C   \| u \|_{H^{-1/2 +  \delta }(\R^d)}^2  \|v \| ^2 _{\mathcal{H}^{ \frac{d-1}{2}  - \delta }(\R^d)},
 \end{aligned}
 \end{multline*} 
 where in the last inequality we used  \eqref{1-comparaison}. 
 \end{proof}

We will show that for all $ \delta \in ]0 , \frac{1}{2} ] $, there exists a constant $ C_\delta > 0 $ such that for all dyadic $ N,M \geq 1$, and  $u, v \in L^2(\R^d)$, if $ M \geq N $ then
\begin{multline*}
\big\| \big(e^{itH } v_N\big) \big( e^{itH } \chi \big( \frac{4|x|^2}{M^2} \big) u_M \big) \big\|_{L^2(  [  -\frac{\pi}{8} , \frac{\pi}{8} ];  L^2(\R^d)) }  \leq \\
\leq C_\delta N^{  \frac{d-2}{2} } \left( \frac{N}{M} \right) ^{1/2-\delta} \|u\|_{L^2(\R^d)} \|v\|_{L^2(\R^d)}.
\end{multline*}  
To do this, using Proposition~\ref{1-steffibis}, it suffices to prove that for any $ \delta \in ]0 , \frac{1}{2} ] $, there exists a constant $ C_\delta > 0 $ such that for all dyadic $ M \geq 1$ and $ u \in L^2(\R^d)$
\begin{equation*} 
\| \chi  \Big( \frac{4|x|^2}{M^2} \Big) u_M \|_{H^{-1/2+\delta}(\R^d) } \leq C_\delta  M^{-1/2+\delta} \| u  \|_{L^2(\R^d)}.
\end{equation*}
We have
\begin{equation*}
\| \chi \Big( \frac{4|x|^2}{M^2} \Big) u_M \|_{L^2(\R^d)} \leq \| u_M \|_{L^2(\R^d)} \leq C \|u\|_{L^2(\R^d)}.
\end{equation*}
 Thus, by interpolation, it is sufficient to show that there exists a constant $ C > 0 $ such that for all dyadic $M\geq 1$ and $ u \in L^2(\R^d)$,
\begin{equation}\label{1-class}
\| \chi \Big( \frac{4|x|^2}{M^2} \Big) u_M \|_{H^{-1}(\R^d)} \leq C M^{-1} \|u\|_{L^2(\R^d)}.
\end{equation}
We then use the semi-classical calculus. For a function $ u $, we define $  \mathfrak{u} : x \longmapsto u( \frac{x}{     \sqrt{ h }  } ) $ where $ h = \frac{1}{M^2} $.

Observe that 
\begin{equation} \label{1-semi-classique1}
\chi \Big( \frac{4|x|^2}{M^2}  \Big)  \Big[ \frac{H}{M^2} \Big] ( u ) (x) =  [   \chi( 4|x|^2 )  (  -h^2 \Delta + |x|^2  )  ] ( \mathfrak{u} )  (  \sqrt{ h } x  )
\end{equation}
and that 
\begin{equation} \label{1-semi-classique2}
\chi \Big( \frac{4|x|^2}{M^2} \Big)  \Big[   \phi \Big(   \frac{H}{M^2} \Big) \Big] ( u ) (x) =  [ \chi( 4|x|^2 ) \phi   (  -h^2 \Delta + |x|^2  )   ] (  \mathfrak{u}  )( \sqrt{  h } x  ).
\end{equation}
Thus, to prove \eqref{1-class}, it is sufficient to show that there exists  a constant $ C > 0 $ such that for all $ h \in ]0,1] $ and $ u \in L^2(\R^d)$,
\begin{equation} \label{1-classb}
 \| \chi ( 4 |x|^2) \phi ( |x|^2 + |h\nabla|^2 ) u \|_{H^{-1}(\R^d) } \leq C h \|u\|_{L^2(\R^d)}.
\end{equation}
Indeed,
\begin{eqnarray*}
\| \chi \Big( \frac{4|x|^2}{M^2} \Big) u_M \|_{H^{-1}(\R^d)} & \leq&  \big\|  [ \chi( 4|x|^2 ) \phi   (  -h^2 \Delta + |x|^2  )   ] (\mathfrak{u}  )( \sqrt{  h } .  )  \big\|_{H^{-1}(\R^d) }\\
& \leq& h^{- \frac{1}{2} - \frac{d}{4} }  \| [ \chi( 4|x|^2 ) \phi   (  -h^2 \Delta + |x|^2  )   ] (\mathfrak{u}  )  \|_{ H^{-1}(\R^d) }\\
 & \leq &C  h^{\frac{1}{2} - \frac{d}{4} } \|  \mathfrak{u} \| _{L^2(\R^d)}\\
 & \leq& C h^{1/2}  \|u\|_{L^2(\R^d)}\\
  & \leq& C M^{-1} \|u\|_{L^2(\R^d)}.
\end{eqnarray*}

Let us then prove \eqref{1-classb}. Thanks to Proposition~\ref{1-approx}, we have 
\begin{multline*}
  \| \chi( 4 |x|^2  ) \phi \big( |x|^2 + |h\nabla|^2 \big) u \|_{H^{-1}(\R^d)}  \leq \\
   \begin{aligned}
   & \leq   \|   \chi( 4|x|^2 ) \big[ \phi ( |x|^2 + |h\nabla|^2 )    - Op_h ( \phi ( |x|^2 +  |\xi|^2 )  ) \chi_2 \big] u\|_{H^{-1}(\R^d)}  + \\
   & \hspace{6cm}  +\|\chi( 4|x|^2 ) Op_h ( \phi ( |x|^2 +   |\xi|^2 )  ) \chi_2 u \|_{H^{-1}(\R^d)} \\
& \leq  h \| u \|_{L^2(\R^d) }   + \| \chi ( 4|x|^2 ) Op_h ( \phi ( |x|^2 +  |\xi| ^2 ) ) \chi_2 u \|_{H^{-1}(\R^d)}.
 \end{aligned}
 \end{multline*}
Thus, it is sufficient to evaluate $ \| \chi ( 4|x|^2 ) Op_h ( \phi ( |x|^2 +  |\xi| ^2 ) ) \chi_2 u \|_{H^{-1}(\R^d)} $. We have
\begin{multline*}
 \chi ( 4|x|^2 )   Op_h \big( \phi ( |x|^2 +   |\xi| ^2 ) \big) \chi_2 u =\\
\begin{aligned}
&= \frac{\chi ( 4|x|^2 )}{(2 \pi)^d}  \int_{\R^d }  e^{ix \cdot \xi }   \phi ( |x|^2 + ( h \xi ) ^2 )    \mathcal{F} ( \chi_2 u ) ( \xi ) d\xi  \\  
 &= \frac{h}{(2 \pi)^d}  \int_{\R^d }  (h \xi )  \cdot \nabla ( e^{ix \cdot \xi } )  \frac{ \chi ( 4 |x|^2 )  }{i| h  \xi | ^2 }   \phi ( |x|^2 + ( h  \xi )^2 )   \mathcal{F} (\chi_2 u)( \xi ) d\xi  
\\   &= \frac{h}{(2 \pi)^d}  \left(  \int_{\R^d }  ( h \xi ) \cdot \nabla   \Big(  e^{ix\cdot  \xi }   \frac{ \chi ( 4 |x|^2 )  }{i |h \xi | ^2 }    \phi ( |x|^2 + | h  \xi |^2 )   \mathcal{F}(\chi_2 u)( \xi ) \Big) d\xi   \right.\\
& \hspace*{2cm} -  \left. \int_{\R^d }  e^{ix \cdot \xi }  (  h \xi ) \cdot \nabla \Big( \frac{ \chi ( 4 |x|^2 )  }{i| h \xi | ^2 }   \phi ( |x|^2 + | h  \xi |^2 ) \Big)   \mathcal{F} (\chi_2 u)( \xi ) d\xi  \right)
\\   &= \frac{h}{(2 \pi)^d}  \left( \textrm{div}_x   \Big(  \int_{\R^d }  ( h \xi )  e^{ix\cdot  \xi }   \frac{ \chi ( 4 |x|^2 )  }{i|h \xi | ^2 }    \phi ( |x|^2 + | h  \xi |^2 )   \mathcal{F}(\chi_2 u)( \xi ) d\xi   \Big) \right.
\\ & \hspace*{2cm} -  \left. \int_{\R^d }  e^{ix\cdot \xi }   (  h \xi ) \cdot \nabla \Big( \frac{ \chi ( 4 |x|^2 )  }{i|h \xi | ^2 }   \phi ( |x|^2 + | h  \xi |^2 ) \Big)   \mathcal{F} (\chi_2 u)( \xi ) d\xi  \right).  
\end{aligned}
\end{multline*}
Then, since $ 4 |x|^2 \leq \frac{1}{2} $ and $ \frac{1}{4}^- \leq |x|^2 +  |\xi| ^2 \leq 2^+ $ implies $ \frac{1}{8}^- \leq |\xi|^2 $, we deduce that
\begin{equation*}
( x , \xi ) \longrightarrow \xi \frac{ \chi ( 4 |x|^2 )  }{|\xi|^2}   \phi ( |x|^2 +   |\xi|^2 )  \in S^0  
\end{equation*}
and
\begin{equation*}
 ( x , \xi ) \longrightarrow  \xi \cdot \nabla \Big( \frac{ \chi ( 4 |x|^2 )  }{ |\xi| ^2 }   \phi ( |x|^2 +  |\xi|^2 ) \Big)   \in S^0.
\end{equation*}
Thus, by Theorem \ref{1-operator}, we have
\begin{multline*}
  \Big\|  \frac{h}{(2 \pi)^d}  \textrm{div}_x   \Big(  \int_{\R^d }  ( h \xi )  e^{ix\cdot \xi }   \frac{ \chi ( 4 |x|^2 )  }{i(h \xi ) ^2 }    \phi ( |x|^2 + ( h  \xi )^2 )   \mathcal{F}(\chi_2 u)( \xi ) d\xi   \Big)   \Big\|_{H^{-1}(\R^d)} \leq \\
\begin{aligned}
 & \leq  \frac{h}{(2 \pi)^d}    \Big\|  \int_{\R^d }  ( h \xi )  e^{ix\cdot \xi }   \frac{ \chi ( 4 |x|^2 )  }{i(h \xi ) ^2 }    \phi ( |x|^2 + ( h  \xi )^2 )   \mathcal{F}(\chi_2 u)( \xi )   d\xi \Big\|_{L^2(\R^d)} \\
 & \leq  C h  \| \chi_2  u  \|_{L^2(\R^d)}  \\
 & \leq  C h  \|  u  \|_{L^2(\R^d)}
\end{aligned}
\end{multline*}
and
\begin{multline*}
  \Big\| \frac{h}{(2 \pi)^d}  \int_{\R^d }  e^{ix \cdot \xi }   (  h \xi ) \cdot \nabla \left( \frac{ \chi ( 4 |x|^2 )  }{i( h \xi ) ^2 }   \phi ( |x|^2 + ( h  \xi )^2 ) \right)   \mathcal{F} (\chi_2 u)( \xi ) d\xi   \Big\|_{  H^{-1}(\R^d)  } \leq  \\
\begin{aligned}
&    \leq    \Big\|  \frac{h}{(2 \pi)^d}  \int_{\R^d }  e^{ix\cdot  \xi }   (  h \xi ) \cdot \nabla \left( \frac{ \chi ( 4 |x|^2 )  }{i( h \xi ) ^2 }   \phi ( |x|^2 + ( h  \xi )^2 ) \right)   \mathcal{F} (\chi_2 u)( \xi ) d\xi    \Big\|_{  L^2 (\R^d)  }\\
&  \leq  C h \| \chi_2  u  \|_{L^2(\R^d)} \\
& \leq  C h \|  u  \|_{L^2(\R^d)}.
\end{aligned}
\end{multline*}
This proves \eqref{1-classb} and the estimate \eqref{1-term1} follows.

\subsection{Estimate of the  term (3.5)}

 \begin{prop} \label{1-clef}
  There exists a time $ T \in ] 0 , \frac{\pi}{4} [$ such that for all $ K\geq 0 $, there exists a constant $ C_K > 0 $ such that for all $ M \geq 1 $ and $u \in L^2(\R^d) $,
\begin{equation*}
\Big\| \chi   \Big( \frac{8 |x|^2}{M^2}   \Big) e^{itH} (1-\chi)   \Big( \frac{4 |x|^2}{M^2}   \Big) u_M   \Big\|_{L^2 ( [-T , T ] ; L^2 ( \R^d ) )}  \leq C_K M^{-K} \|u\|_{L^2(\R^d)}.  
 \end{equation*}
\end{prop}
Let us show that Proposition~\ref{1-clef} implies the bound  \eqref{1-term2}. From the Sobolev embedding
\begin{equation*}
\|u\|_{L^\infty( [a,b] ;  L^\infty (\R^d))} \leq C \|u\|_{L^\infty ( [a,b] ; \mathcal{H}^{ \frac{d}{2}+1 }(\R^d)) } 
\end{equation*}
and Proposition~\ref{1-clef}, we deduce for $ M \geq 10 N $ that
\begin{multline*}
\Big\| \chi \big( \frac{8 |x|^2}{M^2} \big) \Big( e^{itH} (1-\chi) \big( \frac{4 |x|^2}{M^2} \big) u_M\Big) \, \big(e^{itH } v_N\big)  \Big\|_{L^2([-T, T] ; L^2(\R^d))}  \leq 
 \\
\begin{aligned}
 &\leq  \| \chi \left( \frac{8 |x|^2}{M^2} \right) e^{itH} (1-\chi) \left( \frac{4 |x|^2}{M^2} \right) u_M\|_{L^2([-T, T] ; L^2(\R^d)) }    \| e^{itH} v_N\|_{L^\infty([-T, T] ; L^\infty(\R^d))}    \\ 
& \leq   C  \| \chi \left( \frac{8 |x|^2}{M^2} \right)  e^{itH} (1-\chi) \left( \frac{4 |x|^2}{M^2} \right) u_M\|_{L^2([-T, T] ; L^2(\R^d))  }    \|  v_N\|_{ \mathcal{H} ^{\frac{d}{2} +1 } ( \R^d )  }   \\
 & \leq   C_K M^{-K} N^{ \frac{d}{2} +1 } \|u\|_{L^2(\R^d)}   \|v\|_{L^2(\R^d)}  \\
& \leq  C_K M^{-K + \frac{d}{2} +1 }   \|u\|_{L^2(\R^d)} \|v\|_{L^2(\R^d)}. 
\end{aligned}
\end{multline*}
And so, it is enough to estimate, for $ M \geq 10 N$, the following term:
\begin{equation*}
\Big\| \Big(  (1-\chi ) \big( \frac{8 |x|^2}{M^2} \big) \Big) \Big( e^{itH} (1-\chi) \big( \frac{4 |x|^2}{M^2} \big) u_M \Big) \big( e^{itH} v_N    \big)\Big\|_{L^2([-T, T] ; L^2(\R^d))  }. 
\end{equation*}  
Thanks to Proposition~\ref{1-Strichartz}, we obtain that  for all $ R \geq 1$,
\begin{multline*}
\Big\| \big(e^{itH} v_N   \big)  (1-\chi ) \left( \frac{8 |x|^2}{M^2} \right)\Big( e^{itH} (1-\chi) \big( \frac{4 |x|^2}{M^2} \big) u_M \Big) \Big\|_{L^2([-T, T] ; L^2(\R^d))  } =\\
\begin{aligned} 
& =    \Big\| \langle x\rangle  ^ R \big( e^{itH} v_N\big)    \langle x\rangle ^{-R}(1-\chi ) \left( \frac{8 |x|^2}{M^2} \right) \Big( e^{itH} (1-\chi) \big( \frac{4 |x|^2}{M^2} \big) u_M \Big) \Big\|_{ L^2([-T, T] ; L^2(\R^d)) } \\ 
&  \leq  M^{-R}   \Big\| \langle x\rangle  ^ R \big(e^{itH} v_N \big) (1-\chi ) \left( \frac{8 |x|^2}{M^2} \right)  \Big( e^{itH} (1-\chi) \big( \frac{4 |x|^2}{M^2} \big) u_M \Big) \Big\|_{ L^2([-T, T] ; L^2(\R^d)) } \\
&  \leq  M^{-R}   \Big\| (1-\chi ) \left( \frac{8 |x|^2}{M^2} \right)   \langle x\rangle  ^ R \big( e^{itH} v_N\big) \Big\|_{ L^4([-T, T] ; L^4(\R^d)) }\\ 
&  \hspace*{5cm}  \cdot \Big\|  e^{itH} (1-\chi) \left( \frac{4 |x|^2}{M^2} \right) u_M \Big\|_{ L^4([-T, T] ; L^4(\R^d)) } \\
& \leq M^{-R}  \Big\| (1-\chi )\left( \frac{8 |x|^2}{M^2} \right)  \langle x\rangle  ^ R \big( e^{itH} v_N\big) \Big\|_{ L^4([-T, T] ; L^4(\R^d)) } \\ 
& \hspace*{5cm}  \cdot \Big\|  e^{itH} (1-\chi) \left( \frac{4 |x|^2}{M^2} \right) u_M \Big\|_{ L^4([-T, T] ; \W ^{  \frac{d-2}{4} , \frac{2d}{d-1} }(\R^d)) } \\
&\leq M^{-R}  \Big\|  \langle x\rangle  ^ R (1-\chi )\left( \frac{8 |x|^2}{M^2} \right) \big( e^{itH} v_N\big)   \Big\|_{L^\infty([-T, T] ; L^4(\R^d))}  M^{ \frac{d-2}{4}} \|  u_M \|_{L^2(\R^d) }.
  \end{aligned} 
  \end{multline*}
Then, since 
$$ Supp \left\{ (1-\chi )\left( \frac{8 |x|^2}{M^2} \right) \right\} \subset \left\{ |x|^2 \geq  \frac{M^2}{8}   \frac{15}{16} \right\} \subset \left\{ |x|^2  \geq 5 N^2 \right\}$$
from Proposition~\ref{1-propre}, we deduce
\begin{equation*}
\Big\| (1-\chi ) \left( \frac{8 |x|^2}{M^2} \right) \langle x\rangle  ^ R e^{itH} v_N \Big\|_{L^4([-T, T] ; L^4(\R^d))} \leq C  \|v\|_{L^2(\R^d)}.
\end{equation*}

\begin{proof}[Proof of Proposition \ref{1-clef}]
 Using  semi-classical analysis as in~\eqref{1-semi-classique2}, it suffices to prove the existence of a time $ T \in  ]0 , \frac{\pi}{4}[ $ such that for all $ N \geq 1 $, there exists a constant $ C_N > 0 $ such that for all $ u \in L^2(\R^d) $ and $ h \in ]0,1] $,
\begin{equation*}
 \big\| \chi( 8 |x|^2 ) e^{itH_h/h} (1-\chi) ( 4 |x|^2 ) \phi ( |x|^2 + |h\nabla|^2  ) u  \big\|_{L^2( [-T , T ] ; L^2 ( \R^d ) )}  \leq C_N h^{N} \|u\|_{L^2(\R^d)}.
 \end{equation*}
Using Proposition~\ref{1-approx}, it is sufficient to establish the following result: for any function $ g(x, \xi ) \in C_0^\infty( \R^d \times\R^d ) $ satisfying  $ Supp ( g )  \subset  \lbrace \frac{1}{8} \leq  |\xi| ^2 + |x|^2 \leq 4 \rbrace $ and all integer  $ N \geq 1 $, there exists a constant $ C_N > 0 $ such that for all  $ u \in L^2(\R^d) $ and $ h \in ]0,1] $,
\begin{equation*}
\big\| \chi( 8 |x|^2 ) e^{itH_h/h} (1-\chi) ( 4 |x|^2 ) Op_h( g ) u  \big\|_{L^2 ( [-T , T] ;L^2 ( \R^d ) )}  \leq C_N h^{N}  \|u\|_{L^2(\R^d)}.
\end{equation*}

Indeed, by  Proposition~\ref{1-approx},
\begin{multline*}
\big\| \chi( 8 |x|^2 ) e^{itH_h/h} (1-\chi) ( 4 |x|^2 ) \phi ( |x|^2 + |h\nabla|^2  ) u  \big\|_{L^2( [-T , T ] ; L^2 ( \R^d ) )} \leq \\
\begin{aligned}
&\leq   \big\| \chi( 8 |x|^2 ) e^{itH_h/h} (1-\chi) ( 4 |x|^2 ) [ \phi ( |x|^2 + |h\nabla|^2  ) u  -  \sum_{j=0}^{N-1} h^j Op_h(\Psi_j)  u ]   \big\|_{L^2( [-T , T ] ; L^2 ( \R^d ) )}  
\\ 
&\hspace{2cm}+ \sum_{j=0}^{N-1} h^j  \big \| \chi( 8 |x|^2 ) e^{itH_h/h} (1-\chi) ( 4 |x|^2 ) Op_h ( \Psi_j ( x , \xi) ) u  \big\|_{L^2([-T,T]; L^2( \R ^ d)) } 
\\ 
& \leq  C_N h^{N} \|u\|_{L^2 ( \R^d ) }  + \sum_{j=0}^{N-1} h^j  \big\| \chi( 8 |x|^2 ) e^{itH_h/h} (1-\chi) ( 4 |x|^2 ) Op_h( \Psi_j  ) u  \big\|_{L^2([-T,T]; L^2 ( \R ^ d)) } 
\\
& \leq C_N h^{N} \|u\|_{L^2(\R^d)}.
\end{aligned}
\end{multline*}
\end{proof}

\begin{lem}
 There exists a time $ T \in  ] 0 , \frac{\pi}{4}[ $ such that if $ g  \in  C_0^\infty( \R^d \times\R^d )$ satisfying $ Supp ( g ) \subset \lbrace \frac{1}{8} \leq  |\xi| ^2 + |x|^2 \leq 4 \rbrace $ then for any integer $ K \geq 1 $, there exists a constant $ C_K > 0 $ such that for all $ u \in L^2(\R^d) $ and $ h \in ]0,1] $,
\begin{equation*}
\big\| \chi( 8 |x|^2 ) e^{it H_h/h } (1-\chi) ( 4 |x|^2 ) Op_h( g  ) u  \big\|_{L^2 ( [-T , T] ; L^2 ( \R^d ) )}  \leq C_K h^{K} \|u\|_{L^2(\R^d)}.
\end{equation*}
\end{lem}
\begin{proof} We define
\begin{equation*}
w(s,x) = \int_{\R^d} e^{ \frac{i}{h} \Phi(s,x,\xi )  }  a ( s,x, \xi , h ) \hat{u} ( \frac{\xi}{h} ) \frac{d \xi }{ ( 2 \pi h )^d }
\end{equation*}  where
\begin{equation*}
a(s,x, \xi,h) = \sum_{j=0} ^K h^j a_j ( s,x,\xi).
\end{equation*}
Assume that
\begin{equation*}
\left\{
    \begin{aligned} \label{1-hamilton}
     &   \Phi ( 0, x, \xi )= x \cdot \xi, \\
   &     \partial_s  \Phi - | \nabla \Phi |^2 - |x|^2 = 0, 
    \end{aligned}
\right. 
\end{equation*}

\begin{equation*}
\left\{
    \begin{aligned} 
    & a_0 ( 0, x, \xi )= (1-\chi)( 4 |x|^2 ) g ( x, \xi  ), \\
    &    \partial_s a_0 - 2 \nabla a_0  \cdot \nabla \Phi -  a_0 \Delta  \Phi  = 0,
   \end{aligned} 
\right. 
\end{equation*}
and for $ 1 \leq j \leq K $  
\begin{equation*}
\left\{
 \begin{aligned} 
 &        a_j ( 0, x, \xi )= 0,  \\
    &    \partial_s a_j - 2    \nabla a_j \cdot  \nabla \Phi     - a_j  \Delta  \Phi =  -i\Delta a_{j-1}.
   \end{aligned} 
\right. 
\end{equation*}
Then 
\begin{eqnarray*}
 i h  \partial _s w +( -h^2 \Delta  + |x|^2) w & =& - h^{K+2} \int_{\R^d} e^{ \frac{i}{h} \Phi(s,x,\xi )  } \Delta ( a_K ( s,x, \xi) ) \hat{u} ( \frac{\xi}{h} ) \frac{d \xi }{ ( 2 \pi h )^d }  
 \\  
 &:=& h^{K+2 } f.
 \end{eqnarray*} 
 Since $w_0=(1-\chi) ( 4 |x|^2 ) Op_h( g ( x, \xi ) ) u$, hence 
 $$w=e^{it H_h/h } (1-\chi) ( 4 |x|^2 ) Op_h( g ( x, \xi  ) ) u- i h^{K+1} \int_0^t e^{i(t -s ) H_h/h } f(s) ds,$$
which in turn implies 
\begin{multline*} 
 \chi( 8 |x|^2 ) e^{it H_h/h } (1-\chi) ( 4 |x|^2 ) Op_h( g ( x, \xi  ) ) u = \\
=\chi(8|x|^2) w ( t, x)  - i h^{K+1} \chi ( 8|x|^2 ) \int_0^t e^{i(t -s ) H_h/h } f(s) ds.  
\end{multline*}
Let us notice that if $ \Psi $ is a solution of the equation $ \partial _t \Psi + | \nabla \Psi | ^2 = 0 $ with initial data $ \Psi ( 0,x,\xi) = x \cdot \xi $ then $ \Phi(t,x,\xi ) = \Psi (- \frac{  \tan 2 t }{2 } , \frac{x}{ \cos 2t } , \xi ) + \frac{|x|^2 \tan 2 t }{ 2 } $ is a solution of the equation $ \partial_t \Phi - | \nabla \Phi |^2 - |x|^2 = 0 $ with the same initial data.

By the method of characteristics, we obtain
\begin{equation*}
\Psi ( t ,x, \xi ) = - t | \xi |^2 +x\cdot\xi,
\end{equation*}
then we deduce that
\begin{equation*}
 \Phi ( t, x , \xi ) = \frac { \tan(2t) }{2} ( |\xi|^2  + |x|^2 ) + \frac{x \cdot\xi}{\cos( 2t) }.
\end{equation*}
Hence
\begin{equation*}
 \nabla \Phi = \frac{ \xi}{ \cos ( 2t) } + x \tan ( 2 t )  \ \mbox{ and }  \  \Delta \Phi =  d \tan( 2t ). 
\end{equation*}
Using the method of characteristics, we obtain 
\begin{equation} \label{1-a0}
 a_0 \big( t , x - 2 \int_0^t \nabla \Phi , \xi \big) = \frac{ a_0 ( 0 ,x,\xi ) }{ | \cos 2 t |^{  \frac{d}{2}  } },   
 \end{equation}
and
\begin{equation} \label{1-aj}
a_j \big( t, x - 2 \int_0^t \nabla \Phi , \xi \big) = -i \int_0^t  \bigg| \frac{ \cos ( 2 s ) }{ \cos ( 2 t )  } \bigg|^{  \frac{d}{2} }   \Delta a_{j-1} \big( s,x - 2 \int_0^s \nabla \Phi , \xi \big) ds.
\end{equation}
Now, for any $ \xi \in \R^d $ and $ | x | \leq \frac{1}{2} $,
\begin{equation*}
a_0 ( 0 ,x,\xi ) =0.
\end{equation*}
As a consequence, for all $ \xi \in \R^d $, $ t \in ] - \frac{\pi}{4} , \frac{\pi}{4}  [ $, $ | x | \leq \frac{  \sqrt{15} }{8 \sqrt{2} } $ and $ j \in \N$
\begin{equation*}
a_j \big( t, x - 2 \int_0^t \nabla \Phi , \xi \big) = 0 .
\end{equation*}
We have  $ \int _0^t \nabla \Phi =  \xi F ( t) - \frac{x \log \cos (2 t ) }{2} $ with $F$ a continuous function satisfying $ F ( 0) = 0 $. Thus, for all  $ \epsilon > 0 $, there exists a  time $ T \in ]0, \frac{\pi}{4}  ] $ such that if $ |t| \leq T $ then  $ |F(t)| \leq \epsilon $ and $ | \log \cos 2 t  | \leq \epsilon $.

Observe that
\begin{equation*}
  y   = x - 2 \int_0^t \nabla \Phi = x ( 1+ \log \cos 2t ) - 2 \xi F(t)
  \end{equation*}
can be solved by
\begin{equation*}
 x = \frac{y+2 \xi F(t)}{1+\log \cos 2t }.
  \end{equation*}
So if $ | y | \leq \frac{1}{ 4 } $, $ |\xi|^2 \leq  4 $ and $ |t| \leq T $ then  $ |x| \leq \frac{1/ 4 + 4 \epsilon}{1-\epsilon}  \leq \frac{  \sqrt{15} }{ 8 \sqrt{2}  } $, if $ \epsilon <1$ is small enough.

This implies that for any $ j \in \N$ and $ |t| \leq T $
\begin{equation*}
 Supp ( a_j( t) )  \subset  B_x \big( 0, \frac{1}{ 4 }  \big) ^c   \times B_\xi \left( 0, 2  \right) .
\end{equation*}
Therefore, if $ |t| \leq T $, since  $ Supp ( \chi(8 |x|^2) ) \subset B_x( 0, \frac{1}{ 4 }) $, we deduce that 
\begin{equation*}
\chi( 8 |x|^2 ) e^{it H_h/h} (1-\chi) ( 4 |x|^2 ) Op_h( g(x,\xi) )  u = - ih^{K+1} \chi ( 8|x|^2 ) \int_0^t e^{i(t-s) H_h/h} f(s)  ds.
\end{equation*}
Then, by Theorem \ref{1-Strichartz}, we obtain
\begin{multline*}
 \| \chi( 8 |x|^2 ) e^{itH_h/h } (1-\chi) ( 4 |x|^2 ) Op_h( g ( x,\xi ) ) u \|_{L^2( [-T,T] ;  L^2(\R^d))}   \leq \\
\begin{aligned}
& \leq  h^{K+1} \Big\|    \chi ( 8|x|^2 ) \int_0^t e^{i(t-s)  H_h/h } f(s) ds   \Big\|_{L^2( [-T,T] ; L^2(\R^d)) } \\
& \leq  h^{K+1}  \|f\|_{L^1( [-T,T  ] ;  L^2(\R^d)) }  \\
& \leq  h^{K+1} \| \Delta a_K \|_{L^1_t( [-T,T ]  ; L^2_x(\R^d ; L ^2_\xi(\R^d))) }   \|u\|_{L^2(\R^d)}.
\end{aligned}
\end{multline*}
The lemma is therefore proved if  $ \Delta a_K \in L^1_t\big( [-T,T];  L^2_x(\R^d  ; L ^2_\xi(\R^d))\big) $. \\

We prove  by induction on  $ K \in \N $, that for all  $ \alpha \in \N^d$,
$$  \partial^\alpha_x  a_K \in L^1_t\big( [-T,T ] ; L^2_x(\R^d  ; L ^2_\xi(\R^d))\big).$$
 For $ K=0 $, using \eqref{1-a0}, we see by change of variables that 
$$ \partial^\alpha_x  a_0\in L^1_t\big( [-T,T ] ; L^2_x(\R^d  ; L ^2_\xi(\R^d))\big) $$
 if $ \partial^\alpha_x a_0(0) \in L^2_x(\R^d, L ^2_\xi(\R^d)) $.  But  $ a_0 ( 0 ) \in C ^\infty (\R^d \times \R^d )$ with $ Supp ( a_0 ( 0 ) ) \subset \big\{ (x ,\xi ) \, :\,  |x|^2 \leq 1 , |\xi|^2 \leq  4  \big\} $ and the case $ K= 0 $ is obvious.
 
 Let us assume the result established at rank $ K -1 $ and show it at rank $ K $. Using~\eqref{1-aj}, we note that $ \partial^\alpha_x  a_K \in L^1_t\big( [-T,T ] ;L^2_x(\R^d  ; L ^2_\xi(\R^d))\big)$ if  $ \partial^{\alpha+2}_x  a_{K-1} \in L^1_t\big( [-T,T ] ; L^2_x(\R^d  ; L ^2_\xi(\R^d))\big)$. This last statement being clear by induction hypothesis. 
\end{proof}
\subsection{Bilinear estimates and Bourgain spaces}
The aim of this section is to write the bilinear estimate of Theorem \ref{1-bilis} in Bourgain spaces. More precisely, we establish the following result: 

 \begin{prop}
Let $ \psi \in C_0^\infty (\R) $. There exists $ \delta_0 \in ]0 , \frac{1}{2} ]$ such that for all  $ \delta \in ]0, \delta_0 ] $, there exist  $ b' < \frac{1}{2} $ and a constant $ C > 0 $ such that for all  dyadic numbers $ N,M \geq 1$ and $u_0, u, v  \in \mathcal{S}'(\R^d)$,
\begin{multline}\label{1-bilibourgain1}
  \|  \Delta_N ( v )   \Delta_M (u) \|_{L^2 ( \R ; L^2 ( \R^d ))  } \leq \\
\leq  C  \min ( N,M)^{ \frac{d-2}{2} +\delta }   \left( \frac{ \min ( N,M)  }{ \max ( N,M )}   \right) ^{1/2-\delta }   \| \Delta_N(v) \|_{  \X^{0,b'}} \| \Delta_M(u) \|_{  \X^{0,b'}}.
\end{multline}
and 
\begin{multline}\label{1-bilibourgain2}
  \|  \Delta_N \big( \psi(t)  e^{-itH}  u_0 \big)  \Delta_M (u) \|_{L^2 ( \R ; L^2 ( \R^d ))  } \leq \\
 \leq  C  \min ( N,M)^{ \frac{d-2}{2} +\delta }   \left( \frac{ \min ( N,M)  }{ \max ( N,M )} \right) ^{1/2-\delta }  \| \Delta_N(u_0) \|_{ L^2(\R^d)} \| \Delta_M(u) \|_{  \X^{0,b'}}.
\end{multline}
\end{prop}

To prove these results, we adapt the proof of \cite[Lemma 4.4]{burq3}. Let us begin by noting that it is sufficient to prove the following two propositions:

 \begin{prop} \label{1-bilibourgain3}
For all $ b \in ] \frac{1}{2} , 1 ] $ and $ \delta \in ]0, \frac{1}{2} ] $, there exists a constant $ C > 0 $ such that for all  dyadic numbers $ N,M \geq 1$ and $u, v  \in \mathcal{S}'(\R^d)$,
\begin{multline*}
 \|  \Delta_N ( v )   \Delta_M (u) \|_{L^2 ( \R ; L^2 ( \R^d ))  } 
\leq  \\
\leq C   \min ( N,M)^{ \frac{d-2}{2} }   \left( \frac{ \min ( N,M)  }{ \max ( N,M )} \right) ^{1/2-\delta }  \| \Delta_N(v) \|_{  \X^{0,b}} \| \Delta_M(u) \|_{  \X^{0,b}}.
\end{multline*}
\end{prop}

 \begin{prop} \label{1-bilibourgain4}
Let $ \psi \in C_0^\infty (\R) $ then for all $ b \in ] \frac{1}{2} , 1 ] $ and $ \delta \in ]0, \frac{1}{2} ] $, there exists a constant $ C > 0 $ such that for all  dyadic numbers $ N,M \geq 1$ and $u_0, v  \in \mathcal{S}'(\R^d)$
\begin{multline*}
  \|  \Delta_N \big( \psi(t) e^{-itH} u_0  \big)   \Delta_M (u) \|_{L^2 ( \R ; L^2 ( \R^d ))  }  \leq \\
  \leq C  \min ( N,M)^{ \frac{d-2}{2} }   \left( \frac{ \min ( N,M)  }{ \max ( N,M )} \right) ^{1/2-\delta }  \| \Delta_N(u_0) \|_{ L^2 (\R^d)} \| \Delta_M(u) \|_{  \X^{0,b}}.
\end{multline*}
\end{prop}
Indeed, for any $ \epsilon > 0 $, according to Proposition~\ref{1-bourgain2} (with $ \theta = \frac{1}{2} $), we obtain
\begin{eqnarray*}
  \| \Delta_N ( v )   \Delta_M (u)  \|_{L^2(\R ; L^2(\R^d))} & \leq &  \|\Delta_N ( v )\|_{L^4(\R ; L^2(\R^d))}  \|\Delta_M (u)\|_{L^4(\R ; L^\infty(\R^d))}
\\  & \leq &C  \| \Delta_N ( v )\|_{  \X^{0,1/4+\epsilon} }   \| \Delta_M (u) \|_{  \X^{d/2+\epsilon,1/4+\epsilon} },
\end{eqnarray*}

\begin{multline*}
\| \Delta_N \big( \psi(t) e^{itH} u_0 \big)   \Delta_M (u)  \|_{L^2(\R ; L^2(\R^d))}  \leq  \\
\begin{aligned}
   & \leq \|\Delta_N \big( \psi(t) e^{itH} u_0 \big)\|_{L^4(\R ; L^\infty(\R^d))}  \|\Delta_M (u)\|_{L^4(\R ; L^2(\R^d))} \\
    & \leq  C  \| \Delta_N ( u_0 )\|_{  \mathcal{H}^{d/2-1/2+\epsilon}(\R^d) }   \| \Delta_M (u) \|_{  \X^{0,1/4+\epsilon}  } \\
  & \leq  C  \| \Delta_N ( u_0 )\|_{  \mathcal{H}^{d/2+\epsilon}(\R^d) }   \| \Delta_M (u) \|_{  \X^{0,1/4+\epsilon}  },
\end{aligned}
\end{multline*}
and
\begin{multline*}
  \| \Delta_N \big( \psi(t) e^{itH}u_0 \big)   \Delta_M (u)  \|_{L^2(\R ; L^2(\R^d))} \leq \\
  \begin{aligned}
& \|\Delta_N \big( \psi(t) e^{itH} u_0  \big)\|_{L^4(\R ; L^2(\R^d))}  \|\Delta_M (u)\|_{L^4(\R ; L^\infty(\R^d))}\\
&\leq  C  \|\Delta_N \big( \psi(t) e^{itH} u_0  \big)\|_{L^\infty(\R ; L^2(\R^d))}   \| \Delta_M (u) \|_{  \X^{d/2+\epsilon,1/4+\epsilon} }\\
 &\leq  C  \|\Delta_N ( u_0  )\|_{L^2(\R^d)}   \| \Delta_M (u) \|_{  \X^{d/2+\epsilon,1/4+\epsilon} }.
\end{aligned}
\end{multline*}
Therefore, by interpolation, for all $ \theta \in [0,1] $, we obtain
\begin{multline*}
\|  \Delta_N ( v )   \Delta_M (u)  \|_{L^2 ( \R ;  L^2 ( \R^d ))  }  \leq \\
 \leq C  \min( M,N) ^{ \frac{d-2}{2} +\theta ( 1+ \epsilon ) }   \left( \frac{ \min(M,N)  }{ \max(M,N) } \right) ^{(1/2-\delta)(1-\theta) } \\
   \cdot \| \Delta_N(v) \|_{  \X^{0,b(1-\theta)+\theta(1/4+\epsilon)}} \| \Delta_M(u) \|_{  \X^{0,b(1-\theta)+\theta(1/4+\epsilon)}}
\end{multline*}
and
\begin{multline*}
\|  \Delta_N  \big( \psi(t)  e^{itH} u_0  \big)   \Delta_M (u)  \|_{L^2 ( \R ;  L^2 ( \R^d ))  } \leq \\
\leq   C   \min( M,N) ^{ \frac{d-2}{2} +\theta ( 1+ \epsilon ) }
  \left( \frac{ \min(M,N)  }{ \max(M,N) } \right) ^{(1/2-\delta)(1-\theta) }  \\
  \cdot  \| \Delta_N(u_0) \|_{  L^2(  \R^d ) }  \| \Delta_M(u) \|_{  \X^{0,b(1-\theta)+\theta(1/4+\epsilon)}}.
\end{multline*}
Choose  $ \delta = \frac{\epsilon}{2} $ and $ \theta = \frac{\epsilon}{4} $ then 
\begin{eqnarray*}
b(1-\theta)+\theta( \frac{1}{4} +\epsilon) & = & b - \frac{b\epsilon}{4} + \frac{\epsilon}{4} (  \frac{1}{4}+ \epsilon ) \\
 & \leq & b - \frac{\epsilon}{8} + \frac{\epsilon}{16} + \frac{\epsilon^2}{4} \leq b - \frac{\epsilon}{17}.
\end{eqnarray*}
It is then enough to take $ b = \frac{1}{2} + \frac{\epsilon}{34} $ and to set $ b ' = b(1-\theta)+\theta( \frac{1}{4} +\epsilon) < \frac{1}{2} $ to obtain
\begin{multline*}
 \|  \Delta_N ( v )   \Delta_M (u)  \|_{L^2 ( \R ;  L^2 ( \R^d ))  }  \leq \\
\leq   C    \min(N,M)^{ \frac{d-2}{2} + \epsilon }    \left( \frac{ \min(M,N)  }{ \max(N,M) } \right) ^{(1/2-\delta)(1-\theta) }   \| \Delta_N(v) \|_{  \X^{0,b'}} \| \Delta_M(u) \|_{  \X^{0,b'}}
\end{multline*} 
and
\begin{multline*}
 \|  \Delta_N \big( \psi(t) e^{-itH} u_0  \big)   \Delta_M (u)  \|_{L^2 ( \R ;  L^2 ( \R^d ))  } \leq \\
\leq  C   \min(N,M)^{ \frac{d-2}{2} + \epsilon }   \left( \frac{ \min(M,N)  }{ \max(M,N) } \right) ^{(1/2-\delta)(1-\theta) }   \| \Delta_N(u_0) \|_{ L^2(\R^d)} \| \Delta_M(u) \|_{  \X^{0,b'}}.
\end{multline*}
To conclude, it is sufficient to note that
\begin{equation*}
\left( \frac{1}{2} -\delta \right) (1-\theta) = \frac{1}{2} -\frac{5\epsilon}{8} + \frac{\epsilon^2}{8} \geq \frac{1}{2} - \epsilon  
\end{equation*}
and the inequalities \eqref{1-bilibourgain1} and \eqref{1-bilibourgain2} follow with $ \delta_0 = \epsilon $.

Then, as for \cite[Lemma 4.4]{burq3}, to prove Proposition~\ref{1-bilibourgain3}, it is sufficient to establish the following proposition:

 \begin{prop} \label{1-bilibourgain5}
For all $ b \in ]\frac{1}{2} , 1 ] $ and $ \delta \in ]0, \frac{1}{2} ]$, there exists a constant $ C > 0 $ such that for all dyadic numbers $ N,M \geq 1$ and $u, v  \in \mathcal{S}'(\R^d)$
\begin{multline*}
  \|  \Delta_N ( v )   \Delta_M (u) \|_{L^2 ( [0,1] ;  L^2 ( \R^d ))  }  
\leq  \\
\leq C  \min ( N,M)^{ \frac{d-2}{2} }   \left( \frac{ \min ( N,M)  }{ \max ( N,M )} \right) ^{1/2-\delta }   \| \Delta_N(v) \|_{  \X^{0,b}} \| \Delta_M(u) \|_{  \X^{0,b}}.
\end{multline*}
\end{prop}
Finally, to obtain Proposition~\ref{1-bilibourgain4} and Proposition~\ref{1-bilibourgain5}, it is enough to use \cite[Lemma 2.1]{burq1} and Theorem~\ref{1-bilis}.

 %%%%%%%%%%%%%%%%%%%%%%%%%%%%%%%%%%%%%%%%%%%%%%%%%%%%%%%%%%%%%%%%%%%%%%%%%%%%%%

\section{The smoothing effect for the harmonic oscillator}\label{Sect4}

\subsection{Some preliminary results}
We start by establishing two preliminary lemmas.

\begin{lem}  \label{2commutateur}
Let $ s_1 $ and $ s_2 $ be two real numbers.
\begin{enumerate}[(i)]
 \item If $ \max( s_2,s_1+ s_2 ) \leq 1 $ then there exists a constant $ C > 0 $ such that for all $ u \in L^2(\R^d) $,
\begin{equation*}
\| \, [\sqrt H^{s_1+s_2},\langle x\rangle ^{-s_1}]u \|_{L^2(\R^d) } \leq C \|u \|_ {L^2(\R^d)},
 \end{equation*}
 \item If $ s_2 \geq -1 $ then there exists a constant $ C > 0 $ such that for all $ u \in  H^{s_1-1} (\R^d) $,
\begin{equation*}
 \| \, [   \sqrt{-\Delta} ^{s_1},\langle x\rangle ^{-s_2}]u \|_{L^2(\R^d) } \leq C \|u \|_ {H^{s_1-1} (\R^d)},
\end{equation*}
 \item If $ s_2 \leq 1 $ then there exists a constant $ C > 0 $ such that for all $ u \in  \mathcal{H}^{s_1-s_2} (\R^d) $,
\begin{equation*}
 \| \, [  \sqrt{H}^{s_1},\langle x\rangle ^{-s_2}]u \|_{L^2(\R^d) } \leq C \|u \|_ {\mathcal{H}^{s_1-s_2} (\R^d)}.
\end{equation*}
\end{enumerate}
\end{lem}

\begin{proof} To evaluate the regularity of the previous commutators, we use the Wey-H\"ormander  pseudo-differential calculus  associated with the metric $ \frac{dx^2}{1+|x|^2} + \frac{d\xi^2}{1+|\xi|^2} $. \medskip

 The class of symbols $ S(\mu,m) $ associated to the previous metric is the space of regular functions on $ \R^d \times \R ^d $ which satisfy $ | \partial^\alpha_x  \partial_\xi^\beta a (x , \xi )  | \leq C_{ \alpha, \beta  } \langle x\rangle ^{\mu-\alpha} \langle  \xi \rangle ^{m-\beta }  $. \medskip
 
 Thus, we have (see \cite[Section 18.5]{Hormander}, \cite{Robert} or \cite{Bouclet}) that if $ a_1 \in S(\mu_1,m_1 ) $ and $ a_2 \in S( \mu_2, m_2 ) $  then the commutator $ [ Op(a_1),Op(a_2) ] $ is a pseudo-differential operator with a symbol in the class $ S ( \mu_1+\mu_2-1  , m_2+m_2-1) $.

Here we will use that $a(x,\xi)= (|x|^2+|\xi|^2)^{\alpha/2} \in S(\alpha,\alpha)$.

$(i)$ As a consequence, if $ \max( s_2,s_1+ s_2 ) \leq 1 $
\begin{equation*}
[\sqrt H^{s_1+s_2},\langle x\rangle ^{-s_1}] \in S (s_2-1 , s_1+s_2-1) \subset S ( 0 , s_1 + s_2-1 ).
\end{equation*}
Moreover, as recalled in \cite{Martinez}, if $ q \in S(0,\mu ) $ then for any $ s \in \R $, there exists a constant $ C> 0 $ such that
\begin{equation*}
\|Op(q) u \|_{H^{s-\mu}(\R^d)} \leq C \| u \|_{H^s(\R^d)}.
\end{equation*}
Thus, we can take $ s= \mu = s_1+ s_2 -1 $ to obtain that
\begin{equation*}
\| [\sqrt H^{s_1+s_2},\langle x\rangle ^{-s_1}] u \| _{L^2(\R^d) } \leq C \|u\|_{H ^{s_1+s_2 -1}(\R^d) } \leq \|u\|_{L^2 (\R^d)}.
\end{equation*}

$(ii)$ Similarly
\begin{equation*}
[ \  \sqrt{-\Delta} ^{s_1},\langle x\rangle ^{-s_2}] \in S (-s_2-1,s_1-1) \subset S (0,s_1-1),
\end{equation*}
and we can conclude as previously. \medskip

$(iii)$  We have
\begin{equation*}
[  \sqrt{H}^{s_1}, \langle x\rangle ^{-s_2} ]  \sqrt{H}^{s_2-s_1} \in S (-1,s_2-1) \subset S (0,s_2-1).
\end{equation*}
Then
\begin{equation*}
\| [  \sqrt{H}^{s_1}, \langle x\rangle ^{-s_2} ]  \sqrt{H}^{s_2-s_1} u \|_{L^2(\R^d)} \leq C \|u\|_{L^2(\R^d)},
\end{equation*}
and it is enough to replace $ u $ by $ \sqrt{H}^{s_1-s_2} u $ to obtain the desired result.
\end{proof}

\begin{lem}  \label{2inter}
Let $ a \in S(2\eps,0)$.  Then for all  $ u \in \mathcal{H}^{1/2+\eps}(\R^d) $, 
\begin{equation*}
 \bigg| \int_{\R^d} a(x) \cdot \nabla u (x)   \overline{u} (x)  dx \bigg|  \leq C \|u\|^2_{\mathcal{H}^{1/2+\eps}(\R^d)}.
\end{equation*}
\end{lem}

\begin{proof} By duality, we have 
\begin{equation}\label{dual}
 \bigg| \int_{\R^d} a(x) \cdot \nabla u (x)   \overline{u} (x)  dx \bigg|  \leq  \| a u\|_{\mathcal{H}^{1/2-\eps}(\R^d)} \| \nabla u\|_{\mathcal{H}^{-1/2+\eps}(\R^d)}.
\end{equation}
We use again the Weyl-H\"ormander pseudo-differential calculus introduced in the proof of Lemma~\ref{2commutateur}. Since $ a \in S(2\eps,0)$, we deduce that $\| a u\|_{\mathcal{H}^{1/2-\eps}(\R^d)} \leq C \| u\|_{\mathcal{H}^{1/2+\eps}(\R^d)}$. Then we observe that $ \| \nabla u\|_{\mathcal{H}^{-1/2+\eps}(\R^d)} \leq C  \|  u\|_{\mathcal{H}^{1/2+\eps}(\R^d)}$, which implies the result by \eqref{dual}.
\end{proof}

With these different lemmas established, we can proceed to the proof of the smoothing effect.
\medskip
 
\subsection{Proof of \texorpdfstring{\eqref{2effectregularisant1}}{}} ~

\textbf{Step 1:}  Let us  show that for any $ \alpha < 1 $, there exists a constant $ C> 0 $ such that for any  $ u \in \mathcal{H}^{  (2-\alpha)/2 }(\R^d) $
\begin{multline} \label{2fin1}
2(1-\alpha) \Big\| \frac{1}{\langle x\rangle ^{\alpha/2}} \nabla u \Big\|_{L^2(\R^d)}^2 \leq  \\
\leq C \|u\|_{  \mathcal{H}^{(2-\alpha)/2}(\R^d)}^2- \Re \bigg(  \int_{\R^d }  \left[ \frac{x \cdot \nabla}{\langle x\rangle ^\alpha} ,  H   \right] u (x) \overline{u} (x) dx \bigg)  .
\end{multline}
By Theorem~\ref{1-compose} with $h=1$, we obtain, since $\nabla=i Op(\xi)$,
\begin{multline}\label{commut}
  \left[ \frac{x \cdot \nabla}{\langle x\rangle ^\alpha} , \Delta  \right] = \\
  \begin{aligned}
&  =  \big(\frac{x \cdot \nabla}{\langle x\rangle ^\alpha} \big) \Delta-\Delta \big( \frac{x \cdot \nabla}{\langle x\rangle ^\alpha} \big)   \\
& =   i Op \left( \frac{x \cdot\xi}{ \langle x\rangle ^\alpha} \right) Op \left( -|\xi|^2 \right) - i Op \left(  -|\xi|^2 \right) Op \left( \frac{x \cdot \xi }{\langle x\rangle ^\alpha} \right)  \\
& =   2Op \left(      \Big( \frac{|\xi|^2}{\langle x\rangle ^\alpha } -\alpha \frac{(x \cdot \xi)^2}{ \langle x\rangle ^{\alpha+2} } \Big) + i \alpha (d+2) \frac{ x \cdot \xi}{\langle x\rangle ^{\alpha+2}} -  i\alpha ( \alpha+2 ) \frac{ (x \cdot \xi ) |x|^2 }{\langle x\rangle ^{\alpha+4}} \right).
\end{aligned}
\end{multline}
Then, using that $ \alpha < 1 $, we obtain
\begin{multline*}
 \Re  \left(  \int_{\R^d } Op \left(    \Big( \frac{|\xi|^2}{\langle x\rangle ^\alpha } -\alpha \frac{(x \cdot \xi)^2}{ \langle x\rangle ^{\alpha+2} } \Big)  u (x) \,   \overline{u} (x) dx \right) \right) =\\
 \begin{aligned}
 & =  \Re  \int_{\R^d}   \left( \frac{- \Delta u }{\langle x\rangle ^\alpha}     \overline{u} + \alpha \frac{(x \cdot \nabla)^2u -   ( x \cdot \nabla ) u}{\langle x\rangle ^{\alpha+2}} \,  \overline{u}       \right)dx \\
 &=   \Re  \int_{\R^d}  \left(  \nabla u \cdot  \nabla ( \frac{ \overline{u} }{\langle x\rangle ^{\alpha}} ) -\alpha ( x\cdot \nabla u )    \   div \left( \frac{ x  \overline{u} }{\langle x\rangle ^{\alpha+2}}  \right)  - \alpha  \frac{(x \cdot \nabla  u) }{\langle x\rangle ^{\alpha+2}} \, \overline{u}   \right)dx \\
 & =  \int_{\R^d } \left( \frac{|\nabla u |^2}{\langle x\rangle ^\alpha}   -\alpha \frac{| x \cdot \nabla u |^2}{\langle x\rangle ^{\alpha+2}}  \right)dx + \\
 & \hspace*{3cm} + \alpha  \Re   \int_{\R^d}  \left( (\alpha+2) \frac{|x|^2 ( x \cdot \nabla u ) }{\langle x\rangle ^{\alpha+4}}  \,  \overline{u} - (d+2) \frac{ (x \cdot \nabla  u )   }{\langle x \rangle ^{\alpha+2}} \,\overline{u}  \right)dx \\
 & \geq (1-\alpha)   \Big\| \frac{1}{\langle x\rangle ^{\alpha/2}} \nabla u \Big\|^2_{L^2(\R^d)}+ \\
  & \hspace*{3cm} + \alpha  \Re  \int_{\R^d}  \left( (\alpha+2) \frac{|x|^2 ( x \cdot \nabla u ) }{\langle x\rangle ^{\alpha+4}}  \,  \overline{u} - (d+2) \frac{ (x \cdot \nabla  u)    }{\langle x \rangle ^{\alpha+2}} \, \overline{u}   \right)dx.
  \end{aligned}
 \end{multline*}
Thanks to Lemma~\ref{2inter} with $\eps=0$, we establish
\begin{equation*} 
 \Big| \int_{\R^d} Op \Big(   \frac{ x \cdot \xi}{\langle x\rangle ^{\alpha+2}}   \Big) u \,  \overline{u}  dx \Big|,  \; \Big| \int_{\R^d} Op \left(   \frac{ (x \cdot \xi )  |x|^2 }{\langle x\rangle ^{\alpha+4}} \right) u \,  \overline{u}  dx \Big| \leq C \|u\|^2_{\mathcal{H}^{1/2}(\R^d)},
\end{equation*}
and
\begin{equation*} 
\Big|  \int_{\R^d} \Big(  \frac{|x|^2 ( x \cdot \nabla u ) }{\langle x\rangle ^{\alpha+4}}  \,  \overline{u}   \Big)dx\Big| , \; \Big|  \int_{\R^d} \Big(    \frac{ x \cdot \nabla  u    }{\langle x \rangle ^{\alpha+2}}   \overline{u}  \Big)dx\Big| \leq C \|u\|^2_{\mathcal{H}^{1/2}(\R^d)}.
\end{equation*}
Thus by \eqref{commut}, for any $ \alpha < 1 $, there exists a constant $ C> 0 $ such that for any $ u \in H^{1/2}(\R^d) $,
\begin{equation*}
\Re   \bigg( \int_{\R^d }  \left[ \frac{x \cdot \nabla}{\langle x\rangle ^\alpha} ,  \Delta  \right] u (x) \overline{u} (x) dx \bigg) \geq 2(1-\alpha) \Big\| \frac{1}{\langle x\rangle ^{\alpha/2}} \nabla u \Big\|_{L^2(\R^d)}^2 -C \|u\|_{\mathcal{H}^{1/2}(\R^d)}^2.
\end{equation*}

In a similar way, we have  
\begin{equation*}
  \left[  \frac{x \cdot \nabla}{\langle x\rangle ^\alpha} , -|x|^2 \right] = -2Op \left( \frac{|x|^2}{\langle x\rangle ^\alpha} \right),
 \end{equation*}
and therefore
\begin{eqnarray*}
\Re \bigg(  \int_{\R^d}  \left[ \frac{x \cdot \nabla}{\langle x\rangle ^\alpha} , -|x|^2 \right] u(x) \, \overline{u(x)} \bigg) dx & =& -2 \int_{\R^d} \frac{|x|^2}{\langle x\rangle ^\alpha} |u(x)|^2 \  dx 
\\ & \geq &-C \|u\|^2_{   \mathcal{H}^{(2-\alpha)/2}(\R^d)}.
\end{eqnarray*}
This implies \eqref{2fin1}. \medskip

\textbf{Step 2:} We now prove that for any $ \alpha \in ] 0 ,1 [ $, there exists a constant $ C > 0 $ such that for any $ T \geq 0 $ and $ u_0 \in \mathcal{H}^{(2-\alpha)/2}(\R^d)$,
\begin{equation}\label{sl}
\int_0^T  \Big\| \frac{1}{\langle x\rangle ^{\alpha/2 }} \nabla \big(e^{-itH} u_0\big)  \Big\|^2_{L^2(\R^d)} dt \leq   CT \|u_0\|_{  \mathcal{H}^{(2-\alpha)/2}(\R^d)}^2.
\end{equation}

Let us set $ u = e^{-itH} u_0 $, then we get
\begin{multline*}
   -i  \int_{\R^d }  \left[ \frac{x \cdot \nabla}{\langle x\rangle ^\alpha} ,  H  \right] u (t,x)  \, \overline{u} (t,x)  dx =  \\
 \begin{aligned}  
   &  = \int_{\R^d} \frac{x \cdot \nabla \partial_t u(t,x)} {\langle x\rangle ^\alpha} \, \overline{ u(t,x)}  dx +i \int_{\R^d} \frac{x \cdot \nabla}{\langle x\rangle ^\alpha} u(t,x)  \, \overline{Hu(t,x)} dx \\
&   =     \int_{\R^d} \frac{x \cdot \nabla \partial_t u(t,x) }{\langle x\rangle ^\alpha} \, \overline{ u(t,x)}  dx+   \int_{\R^d} \frac{x \cdot \nabla}{\langle x\rangle ^\alpha} u(t,x) \, \overline{\partial_t u(t,x)} dx \\
& =   \partial_t \left( \int_{\R^d }  \frac{x \cdot \nabla}{\langle x\rangle ^\alpha} u (t,x) \, \overline{u(t,x)}  dx  \right) .
 \end{aligned}  
\end{multline*}
Thus, thanks to \eqref{2fin1}, we obtain for $ T \geq 0 $,
\begin{multline*}
   2(1-\alpha) \int_0^T \Big\| \frac{1}{\langle x\rangle ^{\alpha/2}} \nabla( e^{-itH } u_0 )   \Big\|^2_{L^2(\R^d)} dt  \leq \\
 \leq CT\|u_0\|^2_{\mathcal{H}^{(2-\alpha)/2}(\R^d)}  + \Re \left( i \int_{\R^d}   \frac{x \cdot \nabla  u_0 }{\langle x\rangle ^\alpha}  \, \overline{ u_0} - \frac{x \cdot \nabla \big(e^{-iTH} u_0\big) }{\langle x\rangle ^\alpha} \, \overline{e^{-iTH} u_0}   dx  \right) . 
 \end{multline*} 
Then, we can apply  Lemma~\ref{2inter} with $\eps=1-\alpha/2$, to deduce \eqref{sl}. \medskip

\textbf{Step 3 :} In \eqref{sl}, we choose $ \alpha = 1 - 2 \epsilon $ with $ \epsilon \in ] 0 , \frac{1}{2} [ $ so that we have
\begin{equation}\label{born1}
\int_0^T \Big\|   \frac{1}{\langle x\rangle ^{1/2-\epsilon}} \nabla \big(e^{-itH} u_0\big)    \Big\|^2_{L^2(\R^d)} dt \leq   CT \|u_0\|_{ \mathcal{H}^{1/2+\epsilon}(\R^d)}^2.
\end{equation}
Using Lemma~\ref{2commutateur}, we get
\begin{multline*} 
  \int_0^T   \Big\| \frac{1}{\langle x\rangle ^{1/2-\epsilon}} H^{1/2} e^{itH}u_0  \Big\|_{L^2(\R^d)}^2 dt   \leq \\
  \begin{aligned}
 & \leq 2 \int_0^T \Big\|  H^{1/2} \frac{1}{\langle x\rangle ^{1/2-\epsilon}} e^{itH}u_0 \Big\|_{L^2(\R^d)}^2dt +2\int_0^T \Big\|  \Big[ \frac{1}{\langle x\rangle ^{1/2-\epsilon}} , H^{1/2} \Big] e^{itH}u_0 \Big\|_{L^2(\R^d)}^2 dt\\
 & \leq 2 \int_0^T \Big\|  H^{1/2} \frac{1}{\langle x\rangle ^{1/2-\epsilon}} e^{itH}u_0 \Big\| _{L^2(\R^d)}^2dt + CT  \|u_0\|^2_{ \mathcal{H}^{1/2+\epsilon}(\R^d)}.
\end{aligned} 
\end{multline*} 
Then, using  \eqref{1-comparaison} and \eqref{born1}, we obtain
\begin{multline*} 
  \int_0^T \Big\|  H^{1/2} \frac{1}{\langle x\rangle ^{1/2-\epsilon}} e^{itH}u_0 \Big\|_{L^2(\R^d)}^2 \leq \\
\begin{aligned}
 & \leq C\int_0^T \Big\| \langle x\rangle ^{1/2 + \epsilon} e^{itH} u_0 \Big\| ^2_{L^2(\R^d)} dt + C\int_0^T \Big\| (-\Delta)^{1/2} \Big( \frac{1}{\langle x\rangle ^{1/2-\epsilon}}e^{itH} u_0 \Big) \Big\|^2_{L^2(\R^d)} dt  \\
 &\leq  CT \|u_0\|^2_{\mathcal{H}^{1/2+\epsilon}(\R^d)} + C\int_0^T \Big\| \nabla \Big( \frac{1}{\langle x\rangle ^{1/2-\epsilon}}e^{itH} u_0 \Big) \Big\|^2_{L^2(\R^d)} dt  \\
  & \leq CT \|u_0\|^2_{ \mathcal{H}^{1/2+\epsilon}(\R^d)} +C \int_0^T \Big\|  \frac{1}{\langle x\rangle ^{1/2-\epsilon}} \nabla  \big( e^{itH} u_0 \big) \Big\|^2_{L^2(\R^d)} dt  \\
 & \leq CT \|u_0\|^2_{ \mathcal{H}^{1/2+\epsilon}(\R^d)}.
\end{aligned}
\end{multline*} 
Therefore, we obtain 
\begin{equation*}
\int_0^T \Big\| \frac{1}{\langle x\rangle ^{1/2-\epsilon}} H^{1/2} e^{-itH} u_0 \Big\|^2_{L^2(\R^d)} dt   \leq  CT \|u_0\|^2_{ \mathcal{H}^{1/2+\epsilon}(\R^d)}.
\end{equation*}
And we can replace $ u_0 $ by $ H^{-1/4-\epsilon/2} u_0 $ to deduce \eqref{2effectregularisant1}. 
\subsection{Proof of \eqref{2effectregularisant2}} 
Using  Lemma~\ref{2commutateur} and \eqref{2effectregularisant1}, we get
\begin{multline*} 
\Big\| \sqrt H ^{d/2-2\epsilon} \frac{1}{\langle x\rangle ^{1/2-\epsilon} } e^{itH} u_0 \Big\|_{L^2([- \pi,\pi]\times \R^d)}  \leq \\
\begin{aligned}
 & \leq \Big\|  \frac{1}{\langle x\rangle ^{1/2-\epsilon} } \sqrt H ^{d/2-2\epsilon} e^{itH} u_0 \Big\|_{L^2([- \pi,\pi]\times \R^d)} + \\
 &\hspace{5cm}+\Big\|  \Big[ \sqrt H ^{d/2-2\epsilon} , \frac{1}{\langle x\rangle ^{1/2-\epsilon} } \Big]  e^{itH} u_0 \Big\| _{L^2([- \pi,\pi]\times \R^d)} \\
  & \leq C \|u_0\|_{   \mathcal{H}^{  (d-1)/2 } (\R^d)}.
\end{aligned}
\end{multline*} 
Then, using \eqref{1-comparaison}, we establish
\begin{equation*}
\Big\| \sqrt{-\Delta} ^{d/2-2\epsilon}  \Big(  \frac{1}{\langle x\rangle ^{1/2-\epsilon} } e^{itH} u_0  \Big) \Big\|_{L^2([- \pi,\pi]\times \R^d)}  \leq C\|u_0\|_{   \mathcal{H}^{  (d-1)/2  } (\R^d)}.
\end{equation*}
And finally, using Lemma~\ref{2commutateur}, we can conclude that 
\begin{multline*} 
  \Big\|  \frac{1}{\langle x\rangle ^{1/2-\epsilon} } \sqrt{-\Delta}^{d/2-2\epsilon} ( e^{itH} u_0 ) \Big\|_{L^2([- \pi,\pi]\times \R^d)}  \leq \\
\begin{aligned}
 & \leq  \Big\|  \Big[ \frac{1}{\langle x\rangle ^{1/2-\epsilon} } , \sqrt{-\Delta}^{d/2-2\epsilon}  \Big]  e^{itH} u_0 \Big\|_{L^2([- \pi,\pi]\times \R^d)}  +\\
 &\hspace{5cm}+  \Big\|  \sqrt{-\Delta}^{d/2-2\epsilon} \Big( \frac{1}{\langle x\rangle ^{1/2-\epsilon} } e^{itH} u_0 \Big)  \Big\|_{L^2([- \pi,\pi]\times \R^d)} \\
& \leq  C \|u_0\|_{  \mathcal{H}^{  (d-1)/2  } (\R^d)}.  
\end{aligned}
\end{multline*}

\section{Random initial data and Sobolev spaces}\label{Sect5}
This section is devoted to the proof of Theorem \ref{2sobolev2}, which shows that the initial randomized data does not gain any derivative in $ L^2( \R^d )$ and does not enjoy any better localisation property.

The proof of this result will rely on micro-analysis tools. We first step is to establish the following statement which gives a precise description of the phase-space localisation of the Hermite functions.

\begin{prop} For all $ s \geq 0 $, there exist two constants $ C_1,C_2 > 0 $ such that for all $ n \in \N  $,
\begin{equation} \label{2propre2}
C_1 \lambda_n ^ s \leq \|    (-\Delta)^{s/2}  h_n \|_{L^2(\R^d)}  \leq C_2\lambda_n ^ s,
\end{equation}
and 
\begin{equation} \label{3propre2}
C_1 \lambda_n ^ s \leq \|   \langle x\rangle ^{s/2} h_n \|_{L^2(\R^d)}  \leq C_2\lambda_n ^ s.
\end{equation}
\end{prop}

\begin{proof}   Let us first prove this result  in the particular case where $(h_n)_{n \geq 0}$ is the basis of the tensor eigenfunctions, since the argument is then particularly easy. Using~\eqref{1-comparaison}, we have 
\begin{multline*}
 c \big( \| (-\Delta)^{s/2}  h_n \|_{L^2(\R^d)} + \|\langle x\rangle ^s h_n  \| _{L^2(\R^d)} \big)   \leq \lambda_n^s =  \| h_n \|_{ \mathcal{H}^s(\R^d) }  \leq \\
  \leq C \big( \| (-\Delta)^{s/2}  h_n \|_{L^2(\R^d)} + \|\langle x\rangle ^s h_n  \| _{L^2(\R^d)} \big),
 \end{multline*}
then
 \begin{multline*}
 C' \| (-\Delta)^{s/2}  h_n \|_{L^2(\R^d)} \leq \lambda_n^s \leq \\
 \leq C \big( \| (-\Delta)^{s/2}  h_n \|_{L^2(\R^d)} + \|  (-\Delta)^{s/2}  ( \widehat{h}_n  ) \| _{L^2(\R^d)} + \| h_n \|_{L^2(\R^d)} \big).
\end{multline*}
But since the eigenfunctions are tensor functions, then $ h_n(x) = \e^{i \theta_n} \widehat{h}_n(x) $, for some $\theta_n \in \R$, since this equality is true in dimension 1, and \eqref{2propre2} follows. The proof of~\eqref{3propre2} is similar.  \medskip

 Let us now prove \eqref{2propre2} in the general case. We set $ h = \frac{1}{  \lambda_n^2 } $ and $ \Phi_h( x) =  \frac{1}{h^{d/4}}    h_n( \lambda_n x ) $ so that $ ( -h^2 \Delta + |x|^2 -1 )  \Phi_h = 0 $ and $  \| \Phi_h \|_{L^2(\R^d)}  =1$. To prove \eqref{2propre2}, it is sufficient to establish that there exists a constant $ C_1> 0 $ such that for all $ h > 0 $, 
\begin{equation*}
h^s \| (-\Delta)^{s/2}  \Phi_h  \|_{L^2(\R^d)} \geq C_1.
\end{equation*}
Let us proceed by contradiction and suppose that 
\begin{equation} \label{2absurdite}
\underset{h \rightarrow 0}{\underline{\lim}} \ h^s \| (-\Delta)^{s/2}  \Phi  \|_{L^2(\R^d)} = 0.
\end{equation}
By   \cite[Theorem 2]{burq8}, there exists a positive measure $ \mu \in  \mathcal{M}_+ ( \R^d\times\R^d ) $ such that for any function $ a \in C_0^\infty(\R^d\times\R^d   )$ ,
\begin{equation*}
\lim_{ h \rightarrow 0 } \langle  a ( x, h |\nabla| )  \Phi_h , \Phi_h \rangle  _{ L^2(\R^d)\times L ^2(\R^d)   } = \int_{ \R^d \times \R^d  } tr ( a(x, \xi ) )   \mu(  d x d \xi  ).
\end{equation*}
Let us recall that $ (x, \xi) \in Supp( \mu )^c $ if and only if there exists $ r > 0 $ such that for all $ \phi \in C_0^\infty ( B(x,r) \times B(\xi,r)  )$,
\begin{equation*}
\int_{ \R^d\times\R^d   } \phi (x,\xi)  \mu(dx,d\xi) = 0.
\end{equation*}
By Proposition~\ref{1-inverse}, if $ a \in C_0^\infty (\R^d\times\R^d  ) $  with $ Supp ( a)  \cap  \lbrace (x,\xi) :  |x|^2 + |\xi|^2  = 1   \rbrace =  \emptyset $ then  for all  $ N \in \N $, there  exists $ E_N \in Op ( T^{-2} ) $ and $ R_N \in Op ( T^{-(N+1)} ) $ such that 
\begin{equation*}
E_N \circ (-h^2 \Delta + |x|^2 - 1  )   = a ( x, h|\nabla| ) - h^{N+1} R_N.
\end{equation*}
Therefore
\begin{equation*}
\langle  a ( x, h|\nabla| )  \Phi_h , \Phi_h \rangle  _{ L^2(\R^d)\times L ^2(\R^d)   } =  h^{N+1} \langle  R_N \Phi_h , \Phi_h \rangle  _{ L^2(\R^d)\times L ^2(\R^d)   },
\end{equation*}
then 
\begin{equation*}
 \int_{ \R^d \times \R^d  }  a(x, \xi )   \mu(  d x d \xi  ) = 0.
\end{equation*}
And finally, we establish that
\begin{equation*}
Supp ( \mu ) \subset \big\{  (x,\xi) \ :  \ |x|^2  + |\xi|^2 = 1 \big\}.
\end{equation*}
Again, by Proposition~\ref{1-inverse}, if $ a \in C_0^\infty (\R^d\times\R^d  ) $  with  $ Supp ( a)  \cap  \lbrace (x,\xi) : \xi^2  = 0   \rbrace =  \emptyset $ then  for all $ N \in \N $, there exists $ E_N \in Op ( S^{-s} ) $ and $ R_N \in Op ( S^{-(N+1)} ) $ such that
\begin{equation*}
E_N \circ \sum_{j=1}^d | h \partial_{x_j} | ^{s}   = a ( x, h|\nabla| ) - h^{N+1} R_N.
\end{equation*}
But according to \cite{Martinez} and \eqref{2absurdite}, we get
\begin{eqnarray*}
 \lim_{ h \rightarrow 0 }  \big| \langle  E_N \circ \sum_{j=1}^d | h \partial_{x_j} | ^{s} \Phi_h , \Phi_h \rangle  _{ L^2(\R^d)\times L ^2(\R^d)   } \big| & \leq& \lim_{ h \rightarrow 0 } \big\| E_N \circ \sum_{j=1}^d | h \partial_{x_j} | ^{s} \Phi_h \big\|_{L ^2(\R^d)}
 \\ & \leq &\lim_{ h \rightarrow 0 } \big\| \sum_{j=1}^d | h \partial_{x_j}| ^{s} \Phi_h \big\|_{L ^2(\R^d)} = 0.
\end{eqnarray*}
Therefore
\begin{equation*}
 \int_{ \R^d \times  \R^d  }  a(x, \xi )   \mu(  d x d \xi  ) = 0,
\end{equation*}
and we establish that
\begin{equation*}
Supp ( \mu ) \subset \lbrace  (x,\xi)\ : \ |\xi|^2 = 0 \rbrace.
\end{equation*}
Then for $ a \in  C_0^\infty(  \R^d\times\R^d ) $ we have 
\begin{eqnarray*}
0 & =& \int_ {\R^d} [   -h^2 \Delta + |x|^2 - 1 , h^{-1} Op_h(a)   ] \Phi_h \overline{\Phi_h}
\\ & =& \frac{1}{i} \int_{\R^d} \lbrace -h^2 \Delta + |x|^2 - 1 , Op_h(a) \rbrace \Phi_h \overline{\Phi_h} + h    \int_ {\R^d} Op_h(R ) \Phi_h \overline{\Phi_h}.
 \end{eqnarray*}
Thus, we deduce that for any function $ a \in C^\infty_0 ( \R^d) $,
\begin{equation} \label{2support mesure invariant}
\int_{\R^d \times \R^d}  \left( \xi \partial_x a - x \partial_\xi a      \right)   d\mu(x,\xi) = 0.
\end{equation}
Let $ (x,\xi ) \in \R^d\times\R^d $  and set for $ t \in \R $, 
\begin{equation*}
\left\{
    \begin{aligned}
        & x(t) = x \cos(t) + \xi \sin(t) ,  \\ 
      & \xi(t) = \xi \cos(t) - x \sin(t),
    \end{aligned}
\right.
\end{equation*}
{\it i.e.}
\begin{equation*}
  \left\{
      \begin{aligned}
       &  \dot{x}(t)= \xi (t) \mbox{ with } x(0)=x,\\
       &  \dot{\xi}(t)=-x(t) \mbox{ with } \xi(0)= \xi. \\
      \end{aligned}
    \right.
\end{equation*}
Using  \eqref{2support mesure invariant}, we get that for all $ t \in \R $ and $ a \in C_0^\infty(\R^d\times \R^d) $,
\begin{equation*}
\int_{\R^d \times \R^d}  a( x \cos(t) + \xi \sin(t) ,\xi \cos(t) - x \sin(t)) d\mu(x,\xi)  =\int_{\R^d\times \R^d}  a(x,\xi) d\mu(x,\xi).
\end{equation*}
Therefore if $ (x_0,\xi_0) \in Supp( \mu ) $ then for all  $ r > 0 $, there exists $ a \in C^\infty_0 \big( B ( (x_0,\xi_0) ,r )\big) $ such that for all $ t \in \R $,
\begin{equation*}
\int_{\R^d\times\R^d}  a\big( x \cos(t) + \xi \sin(t) ,\xi \cos(t) - x \sin(t)\big)    d\mu(x,\xi)  \neq 0. 
\end{equation*}
However
\begin{align*}
 Supp & \big( a( x \cos(t) + \xi \sin(t) ,\xi \cos(t) - x \sin(t))     \big) 
 \\ & \subset \left\{  (x,\xi) \in \R^d \times  \R^d /  ( x \cos(t) + \xi \sin(t)  , \xi \cos(t) - x \sin(t )  ) \in B ( (x_0,\xi_0) ,r ) \right\} 
 \\ & \subset  B(  x_0\cos(t)  -  \xi_0\sin(t)  ,  2 r )   \times B( x_0 \sin(t)  + \xi_0\cos(t)   ,  2 r  ),
 \end{align*}
and therefore, for all  $ t\in \R$,  $ (x_0  \cos(t) - \xi_0\sin(t)  , x_0\sin(t)  + \xi_0\cos(t)   )  \in Supp( \mu ) $. \medskip

 But for $ \xi_0 = 0$, $x_0^2 = 1 $ then  $x_0 \sin(t)  +\xi_0 \cos(t)  = x_0 \sin(t)  = 0 $  is impossible and therefore the estimate \eqref{2propre2} is proved by contradiction.  The proof of \eqref{3propre2} is analogous.
\end{proof}

Recall the Paley-Zygmund inequality: Let $ X $ be a random variable in $ L^2( \Omega ) $, then for all $ 0 \leq \lambda \leq 1 $,
\begin{equation}\label{zyg}
P\big( X \geq \lambda E(X) \big) \geq (1-\lambda)^2 \frac{E(X)^2}{E(X^2)}.
\end{equation}

For a function $ \chi \in C_0^\infty (\R^d)$ such that  $ 0 \leq \chi \leq 1 $, with $\chi(x)=1 $ if $ |x| \leq 1 $ and $ \chi(x) = 0  $ if $ |x| \geq 2 $, we  define
\begin{align*}
& \sigma_N^2 = \sum_{n \in \N} \chi^2  \left( \frac{ \lambda_n^2 }{N^2} \right) |c_n|^2 \lambda_n^{2s}  \underset{ N \rightarrow + \infty }{\longrightarrow} + \infty,
\\ & S_N = \| \chi \left(  \frac{H}{N^2}  \right) u_0  \|_{H^s(\R^d)},
\\ & M = \displaystyle{  \sup_{N \in \N ^\star} } S_N.
\end{align*}
Let us turn to the proof of Theorem \ref{2sobolev2}.

 \begin{lem} \label{1-kolmogorov} For all $ s \geq 0 $, we have 
\begin{equation*}
P  \Big( \omega \in \Omega \;: \sup_{ N \in \N^\star}   \| \chi \Big( \frac{H}{N^2}  \Big) u_0^\omega \|_{H^s(\R^d)}  = + \infty  \Big)  \in \big\{0,1 \big\}.
\end{equation*}
\end{lem}
\begin{proof}  The random initial condition can be written $ u_0^\omega = \displaystyle{  \sum _ {j \geq 0 } }  X_j(\omega) $, where   $ (X_j)_{ j \in \N} $ is a sequence of independent random variables satisfying $ X_j(\omega) \in H^s( \R^d )$, $\omega-$almost surely. Therefore, we have
\begin{equation*}
  \sup_{ N \in \N^\star} \  \Big\| \chi \Big( \frac{H}{N^2}  \Big) u_0^\omega \Big\|_{H^s(\R^d)} = + \infty  
  \end{equation*}
if and only if for all $ K \in \N $, 
   \begin{equation*}
       \sup_{ N \in \N^\star}   \Big\|  \chi \Big( \frac{H}{N^2}  \Big) \Big( \sum_{j  \geq K }  X_j(\omega) \Big)   \Big\|_{H^s(\R^d)} = + \infty,
  \end{equation*}
so if we put $ F_j = \sigma ( X_j , X_{j+1} , \cdots  ) $ we have that 
\begin{equation*}
\bigg\{   \omega \in \Omega \;: \sup_{N \in \N^\star}  \| \chi \Big( \frac{H}{N^2}  \Big) u^\omega_0 \|_{H^s(\R^d)}   = + \infty \bigg\} \in \underset{K \in \N^* }{ \bigcap }  F_K.
\end{equation*}
As a consequence, the set  $ \Big\{   \omega \in \Omega \;: \displaystyle{ \sup_{N \in \N^\star} }   \| \chi \Big( \frac{H}{N^2}  \Big) u^\omega_0 \|_{H^s(\R^d)}  = + \infty  \Big\} $ is an element of the   asymptotic $\sigma$-algebra and the lemma is proved by the 0-1 law. 
\end{proof}

 \begin{lem} \label{1-divergence} For $ s \geq 0 $, if $ \underset{n  \in \N }{ \sum  }    |c_n|^2 \lambda_n^{2s} = +   \infty $ then
\begin{equation*}
P \Big( \omega \in \Omega \;:  \sup_{N \in \N^\star} \| \chi \Big( \frac{H}{N^2}  \Big) u^\omega_0 \|_{H^s(\R^d)}  = + \infty   \Big) =1.
\end{equation*}
\end{lem}

\begin{proof}  Recall the notation $M = \displaystyle{  \sup_{N \in \N ^\star} } S_N$. Thus by Lemma~\ref{1-kolmogorov} it is  sufficient to establish that 
\begin{equation}\label{posi}
P( M = + \infty ) > 0.
\end{equation}
Thanks to \eqref{2propre2}, we get 
\begin{multline*}
E(S^2_N)=  E \Big(  \| \chi \big( \frac{H}{N^2}  \big)  u_0 \|^2_{H^s(\R^d)} \Big)   \\
\begin{aligned}
&\geq  E \left( \sum_{n,m} \chi \left( \frac{\lambda_n^2}{N^2} \right) \chi \left( \frac{\lambda_m^2}{N^2} \right) c_n \overline{c_m} \ g_n( \omega ) \overline{g_m ( \omega )} \ \int_{\R^d} (-\Delta)^s ( h_n )  (-\Delta)^s( h_m ) dx   \right) \\
&  =  E \left( \sum_{n \in \N } \chi^2 \left( \frac{\lambda_n^2}{N^2} \right) | c_n |^2 | g_n ( \omega )|^2  \| (-\Delta)^s ( h_n )  \|^2_{L^2(\R^d)}   \right) \\
&  \geq   C_1 \sigma_N^2.
\end{aligned}
\end{multline*}
Therefore, thanks to the Zygmund inequality \eqref{zyg} with $X=S^2_N$,  we establish
\begin{eqnarray}
P \Big( M^2 \geq \frac{C_1 \sigma_N^2}{2} \Big)  \geq P \Big( S_N^2 \geq  \frac{ C_1  \sigma_N^2}{2} \Big) &  \geq&  P \Big( S_N^2 \geq  \frac{E(S^2_N) }{2} \Big) \nonumber\\ 
&  \geq & \frac{1}{4}   \frac{E \big( \| \chi \left( \frac{H}{N^2}  \right)  u_0 \|^2_{H^s(\R^d)} \big) ^2}{E  \big( \| \chi \left( \frac{H}{N^2}  \right)  u_0 \|^4_{H^s(\R^d)}\big) } \label{minor}.
\end{eqnarray}
Then, thanks to \eqref{2propre2}, we have
\begin{multline*}
E \Big( \| \chi \big( \frac{H}{N^2}  \big)  u_0 \|^4_{H^s(\R^d)}  \Big)  \leq \\
\begin{aligned}
&  \leq    E \Big( \sum_{n, m}   \chi \big( \frac{\lambda_n^2}{N^2} \big) \chi \big( \frac{\lambda_m^2}{N^2} \big) c_n \overline{c_m} \ g_n( \omega ) \overline{g_m ( \omega )}  \int_{\R^d} (-\Delta)^s ( h_n )  (-\Delta)^s( h_m ) dx  \Big)^2  \\
   & \hspace*{6cm} + E \Big( \sum_{n,m}   \chi \big( \frac{\lambda_n^2}{N^2} \big) \chi \big( \frac{\lambda_m^2}{N^2} \big) c_n \overline{c_m} \ g_n( \omega ) \overline{g_m ( \omega )} \Big)^2   \\
&  \leq   C E \Big(  \sum_{n_1,n_2,n_3,n_4 } \chi \big( \frac{\lambda_{n_1}^2}{N^2} \big) \chi \big( \frac{\lambda_{n_2}^2}{N^2} \big) \chi \big( \frac{\lambda_{n_3}^2}{N^2} \big) \chi \big( \frac{\lambda_{n_4}^2}{N^2} \big)   c_{n_1} \overline{c_{n_2}} c_{n_3} \overline{c_{n_4}} \cdot 
\\  & \hspace*{.2cm}    \cdot\| (-\Delta)^s( h_{n_1} ) \|_{L^2( \R^d )} \| (-\Delta)^s( h_{n_2} ) \|_{L^2( \R^d )}\| (-\Delta)^s( h_{n_3} ) \|_{L^2( \R^d )}\| (-\Delta)^s( h_{n_4} ) \|_{L^2( \R^d )} \Big)    \\
  &  \hspace*{6cm} + C E \Big( \sum_{ n}   \chi^2 \big( \frac{\lambda_n^2}{N^2} \big) | c_n |^2  \Big)^2   \\
& \leq   C_2 \sigma_N ^4.
\end{aligned}
\end{multline*}
Therefore from \eqref{minor} we deduce $\dis P \Big( M^2 \geq \frac{C_1 \sigma_N^2}{2} \Big) \geq \frac{1}{4}   \frac{C_1^2}{C_2}$.
Finally, using the monotone convergence theorem, we obtain $\dis P \left( M = + \infty  \right) \geq \frac{1}{4}   \frac{C_1^2}{C_2}$, which implies \eqref{posi}.
\end{proof}

We are now able to complete the proof of Theorem~\ref{2sobolev2}. Indeed, let us  prove it by contradiction and assume that  $ P\big( \omega \in \Omega: \ u_0^\omega \in H^s( \R^d) \big) > 0 $.  According to Proposition~\ref{1-continuous},
 \begin{equation*}
\forall u \in H^s( \R^d) , \qquad    \sup_{ N \in \N^\star }   \| \chi \Big(  \frac{H}{N^2}  \Big)u \|_{H^s(  \R^d )} \leq C \|u\|_{H^s(\R^d)},
\end{equation*}
and from the latter inequality we deduce that 
$$P \Big( \omega \in  \Omega \;: \ \displaystyle{\sup_{ N \in \N^\star }}   \| \chi \Big(  \frac{H}{N^2}  \Big)u_0^\omega \|_{H^s(  \R^d )} < + \infty  \Big) > 0.$$
 Therefore
\begin{equation*}
P \Big( \omega \in \Omega \;:  \sup_{N \in \N^\star} \| \chi \Big( \frac{H}{N^2}  \Big) u^\omega_0 \|_{H^s(\R^d)}  = + \infty   \Big) <1 
\end{equation*}
which contradicts Proposition~\ref{1-divergence}.

As a consequence, we have proven that if $u_0  \notin \mathcal{H}^s(\R^d)$ then  
$u^\omega_0 \notin H^s(\R^d)$. The fact that  $\langle x\rangle ^s u^\omega_0 \notin L^2(\R^d)$ is obtained similarly, and this concludes the proof of Theorem~\ref{2sobolev2}. \medskip

To finish this part, we evaluate the Sobolev norm of the initial data. This will establish that Theorem \ref{1-thm1}  will hold true for supercritical equations with large initial data. 

\begin{prop}\label{prop54} Let $ \sigma \geq 0, \  u_0 \in \mathcal{H}^\sigma(\R^3) $ and $ s \geq  \sigma $. Assume that for all  $ n \in \N $,
\begin{equation*}
  \lambda_n^{2s} |c_n|^2 \leq 1 
\end{equation*}
then for all $ t \geq 0 $,
\begin{equation*}
\mu \Big( u \;: \| \chi \Big(   \frac{H}{N^2}    \Big)  u \|_{  \mathcal{H}^s(\R^3) } \leq t   \Big)  \leq  
       e^{ t ^2 -   \frac{1}{2} \| \chi \big(   \frac{H}{N^2}    \big)  u_0 \|^2_{\mathcal{H}^s(\R^3) } }  
\end{equation*}
\end{prop}

\begin{proof}  Using that $ -\ln(1+u) \leq - \frac{u}{2} $ for all $ u \in [0,1] $ and the Markov inequality, we obtain
\begin{multline*}
\mu \Big(  u \in \mathcal{H}^\sigma(\R^3) \;: \| \chi \big(   \frac{H}{N^2}    \big)  u \|_{   \mathcal{H}^s(\R^3) } \leq t   \Big)  = \\
\begin{aligned} 
&= P \Big(   \omega \in \Omega \;:  e^{-\| \chi \big(   \frac{H}{N^2}    \big)  u^\omega_0 \|^2_{  \mathcal{H}^s(\R^3) } }  \geq e^{-t^2 }   \Big)\\
 & \leq   e^{t^2} E \Big( e^{-\| \chi \big(   \frac{H}{N^2}    \big)  u^\omega_0 \|^2_{  \mathcal{H}^s(\R^3) } }  \Big) \\
 & \leq   e^{t^2} \prod_{ n \in \N   } E \Big( e^{- \chi^2 \big(   \frac{\lambda_n^2}{N^2}    \big)  \lambda_n^{2s} | c_n|^2 | X |^2 }   \Big)\\
 & \leq  e^{t^2} \prod_{ n \in \N   } \Big( \frac{1}{1+\chi^2 (   \frac{\lambda_n^2}{N^2}    )  \lambda_n^{2s} | c_n|^2} \Big) \\
 & \leq   e^{t^2} \prod_{ n \in \N   }  \Big( e^{  - \frac{1}{2}\chi^2 \big(   \frac{\lambda_n^2}{N^2}    \big)  \lambda_n^{2s} | c_n|^2 } \Big)\\
 & \leq   e^{ t^2 - \frac{1}{2}  \| \chi \big(   \frac{H}{N^2}    \big)  u_0 \|^2_{\mathcal{H}^s(\R^3) }  }. 
\end{aligned}
\end{multline*}
\end{proof}
\textit{Remarks :} $(i)$ For example, if $ u_0 \notin \mathcal{H}^s(\R^3) $ and $ \lambda_n^{2s}  |c_n|^2 \leq 1, \ \forall n \in \N $, we obtain that for all  $ t \geq 0 $, 
\begin{equation*}
\lim_{N \rightarrow + \infty} \mu \Big( u_0 \in \mathcal{H}^\sigma(\R^3) \;: \big\| \chi \big(   \frac{H}{N^2}    \big)  u_0 \big\|_{  \mathcal{H}^s(\R^d) } \leq t   \Big) = 0 .
\end{equation*}
This does mean that the Sobolev norm of the initial data is not small.

$(ii)$  For example, for $ \epsilon \ll 1 $, we can choose $ c_n = \frac{\epsilon}{ \lambda_n^s  }  \frac{1}{ \sqrt{n}  }   $ and obtain for $ t \geq 0$ ,
\begin{equation*}
\mu \Big( u_0 \in \mathcal{H}^\sigma(\R^3) \;: \big\| \chi \big(   \frac{H}{N^2}    \big)  u_0 \big\|_{  \mathcal{H}^s(\R^d) } \leq t    \Big)  \leq \exp  \left( t^2 - C' \epsilon^2 \ln^2 N  \right) \underset{ N \rightarrow + \infty }{ \longrightarrow  } 0  ,
 \end{equation*}
but $ \|u_0\|_{ \mathcal{H}^\sigma ( \R^3  )   } = C '' \epsilon \ll 1 $.

  %%%%%%%%%%%%%%%%%%%%%%%%%%%%%%%%%%%%%%%%%%%%%%%%%%%%%%%%%%%%%%%%%%%%%%%%%%%%%%

\section{The fixed point argument in dimension $d=3$, global existence, and scattering}\label{Sect6}

Let us introduce the following equation: 
\begin{equation} \label{1-schrodingerH}  
  \left\{
      \begin{aligned}
     &   i \frac{ \partial u }{ \partial t } - H u = \kappa \cos (2 t) |u|^2 u,   \quad (t,x) \in \R \times \R^3    \\
 & u(0,x)=u_0(x).
      \end{aligned}
    \right.
\end{equation}
We will show that this equation is locally well posed. Then, thanks to the lens transform we will be able to show that \eqref{1-schrodinger} is globally well-posed.

\subsection{Some nonlinear analysis}

In this part, we establish the estimates which will be used to apply a fixed point theorem.

 \begin{lem} \label{1-easy case} 
There exists $ b' < \frac{1}{2} $ such that for all $ \delta > 0 $ and $ K \geq 1 $, there exists  a constant $ C_K>0 $ such that if   $ N_1 \geq N_2^{1+\delta} $ and $ N_2 \geq N_3 \geq N_4 $ then for all $ u_1, u_2, u_3, u_4 \in \X^{0,b'}$,
\begin{equation*}
 \bigg| \int_{\R \times\R^3} \Delta_{N_1}(u_1) \Delta_{N_2}(u_2) \Delta_{N_3}(u_3) \Delta_{N_4}(u_4) \bigg|  \leq C_K  N_1^{-K} \prod_{j=1}^4 \| \Delta_{N_j}(u_j)\|_{ \X^{0,b'} }.
\end{equation*}
\end{lem}

\begin{proof} We begin with the case $ u_j(t,x) = c_j (t) h_{n_j}(x) $. We have
\begin{multline*}
 \bigg| \int_{\R \times\R^3} \Delta_{N_1}(u_1) \Delta_{N_2}(u_2)\Delta_{N_3}(u_3) \Delta_{N_4}(u_4) \bigg|   =\\
\begin{aligned}
&= \bigg| \int_{\R \times\R^3} \prod_{j=1}^4 \psi( \frac{ \lambda_{n_j}^2 }{N_j^2} ) c_j (t) h_{n_j}(x) dt dx \bigg| \\
&\leq  \prod_{j=1}^4   \psi( \frac{ \lambda_{n_j}^2 }{N_j^2} )   \int_\R | c_1(t) \cdots c_4 (t) | dt  \bigg| \int_{\R^3}  h_{n_1}(x)  \cdots  h_{n_4}(x) dx \bigg|
\end{aligned}
\end{multline*}
Then by Lemma~\ref{1-fast} and Proposition~\ref{1-bourgain2}, we deduce
\begin{multline}\label{star}
 \bigg| \int_{\R \times\R^3} \Delta_{N_1}(u_1) \Delta_{N_2}(u_2)\Delta_{N_3}(u_3) \Delta_{N_4}(u_4) \bigg|   \leq \\
\begin{aligned}
 &\leq  C_K N_1^{-K}   \prod_{j=1}^4   \psi( \frac{ \lambda_{n_j}^2 }{N_j^2} )    \prod_{j=1}^4 \| c_j\|_{ L_t^4(\R) } \\
  &\leq  C_K  N_1^{-K} \prod_{j=1}^4 \| \Delta_{N_j}( u_j) \|_{ L^4( \R ; L^2(\R^3)) }\\
 &\leq  C_K  N_1^{-K} \prod_{j=1}^4 \| \Delta_{N_j}(u_j)\|_{ \X^{0,b'} }.
\end{aligned}
\end{multline}
For the general case, let $\dis  u_j (t,x) = \underset{ k\geq 0 }{\sum} c_{j,k}(t) h_k(x) $, then by \eqref{star}
\begin{multline*}
\bigg| \int_{\R \times\R^3} \Delta_{N_1}(u_1) \Delta_{N_2}(u_2)\Delta_{N_3}(u_3) \Delta_{N_4}(u_4) \bigg|   \leq \\
\begin{aligned}
 &\leq  \sum_{k_1,k_2,k_3,k_4\geq 0} \bigg| \int_{\R \times\R^3} \Delta_{N_1}(c_{1,k_1} h_{k_1}) \Delta_{N_2}(c_{2,k_2} h_{k_2})\Delta_{N_3}(c_{3,k_3} h_{k_3}) \Delta_{N_4}(c_{4,k_4} h_{k_4}) \bigg| \\
  & \leq   C_K  N_1^{-K}   \sum_{k_1,k_2,k_3,k_4\geq 0} \prod_{j=1}^4 \| \Delta_{N_j}(c_{j,k_j} h_{k_j})\|_{ \X^{0,b'} } \\
 & \leq   C_K  N_1^{-K+12}  \sqrt{  \sum_{k_1,k_2,k_3,k_4\geq 0} \prod_{j=1}^4 \| \Delta_{N_j}(c_{j,k_j} h_{k_j})\|^2_{ \X^{0,b'} } } .
\end{aligned}
\end{multline*}
Then we have
\begin{eqnarray*}
\sum_{k_j\geq 0} \| \Delta_{N_j}(c_{j,k_j} h_{k_j})\|^2_{ \X^{0,b'} } & = &\sum_{k_j\geq 0} \sum_{n\geq 0} \| \langle  \tau + \lambda_n \rangle  ^{b'} \widehat{P_n(c_{j,k_j} h_{k_j}  ) } \|^2_{ L_t^2( \R ; L^2_x(\R^3) ) }\\
  & =&  \sum_{n\geq 0} \| \langle  \tau + \lambda_n \rangle  ^{b'} \widehat{P_n(u_j  ) } \|^2_{ L_\tau^2( \R ; L^2_x(\R^3) ) } \\
 & = & \| u_j \|_{ \X^{0,b'} }^2, 
\end{eqnarray*}
which concludes the proof.
\end{proof}

\begin{prop} There exists $ b' < \frac{1}{2} $ such that for all $ b > \frac{1}{2} $ and $ s > \frac{1}{2} $, there exist  two constants $ C>0$ and $\kappa >0 $ such that for any $ v \in \X^{s,b} $ and all $ N_3  \leq N_2 \leq N_1 $,
\begin{equation} \label{1-estim1}
\| \Delta_{N_1}(v) \Delta _{N_2}(v) \Delta_{N_3} (v) \|_{\X^{s,-b'}} \leq C   N_1^{-\kappa} \|v\|_{\X^{s,b}}^3.
\end{equation}
\end{prop}

\begin{proof}  By duality, it is sufficient to show that there exists a constant $\delta > 0 $ such that  
\begin{equation*}
\int_{\R \times\R^3} \Delta_{N_1}(v) \Delta _{N_2}(v) \Delta_{N_3} (v) \Delta_M(w) \leq C  N_1^{-\kappa} M ^{-\delta} \|v\|_{\X^{s,b}}^3 \| w \|_{ \X^{-s,b'} }.
\end{equation*}
Thanks to Lemma~\ref{1-easy case}, we only need to treat  the case where $ M \leq N_1 ^ {1+ \delta } $.

$\bullet$ Case $N_3 \leq M \leq N_1^{1+\delta}$. Using \eqref{1-bilibourgain1}, we obtain
\begin{multline*}
 \int_{\R \times\R^3} \Delta_{N_1}(v) \Delta _{N_2}(v) \Delta_{N_3} (v) \Delta_M(w) \leq \\
\begin{aligned}
& \leq  \| \Delta_{N_1}(v) \Delta _{N_2}(v) \|_ {L^2(\R; L^2(\R^3) ) }  \| \Delta_{M}(w) \Delta _{N_3}(v) \|_ {L^2(\R ; L^2(\R^3)) }\\
& \leq  (N_2 N_3)^{1/2+\delta}  (\frac{N_2}{N_1})^{1/2-\delta}  (\frac{N_3}{M})^{1/2-\delta}   \| \Delta_{N_1}(v) \|_{ \X^{0,b} } \| \Delta_{N_2}(v) \|_{ \X^{0,b} } \| \Delta_{N_3}(v) \|_{ \X^{0,b} } \| \Delta_{M}(w) \|_{ \X^{0,b'} }\\
  & \leq  (\frac{N_2}{N_1})^{1/2-\delta} (\frac{N_3}{M})^{1/2-\delta} (N_2 N_3)^{1/2+ \delta-s} (   \frac{M}{N_1} )^{s}  \| v \|_{ \X^{s,b} }^3 \| w \|_{ \X^{-s,b'} }\\ 
  & \leq  (N_2 N_3)^{1-s}  M ^{-1/2+\delta} N_1^{-1/2+(1+s)\delta}  \| v \|_{ \X^{s,b} }^3 \| w \|_{ \X^{-s,b'} }\\
  & \leq M^{-\delta} M^{1/2-s+(1+s)\delta}   N_1^{1/2-s+2\delta}  \| v \|_{ \X^{s,b} }^3 \| w \|_{ \X^{-s,b'} }\\
 & \leq M^{-\delta} N_1^{1-2s+ (3+s)\delta}  \| v \|_{ \X^{s,b} }^3 \| w \|_{ \X^{-s,b'} }.
\end{aligned}
\end{multline*}

$\bullet$ Case $ M \leq N_3 \leq N_1^{1+\delta} $. Using   \eqref{1-bilibourgain1}, we get
\begin{multline*}
\int_{\R \times\R^3} \Delta_{N_1}(v) \Delta _{N_2}(v) \Delta_{N_3} (v) \Delta_M(w)  \leq \\
\begin{aligned}
& \leq  \| \Delta_{N_1}(v) \Delta _{N_2}(v) \|_ {L^2(\R ; L^2(\R^3) ) }   \| \Delta_{M}(w) \Delta _{N_3}(v) \|_ {L^2(\R ; L^2(\R^3) ) }
\\ & \leq  (N_2 M )^{1/2+\delta}  (\frac{N_2}{N_1})^{1/2-\delta} (\frac{M}{N_3})^{1/2-\delta}  \| \Delta_{N_1}(v) \|_{ \X^{0,b} } \| \Delta_{N_2}(v) \|_{ \X^{0,b} } \| \Delta_{N_3}(v) \|_{ \X^{0,b} } \| \Delta_{M}(w) \|_{ \X^{0,b'} }
\\ & \leq  N_2 M (\frac{1}{N_3N_1})^{1/2-\delta}  (   \frac{M}{N_1N_2 N_3} )^{s}  \| v \|_{ \X^{s,b} }^3 \| w \|_{ \X^{-s,b'} }
\\ & \leq  M^{-\delta}  N_2^{1-s} N_3^{1/2+2 \delta}  (   \frac{1}{N_1} )^{1/2+s-\delta}    \| v \|_{ \X^{s,b} }^3 \| w \|_{ \X^{-s,b'} }
\\  & \leq   M^{-\delta}  N_1^{1-2s+3 \delta}     \| v \|_{ \X^{s,b} }^3 \| w \|_{ \X^{-s,b'} },
\end{aligned}
\end{multline*}
which concludes the proof.
\end{proof}

 \begin{prop} There exists $ b' < \frac{1}{2} $ such that for all $ b > \frac{1}{2} $ and $ s > \frac{1}{2} $, there exist two constants $ C,\kappa >0 $ such that if for some $ \lambda>0 $, we have for all $ N$,
\begin{equation*}
\| \Delta_N( e^{itH} u_0 ) \|_{ L^2 (  [- \pi,  \pi] ;  L^\infty ( \R^3 )) } \leq \lambda N^{-1/6}
\end{equation*}
then for all $ v \in \X^{s,b} $ and all  $ N_3 \leq N_2 \leq N_1 $,
\begin{equation} \label{1-estim2}
\| \Delta_{N_1}(v) \Delta _{N_2}(v) \Delta_{N_3} \big(\psi(t) e^{-itH}u_0\big) \|_{\X^{s,-b'}} \leq C N_1^{-\kappa}  ( \|v\|_{\X^{s,b}}^3 + \lambda ^3 ),
\end{equation}
\begin{equation} \label{1-estim3}
\| \Delta_{N_1}(v) \Delta _{N_2}\big(\psi(t) e^{-itH}u_0\big) \Delta_{N_3} (v) \|_{\X^{s,-b'}} \leq C N_1^{-\kappa}  ( \|v\|_{\X^{s,b}}^3 + \lambda ^3 ).
\end{equation}
\end{prop}
\begin{proof} We show \eqref{1-estim2}, the proof of \eqref{1-estim3} being similar. By duality, it is sufficient to show that there exists a constant $ \delta > 0 $ such that 
\begin{multline*}
\int_{\R \times\R^3}\Delta_{N_1}(v) \Delta _{N_2}(v) \Delta_{N_3} \big(\psi(t) e^{-itH}u_0\big) \Delta_M(w) \leq \\
\leq C N_1^{-\kappa} M^{-\delta}  ( \|v\|_{\X^{s,b}}^3 + \lambda ^3 )  \|w \|_{  \X^{-s,b'} }.
\end{multline*}
Thanks to Lemma~\ref{1-easy case}, we only need to treat the case where $ M \leq N_1^{1+\delta} $. \medskip

$\bullet$ Case $N_2 \leq M \leq N_1^{1+\delta}$.  Using  \eqref{1-bilibourgain1} and Proposition~\ref{1-bourgain1}, we get
\begin{multline*}
 \int_{\R \times\R^3}\Delta_{N_1}(v) \Delta _{N_2}(v) \Delta_{N_3} \big(\psi(t) e^{-itH}u_0\big) \Delta_M(w) \leq \\
 \begin{aligned}
& \leq  \|\Delta_{N_2}(v) \Delta _{M}(w)  \|_{L^2( \R; L^2(\R^3))}    \|\Delta_{N_3} \big(\psi(t) e^{-itH}u_0\big) \|_{L^2( [-\pi, \pi ] ;   L^\infty ( \R^3 )) }   \|\Delta_{N_1}(v)  \|_{L^\infty( [-\pi, \pi ];  L^2(\R^3))}\\
& \leq  N_2^{1/2+\delta} (\frac{N_2}{M})^{1/2-\delta}      \|  \Delta_{N_1}(v) \|_{  \X^{0,b} } \| \Delta _{N_2}(v) \|_{  \X^{0,b} }  \|\Delta_{N_3} ( e^{itH}u_0) \|_{L^2( [-\pi, \pi] ; L^\infty(\R^3))} \|\Delta_{M}(w)  \|_{  \X^{0,b'} } \\
& \leq   N_2^{1/2+\delta} (\frac{N_2}{M})^{1/2-\delta} (\frac{M}{N_1N_2})^{s}   N_3^{-1/6}   \lambda \|  v \|^2_{  \X^{s,b} }    \|w  \|_{ \X^{-s,b'} } \\
& \leq   N_2^{1-s}  M^{s-1/2+\delta}  N_1^{-s}  N_3^{-1/6}  \lambda \|  v \|^2_{  \X^{s,b} }   \|w  \|_{ \X^{-s,b'} } \\
& \leq   M^{-\delta}   N_2^{1-s}  N_1^{-1/2+(3+s) \delta}  N_3^{-1/6}  \lambda \|  v \|^2_{  \X^{s,b} }   \|w  \|_{ \X^{-s,b'} } \\
& \leq    N_1^{1/2-s+(3+s)\delta}  \lambda \|  v \|^2_{  \X^{s,b} }   \|w  \|_{ \X^{-s,b'} }.
 \end{aligned}
\end{multline*}

$\bullet$ Case $ M \leq N_2 \leq N_1^{1+\delta} $. Using \eqref{1-bilibourgain1} and Proposition~\ref{1-bourgain1}, we get
\begin{multline*}
\int_{\R \times\R^3}\Delta_{N_1}(v) \Delta _{N_2}(v) \Delta_{N_3} \big(\psi(t) e^{-itH}u_0\big) \Delta_M(w) \leq  \\
\begin{aligned}
& \leq  \|\Delta_{N_2}(v) \Delta _{M}(w)  \|_{L^2( \R; L^2(\R^3))}   \|\Delta_{N_3} \big(\psi(t) e^{-itH}u_0\big) \|_{L^2( [-\pi, \pi ] ;   L^\infty ( \R^3 )) }    \|\Delta_{N_1}(v)  \|_{L^\infty( [-\pi, \pi ]; L^2(\R^3))} \\
& \leq  M^{1/2+\delta} (\frac{M}{N_2})^{1/2-\delta}       \|  \Delta_{N_1}(v) \|_{  \X^{0,b} } \| \Delta _{N_2}(v) \|_{  \X^{0,b} }  \|\Delta_{N_3} ( e^{itH}u_0) \|_{L^2( [-\pi, \pi] ; L^\infty(\R^3))} \|\Delta_{M}(w)  \|_{  \X^{0,b'} } \\
& \leq  M^{1/2+\delta}  (\frac{M}{N_2})^{1/2-\delta} (\frac{M}{N_1N_2})^{s}   N_3^{-1/6}  \lambda \|  v \|^2_{  \X^{s,b} }    \|w  \|_{ \X^{-s,b'} } \\
& \leq   N_2^{-1/2-s+\delta} M^{1+s} N_1^{-s} N_3^{-1/6}  \lambda \|  v \|^2_{  \X^{s,b} }   \|w  \|_{ \X^{-s,b'} } \\
& \leq   N_2^{1/2+\delta}  N_1^{-s} N_3^{-1/6}  \lambda \|  v \|^2_{  \X^{s,b} }   \|w  \|_{ \X^{-s,b'} } \\
& \leq    M^{-\delta }   N_1^{1/2-s+2\delta}  \lambda \|  v \|^2_{  \X^{s,b} }   \|w  \|_{ \X^{-s,b'} },
\end{aligned}
\end{multline*}
whiv
\end{proof}

  \begin{prop} There exists $ b' < \frac{1}{2} $ such that for all $ b > \frac{1}{2} $ and $ s > \frac{1}{2} $, there exists a constant $ C >0 $ such that if for some $ \lambda > 0 $, we have for all $ N $,
\begin{equation*}
\| \Delta_N( e^{itH} u_0 )  \|_{ L^4 (  [-\pi, \pi] ; L^\infty ( \R^3 )) } \leq \lambda N^{-1/6}\quad \mbox{and} \quad \|  \, [ e^{itH} u_0 ]^2 \,  \|_{ L^4 (  [-\pi, \pi] ; \mathcal{H}^s ( \R^3 )) } \leq \lambda ^2 
\end{equation*}
then for all $ v \in \X^{s,b} $,
\begin{equation} \label{1-estim4}
\| v \big(\psi(t) e^{-itH}u_0\big) \big(  \psi(t) e^{-itH}u_0 \big)\|_{\X^{s,-b'}} \leq C  ( \|v\|_{\X^{s,b}}^3 + \lambda ^3 ).
\end{equation}
\end{prop}

\begin{proof} Using  Proposition~\ref{1-bourgain3} and Proposition~\ref{1-bourgain1}, we get
\begin{multline*}
\| v \big(\psi(t) e^{-itH}u_0\big) \big(  \psi(t) e^{-itH}u_0 \big)\|_{\X^{s,-b'}}   \leq \\
\begin{aligned}
  & \leq \| v \big(\psi(t) e^{-itH}u_0\big) \big(  \psi(t) e^{-itH}u_0 \big)\|_{L^{1+\delta}( \R ; \mathcal{H}^{s}(\R^3))} 
\\& \leq \| v \big(\psi(t) e^{-itH}u_0\big) \big(  \psi(t) e^{-itH}u_0 \big)\|_{L^{1+\delta}( [-\pi,\pi] ; \mathcal{H}^{s}(\R^3))} 
\\  & \leq  \|v\|_{L^\infty ( [-\pi, \pi ];  \mathcal{H}^s(\R^3))}  \|  e^{itH} u_0 \|^2_{L^4 ( [-\pi, \pi] ; L^\infty(\R^3))} +\\
&\hspace{5cm} +\|v\|_{L^2 ( [-\pi, \pi ] ; L^\infty(\R^3))}  \|  \, [ e^{itH} u_0 ] ^2 \,  \|_{L^4 ( [-\pi, \pi] ;  \mathcal{H}^s(\R^3))} 
\\  & \leq \lambda ^2 \|v\|_{ \overline{ X } ^{s,b}} +  \lambda^2 \|v\|_{L^2 ( \R ;   \W^{s,6} (\R^3))}  
\\  & \leq  C  ( \|v\|_{\X^{s,b}}^3 + \lambda ^3 ),
\end{aligned}
\end{multline*}
hence the result.
\end{proof}

 \begin{prop} 
There exists $ b' < \frac{1}{2} $ such that for all $ b > \frac{1}{2} $ and $ s > \frac{1}{2} $, there exists  a constant $ C >0 $ such that if for some $ \lambda > 0 $, we have 
\begin{equation*}
\| \, [ e^{itH} u_0 ]^3 \,  \|_{L^4 ([-\pi , \pi]; \mathcal{H} ^{s} (\R^3)) } \leq \lambda^3 
\end{equation*}
then for all $ v \in \overline{ X} ^{s,b}$,
\begin{equation} \label{1-estim5}
\big\| \big( \psi(t) e^{-itH}u_0\big) \big(  \psi(t) e^{-itH}u_0\big)  \big(\psi(t) e^{-itH}u_0\big)  \big\|_{\X^{s,-b'}} \leq C \lambda ^3.
\end{equation}
\end{prop} 

\begin{proof} Using  Proposition~\ref{1-bourgain3}, we get 
\begin{multline*}
\big\| \big( \psi(t) e^{-itH}u_0\big) \big(  \psi(t) e^{-itH}u_0\big)  \big(\psi(t) e^{-itH}u_0\big)  \big\|_{\X^{s,-b'}}  \leq  \\
\begin{aligned}
& \leq  \big\| \big( \psi(t) e^{-itH}u_0\big) \big(  \psi(t) e^{-itH}u_0\big)  \big(\psi(t) e^{-itH}u_0\big)  \big\|_{L^{1+\delta} ( \R ; \mathcal{H}^{s}(\R^3))}
\\ & \leq C \| \, [ e^{itH}u_0 ]^3  \,  \|_{L^4 ( [-\pi , \pi ] ; \mathcal{H} ^{s}(\R^3))}
\\  & \leq C \lambda ^3, 
\end{aligned}
\end{multline*}
hence the result.
\end{proof}
 
 \begin{prop} \label{1-Rfixe}
There exists $b' < \frac{1}{2}$ such that for all $ b > \frac{1}{2} $ and $ s \in  ] \frac{1}{2} ,1 [ $, there exist  $ C>0, \kappa >0 $ and $ R \in [ 2,+ \infty[ $ such that if for some $ \lambda > 0 $, we have for all $N\geq 1$,
\begin{equation*}
\left\{
    \begin{aligned} 
      &  \|  u_0 \|_{L^2(\R^3) } \leq \lambda, \\
   &    \| \Delta_N(e^{itH} u_0 ) \|_{L^4( [-\pi,\pi] ; L^\infty (\R^3)) } \leq \lambda  N^{-1/6}, \\
    &   \| \Delta_N( e^{itH} u_0 ) \|_{L^R( [-\pi,\pi] ; \W^{s,4} (\R^3)) } \leq \lambda  N^{s-1/4},
    \end{aligned}
\right.
\end{equation*}

then for all $ v \in \overline{ X}^{s,b}$ and all $ N_3 \leq N_2 \leq N_1 $,
\begin{equation} \label{1-estim6} 
\| \Delta_{N_1}\big(\psi(t) e^{-itH}u_0\big) \Delta _{N_2}(v) \Delta_{N_3} (v) \|_{\X^{s,-b'}} \leq C N_1^{-\kappa}  (  \lambda ^3  + \| v \|_{  \overline{ X } ^{s,b}}^3  ).
\end{equation}
\end{prop}

\begin{proof} Let $ \delta > 0  $ be small enough, to be fixed later. 
\\\textbf{Case $  N_1 \geq (N_2 N_3 )^{  \frac{1-s}{1-s-4 \delta} } $.} By duality, it is sufficient to establish
\begin{equation*}
\int_{\R\times \R^3} \Delta_{N_1}\big(\psi(t) e^{itH}u_0\big) \Delta _{N_2}(v) \Delta_{N_3} (v) \Delta_M(w) \leq C  N_1^{-\kappa} M ^{-\delta}   \|w\|_{\X^{-s,b'}}  ( \lambda ^3 +  \| v \|^3_{ \X^{s,b} } ).
\end{equation*}
Thanks to Lemma~\ref{1-easy case}, we only need to treat the case where $ M \leq N_1^{1+\delta} $.
\\If $ N_3 \leq M $ then using \eqref{1-bilibourgain1} and \eqref{1-bilibourgain2}, we obtain
\begin{multline*}
 \int_{\R \times\R^3} \Delta_{N_1}\big(\psi(t) e^{-itH}u_0\big) \Delta _{N_2}(v) \Delta_{N_3} (v) \Delta_M(w)  \leq  \\
 \begin{aligned}
& \leq   \| \Delta_{N_1}\big(\psi(t) e^{-itH}u_0\big) \Delta _{N_2}(v) \|_{L^2( \R ; L ^2(\R^3)) }  \|  \Delta_{N_3} (v) \Delta_M(w) \|_{L^2( \R ; L ^2(\R^3)) }
\\  & \leq  (N_2 N_3)^{1/2+\delta} (\frac{N_2}{N_1})^{1/2-\delta} (\frac{N_3}{M})^{1/2-\delta} \| \Delta_{N_1}(u_0) \|_{L^2} \| \Delta_{N_2}(v)\|_{ \overline{ X}^{0,b}} \| \Delta_{N_3}(v)\|_{ \overline{ X}^{0,b}} \| \Delta_{M}(w)\|_{ \overline{ X}^{0,b'}}
\\  & \leq (N_2 N_3)^{1/2+\delta} (\frac{N_2}{N_1})^{1/2-\delta} (\frac{N_3}{M})^{1/2-\delta} (\frac{M}{N_2N_3})^s  \lambda \| \Delta_{N_2}(v)\|_{ \overline{ X}^{s,b}} \| \Delta_{N_3}(v)\|_{ \overline{ X}^{s,b}} \| \Delta_{M}(w)\|_{ \overline{ X}^{-s,b'}}
\\  & \leq  M ^{-\delta} N_1^{-\delta} N_1^{-1+s+4\delta} (N_2 N_3)^{1-s}   \lambda   \| v\|^2_{ \overline{ X}^{s,b}} \| w \|_{ \overline{ X}^{-s,b'}}
\\  & \leq  M ^{-\delta} N_1^{-\delta}   \lambda   \| v\|^2_{ \overline{ X}^{s,b}} \| w \|_{ \overline{ X}^{-s,b'}}.
\end{aligned}
\end{multline*}
Then, if $ N_3 \geq M $, using \eqref{1-bilibourgain1} and \eqref{1-bilibourgain2}, we obtain
\begin{multline*}
 \int_{\R \times\R^3} \Delta_{N_1}\big(\psi(t) e^{-itH}u_0\big) \Delta _{N_2}(v) \Delta_{N_3} (v) \Delta_M(w)  \leq \\
 \begin{aligned}
& \leq \| \Delta_{N_1}\big(\psi(t) e^{-itH}u_0\big) \Delta _{N_2}(v) \|_{L^2( \R ; L ^2(\R^3)) }  \|  \Delta_{N_3} (v) \Delta_M(w) \|_{L^2( \R ; L ^2(\R^3)) } \\
  &\leq   (N_2 M)^{1/2+\delta} (\frac{N_2}{N_1})^{1/2-\delta} (\frac{M}{N_3})^{1/2-\delta}  \| \Delta_{N_1}(u_0) \|_{L^2} \| \Delta_{N_2}(v)\|_{ \overline{ X}^{0,b}} \| \Delta_{N_3}(v)\|_{ \overline{ X}^{0,b}} \| \Delta_{M}(w)\|_{ \overline{ X}^{0,b'}}
\\ &  \leq    (N_2 M)^{1/2+\delta} (\frac{N_2}{N_1})^{1/2-\delta} (\frac{M}{N_3})^{1/2-\delta} (\frac{M}{N_2N_3})^s   \lambda \| \Delta_{N_2}(v)\|_{ \overline{ X}^{s,b}} \| \Delta_{N_3}(v)\|_{ \overline{ X}^{s,b}} \| \Delta_{M}(w)\|_{ \overline{ X}^{-s,b'}}
\\  &\leq   M^{-\delta} N_1^{-\delta} N_1^{-1+s+4\delta} (N_2 N_3)^{1-s}   \lambda   \| v\|^2_{ \overline{ X}^{s,b}} \| w \|_{ \overline{ X}^{-s,b'}}
\\&  \leq   M^{-\delta} N_1^{-\delta}   \lambda   \| v\|^2_{ \overline{ X}^{s,b}} \| w \|_{ \overline{ X}^{-s,b'}}.
\end{aligned}
\end{multline*}

\textbf{Case $  N_1 \leq (N_2 N_3 )^{  \frac{1-s}{1-s-4 \delta} } $.} Using Proposition~\ref{1-bourgain3} and Proposition~\ref{1-bourgain1}, we establish
\begin{multline*}
 \| \Delta_{N_1}\big(\psi(t) e^{-itH}u_0\big) \Delta _{N_2}(v) \Delta_{N_3} (v) \|_{  \X^{s,-b'} } \leq \\
\begin{aligned}
& \leq    \| \Delta_{N_1}\big(\psi(t) e^{-itH}u_0\big) \Delta _{N_2}(v) \Delta_{N_3} (v) \|_{  L^{1+\delta}( \R ; \mathcal{H}^s(\R ^3 ) )  }
\\  & \leq   \| \Delta_{N_1}(e^{-itH}u_0) \Delta _{N_2}(v) \Delta_{N_3} (v) \|_{  L^{1+\delta}( [-\pi,\pi] ; \mathcal{H}^s(\R ^3 ) )  }
\\  & \leq  \| \Delta_{N_1}(e^{itH}u_0) \|_{  L^\frac{(1+\delta)(1+2\delta)}{\delta} ( [-\pi, \pi ]  ;  \W^{s,4}(\R^3)) }    \prod_{j=1}^2  \| \Delta _{N_j}(v) \|_{  L^{2(1+2\delta)}( \R ; L^8(\R ^3 ) )  } 
\\ & \hspace*{.3cm} +  \| \Delta_{N_1}(e^{itH}u_0) \|_{  L^4 ( [-\pi, \pi ]  ; L^\infty (\R^3)) }   \| \Delta _{N_2}(v) \|_{  L^\infty( [-\pi, \pi ] ;  \mathcal{H}^s (\R ^3 ) )  } \| \Delta_{N_3} (v) \|_{  L^2( [-\pi, \pi ] ;   \W^{s,6} (\R ^3 ) )  }
\\ & \hspace*{.3cm} +   \| \Delta_{N_1}(e^{itH}u_0) \|_{  L^4 ( [-\pi, \pi ]  ; L^\infty (\R^3)) }   \| \Delta_{N_2} (v) \|_{  L^2( [-\pi, \pi ] ;  \W^{s,6} (\R ^3 ) )  }\| \Delta _{N_3}(v) \|_{  L^\infty( [-\pi, \pi ] ;  \mathcal{H}^s (\R ^3 ) )  }
\\  & \leq  N_1^{s-1/4}  (N_2N_3)^{ \frac{9}{8} - \frac{1}{1+2\delta} -s  } \lambda \|v\|^2_{  \X^{s,b} } +  N_1^{-1/6}   \lambda \|v\|^2_{  \X^{s,b} }
\\ & \leq  (N_2 N_3 )^{  \frac{(1-s)(s-1/4+\delta)}{1-s-4 \delta} } (N_2N_3)^{ \frac{9}{8} - \frac{1}{1+\delta} -s  } N_1^{-\delta}  \lambda \|v\|^2_{  \X^{s,b} } +  N_1^{-1/6} \lambda \|v\|^2_{  \X^{s,b} },
\end{aligned}
\end{multline*}
with 
\begin{eqnarray*}
 \frac{(1-s)(s-1/4+\delta)}{1-s-4 \delta} + \frac{9}{8} - \frac{1}{1+2\delta} -s  = \hspace{5cm}\\
\begin{aligned}
&= s- \frac{1}{4}+  \delta + \frac{4 \delta (s- \frac{1}{4} + \delta )}{1-s-4 \delta} + \frac{9}{8} - \frac{1}{1+2\delta} -s \\
 & = \frac{7}{8}+ \frac{4 \delta (s- \frac{1}{4} + \delta  )}{1-s-4 \delta}  - \frac{1}{1+2\delta}\\
 & = - \frac{1}{8} + o(\delta) < 0.
 \end{aligned}
 \end{eqnarray*}
And finally, the proposition is proved with $ R = \frac{(1+\delta)(1+2\delta)}{\delta} $.
\end{proof}

 \subsection{Local well-posedness of equation (6.1)}
 
Let $ \lambda > 0 $ and define $ E_0( \lambda ) $ as the set of functions $ u_0 \in L^2(\R^3)$ which satisfy
\begin{equation}\label{defE}
\left\{
    \begin{aligned}
 &     \|u_0\|_{ L^2(\R^3) } \leq \lambda  \\ 
  &     \| \, [e^{itH} u_0]^2 \, \|_{  L^4( [-\pi,\pi] ; \mathcal{H}^s (\R^3)) } \leq  \lambda ^2  \\
  &      \| \, [e^{itH} u_0]^3 \, \|_{ L^4( [-\pi,\pi] ; \mathcal{H}^s (\R^3))  } \leq \lambda ^3 \\
      &  \| \Delta_N(e^{itH} u_0 ) \|_{L^4( [-\pi,\pi] ; L^\infty (\R^3)) } \leq \lambda  N^{-1/6} \ , \  \forall N \geq 1 \\
       &      \| \Delta_N( e^{itH} u_0 ) \|_{L^R( [-\pi,\pi] ; \W^{s,4} (\R^3)) } \leq \lambda  N^{s-1/4}   \ , \ \forall N  \geq 1 
    \end{aligned} 
\right.
\end{equation}
where $R>2$ is fixed by Proposition~\ref{1-Rfixe}.

\begin{prop} Let $ \frac{1}{2} < s < 1 $  then there exist   $ C > 0 $ and  $ b > 1/2 $ such that if $ u_0 \in E_0(\lambda ) $ with $ \lambda > 0 $, then for all $ v \in \X^{s,b} $
 \begin{multline*}
  \Big\|  \psi(t)  \int_0^t   e^{-i(t-s)H} \psi(s)  \cos ( 2s ) |  \psi(s) e^{-isH} u_0 + v |^2  (   \psi(s) e^{-isH} u_0 + v ) ds  \Big\|_{  \X^{s,b} } \leq  \\
 \leq C  ( \lambda ^3 +  \|v\|^3_{  \X^{s,b} } ).
\end{multline*}  
\end{prop} 
\begin{proof} For all $b>\frac{1}{2}$, using Proposition~\ref{1-bourgain4} and Proposition~\ref{1-bourgain5}, we obtain
\begin{multline*}
  \big\|  \psi(t)  \int_0^t   e^{-i(t-s)H}  \cos ( 2s ) \psi(s) |  \psi(s) e^{-isH} u_0 + v |^2  (   \psi(s) e^{-isH} u_0 + v ) ds   \big\|_{  \X^{s,b} }  
 \leq\\ 
 \begin{aligned}
 & \leq   C  \|  \cos ( 2s ) \psi(s)  |  \psi(s) e^{-isH} u_0 + v |^2  (   \psi(s) e^{-isH} u_0 + v )   \|_{  \X^{s,b-1} }  
\\  &\leq   C  \|  |  \psi(s) e^{-isH} u_0 + v |^2  (   \psi(s) e^{-isH} u_0 + v )   \|_{  \X^{s,b-1} }.
 \end{aligned}
\end{multline*}
Then using \eqref{1-estim1}, \eqref{1-estim2}, \eqref{1-estim3}, \eqref{1-estim4}, \eqref{1-estim5}, \eqref{1-estim6}, we establish the existence of an integer $ b' < \frac{1}{2} $ such that for all    $ u_0 \in E_0(\lambda) $,
$$
\big\|   |  \psi(s) e^{-isH} u_0 + v |^2  (   \psi(s) e^{-isH} u_0 + v )   \big\|_{  \X^{s,-b'} }  \leq C ( \lambda^3 +   \|v\|^3_{  \X^{s,b} }   ).
$$
It is then sufficient to choose $ b = 1 - b' > \frac{1}{2} $ and the proposition is proved.  
\end{proof}

\begin{prop} \label{1-fp}
Let $ \frac{1}{2} < s < 1 $  then there exists a constant $ C > 0 $ and $ b > 1/2 $ such that if $ u_0 \in E_0(\lambda ) $ with $ \lambda > 0 $, then for all $ v \in \X_T^{s,b} $ 
\begin{multline*}
\Big\|  \psi(t)   \int_0^t   e^{-i(t-s)H}  \cos (2s) \psi(s)   \big|  \psi(s) e^{-isH} u_0 + v \big|^2  (   \psi(s) e^{-isH} u_0 + v ) ds     \Big\|_{  \X_T^{s,b} }   
\leq \\
\leq  C ( \lambda ^3 +  \|v\|^3_{  \X_T^{s,b} } ).
\end{multline*} 
\end{prop} 
\begin{proof} Let $ w \in \X^{s,b} $  such that  $ w |_{[-T,T]}=v $ then
\begin{multline*}
 \big\|  \psi(t)   \int_0^t   e^{-i(t-s)H}  \cos (2s) \psi(s) \big|  \psi(s) e^{-isH} u_0 + v \big|^2  (   \psi(s) e^{-isH} u_0 + v ) ds   \big\|_{  \X_T^{s,b} }   \leq \\
 \begin{aligned}
& \leq  \|  \psi(t)  \int_0^t   e^{-i(t-s)H} \cos (2s) \psi(s) \big|  \psi(s) e^{-isH} u_0 + w \big|^2  (   \psi(s) e^{-isH} u_0 + w ) ds   \|_{  \X^{s,b} } 
\\ & \leq  C ( \lambda ^3 +  \|w\|^3_{  \X^{s,b} } )
\end{aligned}
\end{multline*}
for all  $w\in \X^{s,b} $.
\end{proof}

In a similar way, one could prove the following result:

 \begin{prop} \label{2-fp}
Let $ \frac{1}{2} < s < 1 $   then there exists a constant $ C > 0 $ and a real $ b > 1/2 $ such that if $ u_0 \in E_0(\lambda ) $ with  $ \lambda > 0 $, then for all  $ v_1,v_2 \in \X_T^{s,b} $
 \begin{align*}
&  \Big\|  \psi(t)   \int_0^t   e^{-i(t-s)H}    \cos (2s) \psi(s) \big|  \psi(s) e^{-isH} u_0 + v_1 \big|^2  (   \psi(s) e^{-isH} u_0 + v_1 ) ds  
\\ & \hspace*{0.1cm} - \psi(t)   \int_0^t   e^{-i(t-s)H}  \cos (2s) \psi(s) \big|  \psi(s) e^{-isH} u_0 + v_2 \big|^2  (   \psi(s) e^{-isH} u_0 + v_2 ) ds    \Big\|_{  \X_T^{s,b} }  
\\  \leq  C  & \|v_1-v_2\|_{ \X^{s,b}_T }  ( \lambda ^2 +  \|v_1\|^2_{  \X_T^{s,b} } + \|v_2\|^2_{  \X_T^{s,b} } ).
\end{align*} 
\end{prop}

\begin{thm}  \label{1-existencebis}
Let $ \frac{1}{2} < s < 1 $ then there exists a constant $ C > 0 $ and a real $ b > 1/2 $ such that if  $ u_0 \in E_0(\lambda ) $ with  $ \lambda <   \frac{1}{ 2  \sqrt{C} } $ then there is a unique solution to the equation~\eqref{1-schrodingerH} with initial data $ u_0 $ in the space $ e^{-itH}u_0 + B_{\X^{s,b}_T } (0, \frac{1}{2} \sqrt{  \frac{1}{C} } ) $.
\end{thm}

\begin{proof} We define
\begin{multline*}
L(v):= \\
  -i \kappa\psi(t)  \int_0^t   e^{-i(t-s)H}  \cos (2s) \psi(s)  \big|  \psi(s) e^{-isH} u_0 + v(s) \big|^2 ( \psi(s) e^{-isH} u_0 + v(s) ) ds,
\end{multline*}
$ u = e^{-itH} u_0 + v $ is the unique solution to \eqref{1-schrodingerH} in the space  $ e^{-itH}u_0 + B_{\X^{s,b}_T}(0 , R ) $ if and only if $ v $ is the unique fixed point of $ L $ in space $ B_{\X^{s,b}_T}(0 , R ) $. According to Proposition~\ref{1-fp} and Proposition~\ref{2-fp}, there exists a constant $ C > 0 $ such that 
\begin{eqnarray*}
 \| L (v)  \|_{  \overline{ X } ^{s,b}_T  } &\leq& C ( \lambda ^3 + \| v\|^3 _{ \overline{ X } ^{s,b}_T }), \\
  \| L (v_1)-L(v_2)  \|_{  \overline{ X } ^{s,b}_T  } &\leq& C \| v_1-v_2\| _{ \overline{ X } ^{s,b}_T } ( \lambda ^2 + \| v_1\|^2 _{ \overline{ X } ^{s,b}_T }+ \| v_2\|^2 _{ \overline{ X } ^{s,b}_T } ).
\end{eqnarray*}
Thus, if $ \lambda < \frac{1}{2 \sqrt{ C} } $ then $ L $ is a contraction of the complete space $ B_{\X^{s,b}_T } (0, \frac{1}{2} \sqrt{  \frac{1}{C} } ) $ and therefore has a unique fixed point.
\end{proof}

 %%%%%%%%%%%%%%%%%%%%%%%%%%%%%%%%%%%%%%%%%%%%%%%%%%%%%%%%%%%%%%%%%%%%%%%%%%%%%%

\subsection{Global solutions and scattering for the equation \eqref{1-schrodinger}}\label{1-Sect6}
Thanks to the lens transform and the result of Theorem~\ref{1-existencebis}, we are now able to establish the existence of global solutions for the equation \eqref{1-schrodinger}. We also prove the uniqueness of the solution and then show that this solution scatters at $t \longrightarrow \pm \infty$. Recall that the set $E_0(\lambda)$ is defined in~\eqref{defE}.

  \begin{thm}  \label{1-existence}
If $ \frac{1}{2} < s < 1 $ then there exist two constants $ C,c > 0 $ such that if $ u_0 \in E_0(\lambda ) $ with $ \lambda < \frac{1}{2 \sqrt{C} } $ then there exists a unique global solution to the equation~\eqref{1-schrodinger} with initial data $ u_0 $ in the space $ e^{it\Delta}u_0 + B_{X^s } (0, \frac{c}{2} \sqrt{  \frac{1}{C} } ) $.
\end{thm}

\begin{proof} Let $ u $ be given by Theorem~\ref{1-existencebis} and define
\begin{equation*}
 U (t,x) = \Big( \frac{1}{\sqrt{1+4t^2}} \Big)  ^{3/2}  u \left( \frac{\arctan(2t)}{2}  , \frac{x}{\sqrt{1+4t^2} } \right)    e^{ \frac{i|x|^2t}{1+4t^2}  }.
\end{equation*}
According to Section~\ref{1-214}, since $ u $ is a solution of~\eqref{1-schrodingerH} on $ [-\pi/4, \pi/4] $ then $ U $ is a global solution of~\eqref{1-schrodinger}.

Thus, to obtain the theorem, it suffices to notice that
\begin{equation*}
(e^{it\Delta} u_0)(t,x) = \left( \frac{1}{\sqrt{1+4t^2}} \right) ^{3/2}  ( e^{-itH}  u_0 ) \left( \frac{ \arctan(2t)}{2} , \frac{x}{\sqrt{1+4t^2} } \right)   e^{ \frac{i|x|^2t}{1+4t^2}  },
\end{equation*}
and to use Proposition~\ref{tvb}. 
\end{proof}

 The existence of the solutions being proved, one with analogous estimates, one can show that the solutions are unique.

We then prove that the constructed solutions scatter for  $ t \longrightarrow + \infty $ and $ t \longrightarrow - \infty $.  

  \begin{thm} \label{1-scattering}
Let $ U $ be the unique solution of \eqref{1-schrodinger} constructed in Theorem~\ref{1-existence}, then there exist $ L_+, L_- \in \mathcal{H}^s(\R^3)$ such that 
\begin{equation}\label{scat1}
\begin{aligned}
& \lim_{t \rightarrow + \infty } \| U(t) -e^{it\Delta} \big(u_0 +   L_+\big) \|_{H^s(\R^3)} = 0, 
\\ & \lim_{t \rightarrow - \infty } \| U(t) -e^{it\Delta} \big(u_0 +   L_-\big) \|_{H^s(\R^3)} = 0,
\end{aligned}
\end{equation}
and
\begin{equation}\label{scat2}
\begin{aligned}
& \lim_{t \rightarrow + \infty } \|e^{-it\Delta}  U(t) -\big(u_0 +   L_+\big) \|_{\mathcal{H}^s(\R^3)} = 0, 
\\ & \lim_{t \rightarrow - \infty } \| e^{-it\Delta}  U(t) -  \big(u_0 +   L_-\big) \|_{\mathcal{H}^s(\R^3)} = 0.
\end{aligned}
\end{equation}
\end{thm}

\begin{proof}  Let us treat the case $t \longrightarrow+\infty$. In the following, we set $T=\frac{\pi}4$. We have shown that
\begin{equation*}
-i \psi(t)   \int_0^t   e^{-i(t-s)H} \cos (2s) \psi(s)   |  \psi(s) e^{-isH} u_0 + v(s) |^2 ( \psi(s) e^{-isH} u_0 + v(s) ) ds \in \X^{s,b}_T.
\end{equation*}
hence by Lemma~\ref{1-bourgain7},
\begin{multline*}
- i e^{-itH} \int_0^t   e^{isH}  \cos (2s) \psi(s) |  \psi(s) e^{-isH} u_0 + v(s) |^2 ( \psi(s) e^{-isH} u_0 + v(s) ) ds \\
 \in C ^0 ( [-T,T]  ; \mathcal{H}^s(\R^3)).
\end{multline*}
And therefore, there exists a function $ L \in \mathcal{H}^s(\R^3) $ such that 
\begin{align*}
& \lim_{ t \rightarrow T } \Big\| L- i \kappa e^{-itH} \int_0^t   e^{isH}  \cos (2s) \psi(s)  |  \psi(s) e^{-isH} u_0 + v(s) |^2 ( \psi(s) e^{-isH} u_0 + v(s) ) ds  \Big\|_{ \mathcal{H}^s(\R^3) } 
\\  = & \lim_{ t \rightarrow T } \Big\|  e^{itH} L - i  \kappa \int_0^t   e^{isH}\cos (2s) \psi(s)  |  \psi(s) e^{-isH} u_0 + v(s) |^2 ( \psi(s) e^{-isH} u_0 + v(s) ) ds   \Big\|_{ \mathcal{H}^s(\R^3) } 
\\  = &\lim_{ t \rightarrow T }  \Big\|   e^{iTH} L  - i  \kappa\int_0^t   e^{isH} \cos (2s) \psi(s)  |  \psi(s) e^{-isH} u_0 + v(s) |^2 ( \psi(s) e^{-isH} u_0 + v(s) ) ds   \Big\|_{ \mathcal{H}^s(\R^3)} 
\\ & \hspace*{4cm} = 0 .
\end{align*}
But for $ t \in [-T,T]$,
\begin{equation*}
 u(t) =  e^{-itH} u_0  -i  \kappa e^{-itH} \int_0^t   e^{isH}  \cos (2s) \psi(s)  |  \psi(s) e^{-isH} u_0 + v(s) |^2 ( \psi(s) e^{-isH} u_0 + v(s) ) ds.
\end{equation*}
So, by \eqref{sym2}, we obtain
\begin{multline*}
U(t)  =  e^{it\Delta} u_0 +\\
\begin{aligned}
  \quad +  e^{it\Delta}  \Big[-i  \kappa \int_0^{ \frac{\arctan (2t)}{2}   }   e^{isH} \cos (2s) \psi(s)  |  \psi(s) e^{-isH} u_0 + v(s) |^2 ( \psi(s) e^{-isH} u_0 + v(s) ) ds \Big] \\
:=  e^{it\Delta} u_0 +  e^{it\Delta}  F( t ).
\end{aligned}
\end{multline*}
Then 
\begin{equation*}
\lim_{t \rightarrow + \infty} \| F(t)- L_+  \|_{ \mathcal{H}^s(\R^3)  } = 0,
\end{equation*}
with $L_+ =  e^{iTH} L \in \mathcal{H}^s(\R^3)$, which proves \eqref{scat2}. Then we have
\begin{eqnarray*}
\lim_{ t \rightarrow + \infty } \| e^{it\Delta}  F( t ) - e^{it \Delta} L_+ \|_{  H^s(\R^3) } & =& \lim_{ t \rightarrow + \infty } \|  F( t ) - L_+ \|_{  H^s(\R^3) } \\
 & \leq &C \lim_{ t \rightarrow + \infty } \|  F( t ) - L_+ \|_{  \mathcal{H}^s(\R^3) } = 0,
 \end{eqnarray*}
 hence \eqref{scat1} follows.
 \end{proof}
 
  %%%%%%%%%%%%%%%%%%%%%%%%%%%%%%%%%%%%%%%%%%%%%%%%%%%%%%%%%%%%%%%%%%%%%%%%%%%%%%

\section{Estimation of the regularity of the random initial data and proof of Theorem~\ref{1-thm1}}\label{Sect7}

\subsection{Estimation of the regularity of the random initial data} 
In this section, we estimate the regularity of the random data by proving large deviation type estimates. In particular, we establish that $ u^\omega_0 \in \underset{ \lambda \geq 1}{ \bigcup } E_0( \lambda) $, for almost any $ \omega \in \Omega$ (recall that $E_0(\lambda)$ is defined in~\eqref{defE}). 
 
 For $ \Lambda > 0 $ we define,
 \begin{equation}\label{defomega}
\begin{aligned}
\Omega_\Lambda =   \bigg\{ \omega \in \Omega  \ : \ & \left\{  \,  \|u^\omega_0\|_{ \mathcal{H}^\sigma(\R^3) } \leq \Lambda \right\}  \bigcap  \left\{   \,  \| \, [e^{itH} u^\omega_0]^2 \, \|_{  L^4( [-\pi,\pi] ;  \mathcal{H}^s (\R^3)) } \leq  \Lambda ^2 \right\} 
\\ &\bigcap  \left\{   \,  \| \, [e^{itH} u^\omega_0]^3 \, \|_{ L^4( [-\pi,\pi] ;  \mathcal{H}^s (\R^3))  } \leq \Lambda^3\right\}
\\ &     \underset{N \; dyadic}{\bigcap}    \left\{   \,  \| \Delta_N(e^{itH} u^\omega_0 ) \|_{L^4( [-\pi,\pi] ;  L^\infty (\R^3)) } \leq \Lambda  N^{-1/6}  \right\}
\\ &  \underset{N \; dyadic}{\bigcap}  \left\{   \,  \| \Delta_N( e^{itH} u^\omega_0 ) \|_{L^R( [-\pi,\pi] ;  \W^{s,4} (\R^3)) } \leq \Lambda  N^{s-1/4} \right\}  \bigg\}
\end{aligned}
 \end{equation}
and the aim of this part is to prove the following theorem:

\begin{thm} \label{1-deviation}
There are two constants $ C, c > 0$ such that for any $ \Lambda \geq  0 $  and all  $ u_0 \in \mathcal{H}^{\sigma} (\R^3) $, 
\begin{equation} \label{inGauss} 
P( \Omega_\Lambda^ c ) \leq C \exp \big( -  \frac{c \Lambda^2}{ \|u_0\|^2_{ \mathcal{H}^{\sigma} (\R^3)  }   }  \big).
\end{equation}
\end{thm}
We start by establishing Wiener chaos type inequalities for complex Gaussian random variables.

\begin{prop}\label{prop-chaos}
 Assume that  $ g_n \sim \mathcal{N}_\C(0,1) $  are independent random variables, then there exists a constant $ C> 0 $ such that for all $ q \geq 2 $ and all sequences $ (c_n)_{n} \in \ell^2(\N)$, $(c_{n,m})_{n,m} \in \ell^2(\N \times \N) $ and $ (c_{n,m,k})_{n,m,k} \in \ell^2(\N \times \N \times \N ) $, 
\begin{align}  
& \label{zyg1}   \Big\|  \sum_{ n \in \N   }  c_n  g_n(\omega)   \Big\|_{L^q(\Omega)} \leq C   q^{ \frac{1}{2} }   \sqrt{ \sum_{ n \in \N } |c_n |^2  },
\\ & \label{zyg2}  \Big\|  \ \sum_{ n , m \in \N  }  c_{n,m}   g_n(\omega)   g_m(\omega)     \Big\| _{L^q(\Omega)} \leq C    q   \sqrt{ \sum_{n,m \in \N} |c_{n,m} |^2 }     ,
\\ & \label{zyg3}  \Big\| \sum_{ n,m,k \in \N  }  c_{n,m,k }  g_n(\omega)    g_m(\omega)   g_k(\omega)     \Big\| _{L^q(\Omega)} \leq C    q^{ \frac{3}{2}  }     \sqrt{ \sum_{n,m,k \in \N} |c_{n,m,k} |^2 }.
\end{align}
\end{prop}

\begin{proof} The bound \eqref{zyg1} is the Khintchine inequality, and we refer to  \cite{burq4} for a proof.\medskip

Since the random variables  $ g_n$ are independent complex Gaussians, by \cite[Proposition 2.4]{thomann1}  (Wiener chaos estimates) there exists a constant $ C > 0 $ such that for any $ q \geq 2 $,
\begin{equation} \label{chaos} 
 \Big\|  \sum_{ n,m \in \N   }  c_{n,m}   g_n(\omega)   g_m(\omega)    \Big\|_{L^q(\Omega)} \leq C    q    \Big\|   \sum_{ n,m \in \N   }  c_{n,m}   g_n(\omega)   g_m(\omega)    \Big\|_{L^2(\Omega)}.
\end{equation}
Next,
\begin{multline*}
  \Big\|   \sum_{ n,m \in \N   }  c_{n,m}   g_n(\omega)   g_m(\omega)     \Big\| ^2 _{L^2(\Omega)}=\\
\begin{aligned}
&   =   \sum_{ n, n' , m , m' \in \N   }  c_{n,m}  \overline{c_{n',m'} }   E \left( g_n(\omega)   g_m(\omega)  \overline{ g_{n'}(\omega)   g_{m'}(\omega) } \right)
\\  & \leq  \sum_{ n = n' =  m = m' \in \N   } | \ | + \sum_{ n = n' , m = m' \in \N   } | \ | + \sum_{ n = m  ,   n' = m' \in \N   } | \ | + \sum_{ n = m'  ,   n' = m \in \N   }| \ | .
\end{aligned}
\end{multline*}
Then
\begin{eqnarray*}
 \sum_{ n=n'=m=m'  \in \N   } \bigg| c_{n,m}  \overline{c_{n',m'} }   E \left( g_n(\omega)   g_m(\omega)  \overline{ g_{n'}(\omega)   g_{m'}(\omega) } \right) \bigg|  &=&
   \sum_{ n \in \N   } |c_{n,n}|^2 E ( |g_n(\omega)|^4 )
\\ &\leq &   \sum_{n,m \in \N} |c_{n,m} |^2 ,
 \end{eqnarray*}
and using that  $ E( g_n(\omega)^2  ) = 0 $, we obtain
\begin{multline*}
 \sum_{ n = m , n'= m' \in \N   } \bigg| c_{n,m}  \overline{c_{n',m'} }   E \left( g_n(\omega)   g_m(\omega)  \overline{ g_{n'}(\omega)  g_{m'}(\omega) } \right) \bigg| = \\
\begin{aligned}
& =\sum_{ n,m \in \N   } | c_{n,n} |  | c_{m,m}|  \bigg|  E \left( g_n(\omega)^2   \overline{ g_m(\omega) ^2 } \right)\bigg|\\
 &=   \sum_{ n \in \N   } |c_{n,n}|^2   E \left( |g_n(\omega)|^4 \right) \\
 &\leq   \sum_{n,m \in \N} |c_{n,m} |^2,
\end{aligned}
\end{multline*}
thus \eqref{zyg2} is proved. We can proceed in the same way to obtain \eqref{zyg3} since the inequality \eqref{chaos} is true for any product of random variables. 
\end{proof}

We now prove Theorem \ref{1-deviation}. By definition \eqref{defomega} of the set $\Omega_\Lambda$ we have
\begin{align*}
P( \Omega_\Lambda^c  ) & \leq P \left( \omega \in \Omega \;:  \|u^\omega_0\|_{ \mathcal{H}^\sigma(\R^3) } \geq \Lambda \right)
\\ &\qquad\qquad  + P \left( \omega \in \Omega \;: \| \, [e^{itH} u^\omega_0]^2 \, \|_{  L^4( [-\pi,\pi] ; \mathcal{H}^s (\R^3)) } \geq  \Lambda ^2 \right)
\\ &  \qquad\qquad+  P \left( \omega \in \Omega \;: \| \, [e^{itH} u^\omega_0]^3 \, \|_{ L^4( [-\pi,\pi] ; \mathcal{H}^s (\R^3))  } \geq \Lambda ^3 \right)
\\ & \qquad\qquad+  P \left( \underset{N \; dyadic}{ \bigcup } \left\{ \omega \in \Omega \;: \| \Delta_N(e^{itH} u^\omega_0 ) \|_{L^4( [-\pi,\pi] ; L^\infty (\R^3)) } \geq \Lambda  N^{-1/6} \right\} \right)
\\ & \qquad\qquad+  P \left( \underset{N \; dyadic}{ \bigcup } \left\{ \omega \in \Omega \;:  \| \Delta_N( e^{itH} u^\omega_0 ) \|_{L^R( [-\pi,\pi] ; \W^{s,4} (\R^3)) } \geq \Lambda  N^{s-1/4}  \right\} \right)
\end{align*}
so it is enough to establish the bound \eqref{inGauss} for each of the terms. Let us carry out the proof for the second and the fourth term (for the other terms, the approach is similar). \medskip

  \textbf{ $\bullet$ Case $ \| \, [e^{itH} u^\omega_0]^2 \, \|_{  L^4( [-\pi,\pi] ;  \mathcal{H}^s (\R^3)) } \geq  \Lambda ^2 $.} Since $ s \in ]  \frac{1}{2} , \frac{1}{2} + \sigma [$, it is enough to prove the following lemma:

\begin{lem}  \label{1-lemme-X} For any $ \epsilon > 0 $, there exist two constants $ C,c > 0 $ such that for any $ \Lambda > 0 $ and any function $ u_0 \in  \mathcal{H}^{\sigma} (\R^3) $,
\begin{equation*}
P \left( \omega \in \Omega \;:    \| \, [e^{itH} u^\omega_0]^2 \, \|_{  L^4( [-\pi,\pi] ;  \mathcal{H}^{\sigma + 1 /2 -2\epsilon} (\R^3)) } \geq  \Lambda ^2  \right)  \leq C e ^{ -  \frac{c \Lambda^2}{ \|u_0\|^2_{ \mathcal{H}^{\sigma} (\R^3)  }   }  } .
\end{equation*}
\end{lem}

\begin{proof} Note that it is sufficient to prove the estimate for $ \Lambda \geq C \|u_0\|_{ \mathcal{H}^{\sigma}(\R^3)  } $. Thanks to the  Markov and Minkowski inequalities, we obtain for $ q \geq 4 $,
\begin{multline*}
 P \left( \omega \in \Omega \;:    \| \, [e^{itH} u^\omega_0]^2 \, \|_{  L^4( [-\pi,\pi] ;  \mathcal{H}^{\sigma + 1 /2 -2\epsilon} (\R^3)) } \geq  \Lambda ^2  \right)  = \\
\begin{aligned}
& =  P \left( \omega \in \Omega \;: \| \,   H^{  \sigma/2+1/4-\epsilon }  \,  [e^{itH} u^\omega_0]^2 \, \|^q_{  L^4( [-\pi,\pi]; L^2 (\R^3)) } \geq  \Lambda ^{2q} \right)
\\  & \leq   \Lambda^{-2q}  E_\omega  \left( \| \,   H^{  \sigma/2+1/4-\epsilon }  \,  [e^{itH} u^\omega_0]^2 \, \| ^q _{  L^4( [-\pi,\pi] ; L^2 (\R^3)) } \right)
\\ & = \Lambda^{-2q}   \| \,   H^{  \sigma/2+1/4-\epsilon }  \,  [e^{itH} u^\omega_0]^2 \, \| ^q _{ L ^q( \Omega ;   L^4( [-\pi,\pi] ; L^2 (\R^3))) } 
\\  & \leq   \Lambda^{-2q}   \| \,   H^{  \sigma/2+1/4-\epsilon }  \,  [e^{itH} u^\omega_0]^2 \, \| ^q _{   L^4( [-\pi,\pi] ; L^2 (\R^3; L ^q( \Omega) )) } .
\end{aligned}
\end{multline*}
Then with  \eqref{zyg2} we get
\begin{multline*}
 \| \,  H^{  \sigma/2+1/4-\epsilon }  \,  [e^{itH} u^\omega_0]^2 \, \| _{   L ^q( \Omega) }  = \\
\begin{aligned}
&=    \Big\|   \ H^{  \sigma/2+1/4-\epsilon }  \  \Big[  \sum_{n,m \in \N} e^{it(\lambda_n^2+\lambda_m^2)} c_n c_m  h_n(x) h_m(x)   g_n(\omega )   g_m(\omega) \ \Big] \  \Big\| _{   L ^q( \Omega) } 
\\ &\leq   \Big\|  \  \sum_{n,m \in \N} e^{it(\lambda_n^2+\lambda_m^2)} c_n c_m   H^{  \sigma/2+1/4-\epsilon }  \  [  h_n(x) h_m(x)  ]   g_n(\omega )   g_m(\omega)  \  \Big\|  _{   L ^q( \Omega) } 
\\ &\leq C  q  \sqrt{   \sum_{n,m \in \N } |c_n|^2   |c_m|^2 |  \,H^{  \sigma/2+1/4-\epsilon }  \  [  h_n(x) h_m(x)  ] \,   |^2     }.
\end{aligned}
\end{multline*}
 By the triangle inequality, we obtain
\begin{multline*}
  P \left( \omega \in \Omega \;: \| \, [e^{itH} u^\omega_0]^2 \, \|_{  L^4( [-\pi,\pi] ;  \mathcal{H}^{\sigma+1/2-2\epsilon} (\R^3)) } \geq  \Lambda ^2 \right) \\
  \begin{aligned}
 & \leq  \left( \frac{C q }{ \Lambda^2 }  \right) ^q    \Big\| \sum_{ n,m \in \N } |c_n|^2  |c_m|^2  |  \, H^{  \sigma/2+1/4-\epsilon }   [  h_n(x) h_m(x)  ] \,   |^2   \Big\| _{L^2( [-\pi,\pi] ; L^1 (\R^3 ) )     }  ^{q/2}
\\  & = C \left(  \frac{C q }{ \Lambda^2 } \right) ^q  \Big( \sum_{ n , m \in  \N} |c_n|^2  |c_m|^2  \| \,  H^{  \sigma/2+1/4-\epsilon }    [  h_n(x) h_m(x)  ] \,     \|^2_{L^2 (\R^3  )     } \Big) ^{q/2}.
\end{aligned}
\end{multline*}
Then thanks to Proposition~\ref{1-proper6} with $ \delta = \epsilon $, we deduce that
\begin{align*} 
\big\| \,  H^{  \sigma/2+1/4-\epsilon }  \  [  h_n(x) h_m(x)  ] \,     \big\|^2_{L^2 (\R^3 )     }  & \leq C_\epsilon  \max( \lambda_n , \lambda_m )^{2 ( \sigma - \epsilon ) } 
\\ & \leq C  \max( \lambda_n , \lambda_m )^{2  \sigma  }.
\end{align*}
And finally, we have 
\begin{multline*}
  P \left( \omega \in \Omega  \;: \big\| \,[e^{itH} u^\omega_0]^2 \,\big\|_{  L^4( [-\pi,\pi] ; \mathcal{H}^{\sigma+1/2-2\epsilon} (\R^3)) } \geq  \Lambda ^2 \right) \leq \\
    \begin{aligned}
& \leq \left(  \frac{C q }{ \Lambda^2 } \right) ^q  \left( \sum_{n,m \in \N} |c_n|^2  |c_m|^2  \max( \lambda_n , \lambda_m )^{2  \sigma  }     \right)  ^{q/2}
\\  & \leq  \left(  \frac{C q   \|u_0\|^2_{ \mathcal{H}^{\sigma}  (\R^3)} }{ \Lambda^2 }  \right) ^q.
\end{aligned}
\end{multline*}
Then it is enough to  choose $ q = \frac{t^2}{2C\|u_0\|^2_{ \mathcal{H}^{\sigma } (\R^3)  }} \geq 4 $ to conclude.
\end{proof}

\textbf{$\bullet$ Case $ \| \Delta_N[ e^{itH} u^\omega_0 ] \|_{  L^4 ( [-\pi,\pi] ;  L^\infty (\R^3)) } \geq  N^{-1/6 } \Lambda$.} Let us start by establishing the following lemma:

 \begin{lem} \label{LEFT}
  For all $ p_1,p_2 \in [ 2, + \infty [ $ and $ \epsilon > 0 $, there exist two constants $ C,c > 0 $ such that for all $ \Lambda > 0 $, $ N \geq 1 $ and $ u_0 \in \mathcal{H}^{\sigma-\epsilon}(\R^3)$,
\begin{multline*}
 P \left( \omega \in \Omega \;: \| \Delta_N(e^{itH} u^\omega_0 ) \|_{L^{p_1}( [-\pi,\pi] ; W^{\epsilon,p_2} (\R^3)) } \geq \Lambda  N^{-1/6-\sigma+2 \epsilon} \right) \leq  \\
\leq C e^{-  \frac{c \Lambda^2}{ \|  \Delta_N( u_0 ) \|^2_{ \mathcal{H}^{\sigma-\epsilon}(\R^3)  }   }    }.
\end{multline*}
\end{lem}

\begin{proof}  We replace $ u_0 $ by $ \Delta_N (u _ 0 ) $, therefore we are led to prove that  
\begin{equation*}
 P \left( \omega \in \Omega \;: \| \Delta'_N(e^{itH} u^\omega_0 ) \|_{L^{p_1}( [-\pi,\pi] ; W^{\epsilon,p_2} (\R^3)) } \geq \Lambda  N^{-1/6-\sigma+2 \epsilon} \right) \leq C e^{-  \frac{c \Lambda^2}{ \| u_0  \|^2_{ \mathcal{H}^{\sigma-\epsilon}(\R^3)  }   }    }.
\end{equation*}
As for  Lemma~\ref{1-lemme-X}, it is sufficient to prove the estimate for $ \Lambda \geq C \|u_0\|_{ \mathcal{H}^{\sigma-\epsilon}(\R^3) } $.  Thanks to the Markov and Minkowski inequalities, we obtain for $ q \geq p_1,p_2 $,
\begin{multline*}
  P \left( \omega \in \Omega  \, :\,  \| \Delta'_N(e^{itH} u^\omega_0 ) \|_{L^{p_1}( [-\pi,\pi] ; W^{\epsilon,p_2} (\R^3)) } \geq  \Lambda N^{-1/6-\sigma+2\epsilon}  \right) \leq \\
  \begin{aligned}
  &   \leq \left( \frac{N^{1/6+\sigma-2\epsilon}   E_\omega (    \| \Delta'_N(e^{itH} u^\omega_0 ) \|_{L^{p_1}( [-\pi,\pi] ; W^{\epsilon,p_2} (\R^3)) }  ) }{\Lambda} \right) ^ q\\
& \leq  \left( \frac{ N^{1/6+\sigma-2\epsilon} }{\Lambda} \right) ^q      \| H^{  \epsilon/2 }  \Delta'_N(e^{itH} u^\omega_0 ) \|_{ L^q(\Omega)  ,  L^{p_1}( [-\pi,\pi] ) ; L^{p_2} (\R^3) }   ^q\\
& \leq  \left( \frac{ N^{1/6+\sigma-2\epsilon} }{\Lambda} \right) ^q        \| H^{  \epsilon/2 }  \Delta'_N(e^{itH} u^\omega_0 ) \|_{  L^{p_1}( [-\pi,\pi] ) ; L^{p_2} (\R^3) ; L^q(\Omega) }   ^q.
\end{aligned}
\end{multline*}
But using \eqref{zyg1},
\begin{eqnarray*}
\| H^{  \epsilon/2 }  \Delta'_N(e^{itH} u^\omega_0 ) \|_{   L^q(\Omega) } &  \leq & \Big\|  \sum_{ n \in \N }\phi \big( \frac{\lambda_n^2}{N^2} \big)\lambda_n^{\epsilon}  e^{it\lambda_n^2} c_n  h_n(x) g_n( \omega )  \Big\|_{   L^q(\Omega) }    \\
& \leq & \Big\|  \sum_{ \lambda_n \sim  N } \phi   \big( \frac{\lambda_n^2}{N^2} \big) \lambda_n^{\epsilon}  e^{it\lambda_n^2} c_n h_n(x) g_n ( \omega )  \Big\| _{   L^q(\Omega) }  \\
 & \leq &C  \sqrt{q}   \sqrt{ \sum_{ \lambda_n \sim  N } \phi^2 \big( \frac{\lambda_n^2}{N^2} \big) \lambda_n^{2 \epsilon}   | c_n |^2  | h_n(x) |^2 }.
\end{eqnarray*}
Then, using \eqref{1-propre1} and \eqref{1-propre2}, we obtain
\begin{multline*}
 \| H^{  \epsilon/2 }   \Delta'_N(e^{itH} u^\omega_0 ) \|_{  L^{p_1}( [-\pi,\pi] ) ; L^{p_2} (\R^3) ; L^q(\Omega) }  \leq 
 \\ 
 \begin{aligned}
&\leq  C  \sqrt{q}  \Big\| \sum_{ \lambda_n \sim  N } \phi^2 \big( \frac{\lambda_n^2}{N^2} \big)   \lambda_n^{2 \epsilon}   | c_n |^2  | h_n(x) |^2  \Big\| ^{1/2}_{ L^{p_1/2}( [-\pi,\pi] ) ; L^{p_2/2} (\R^3) }
\\ 
&\leq  C  \sqrt{q}  \sqrt{  \sum_{ \lambda_n \sim  N } \phi^2 \big( \frac{\lambda_n^2}{N^2} \big)   \lambda_n^{2 \epsilon}  | c_n |^2     \| h_n(x)  \|^2_{ L^{p_2} (\R^3) } }
\\ &\leq C  \sqrt{q}  \sqrt{  \sum_{ \lambda_n \sim  N } \phi^2 \big( \frac{\lambda_n^2}{N^2} \big)   \lambda_n^{2 \epsilon-1/3}   | c_n |^2   }
\\  &\leq C  \sqrt{q}  N^{-\sigma - 1/6 + 2\epsilon}   \sqrt{  \sum_{ \lambda_n \sim  N } \phi^2 \big( \frac{\lambda_n^2}{N^2} \big) \lambda_n^{2 (\sigma - \epsilon )  }   | c_n |^2   }
\\ &\leq C  \sqrt{q}  N^{-\sigma - 1/6 + 2\epsilon}   \|  \Delta'_N (  u_0 ) \|_{ \mathcal{H}^{\sigma-\epsilon}(\R^3) }
\\ &\leq C  \sqrt{q}  N^{-\sigma - 1/6 + 2\epsilon}   \|   u_0  \|_{ \mathcal{H}^{\sigma-\epsilon}(\R^3) }.
\end{aligned}
\end{multline*}
Finally, we have for all $q \geq p_1,p_2 $,
\begin{multline*}
 P  \left( \omega \in \Omega \;: \| \Delta'_N(e^{itH} u^\omega_0 ) \|_{L^{p_1}( [-\pi,\pi] ; W^{\epsilon,p_2} (\R^3)) }   \geq 
 \Lambda N^{-1/6-\sigma+2\epsilon}  \right) \leq \\
 \leq  \left( \frac{C  \sqrt{q}  \|  u_0 \|_{ \mathcal{H}^{\sigma-\epsilon}(\R^3) } }{\Lambda}  \right) ^q.
\end{multline*}
It is then sufficient to choose $ q =  \left( \frac{\Lambda}{4C \|u_0\|_{ \mathcal{H}^\sigma(\R^3) }   } \right) ^2 \geq p_1,p_2 $ to prove Lemma~\ref{LEFT}.
\end{proof}

Then for $ p_2 = \frac{3}{\epsilon} + \epsilon  $ , we have
\begin{equation*}
W^{ \epsilon, p_2 } (\R^3) \hookrightarrow L^{\infty } ( \R^3).
\end{equation*}
Hence for all  $ p_1 \in [2,+ \infty[$ and $ \epsilon > 0  $, there exist two constants $ C,c >0$ such that for all $ \Lambda > 0 $, $ N\geq 1 $ and $ u_0 \in \mathcal{H}^{\sigma-\epsilon}(\R^3)   $,
\begin{equation*}
 P \left( \omega \in \Omega \;: \| \Delta_N(e^{itH} u^\omega_0 ) \|_{L^{p_1}( [-\pi,\pi] ; L^{\infty} (\R^3)) } \geq \Lambda N^{-1/6-\sigma+2\epsilon} \right) \leq C e^{-  \frac{c \Lambda^2}{ \|  \Delta_N ( u_0 ) \|^2_{ \mathcal{H}^{\sigma-\epsilon}(\R^3)  }   }    }.
\end{equation*}
Then we can choose $ p_1 = 4 $ and use that $ \| \Delta_N( u_0 ) \|^2_{  \mathcal{H}^{\sigma-\epsilon} ( \R^3 ) }  \leq N^{-2\epsilon}  \|u_0\|^2_{  \mathcal{H}^{\sigma} ( \R^3 ) } $ to obtain for any $ \epsilon > 0 $ the existence of two constants $ C,c > 0 $ such that for all  $ \Lambda > 0 $, $ N\geq 1 $ and $ u_0 \in \mathcal{H}^{\sigma}(\R^3)   $,
\begin{align} \label{1-inegalite gauss 1}
 P \left( \omega \in \Omega \;: \| \Delta_N(e^{itH} u^\omega_0 ) \|_{L^4( [-\pi,\pi] ; L^{\infty} (\R^3)) } \geq \Lambda  N^{-1/6} \right) \leq C e^{-  \frac{c N^{2(\sigma-\epsilon)} \Lambda^2}{ \| u_0  \|^2_{ \mathcal{H}^{\sigma}(\R^3)  }   }    }.
\end{align}
We have to prove that there exist two constants $ C,c > 0 $ such that for all $ \Lambda > 0 $, $ N\geq 1 $ and $ u_0 \in \mathcal{H}^{\sigma}(\R^3)   $,
\begin{multline} \label{1-inegalite gauss 2}
P  \left( \underset{ N }{ \bigcup }  \left\{ \omega \in \Omega \;: \| \Delta_N(e^{itH} u^\omega_0 ) \|_{L^4( [-\pi,\pi] ; L^\infty (\R^3)) } \geq \Lambda  N^{-1/6} \right\} \right) \leq \\
\leq C e^{-  \frac{c \Lambda^2}{ \|  u_0  \|^2_{ \mathcal{H}^{\sigma}(\R^3)  }   }    }.
\end{multline}
We can assume that $ \Lambda \geq C \|u_0\|_{  \mathcal{H}^{\sigma}(\R^3) } $, and we set  $ \alpha =  \left( \frac{\Lambda}{C\|u_0\|_{  \mathcal{H}^{\sigma}(\R^3) } }  \right) ^2 $  and choosing $ \epsilon < \sigma $ in~\eqref{1-inegalite gauss 1}, it is enough to show that 
\begin{equation*} 
\forall \delta >0, \ \exists \ C,c >0 / \ \forall \alpha \geq 1 , \ \sum_{N \; dyadic} e^{-\alpha N^\delta } \leq C e^{-c\alpha }.
\end{equation*}
Using that there exists $c>0$, such that $ck \leq 2^{\delta k}-1$
\begin{equation*}
\sum_{N \; dyadic} e^{-\alpha N^\delta }  =e^{-\alpha } \sum_{ k \geq 0}e^{-\alpha (2^{\delta k} -1)}  \leq e^{-\alpha } \sum_{ k \geq 0}e^{-c \alpha k}  \leq C   e^{-\alpha },
\end{equation*}
and \eqref{1-inegalite gauss 2} is proved. This completes the proof of Theorem~\ref{1-deviation}.

 %%%%%%%%%%%%%%%%%%%%%%%%%%%%%%%%%%%%%%%%%%%%%%%%%%%%%%%%%%%%%%%%%%%%%%%%%%%%%%

\subsection{Proof of Theorem \ref{1-thm1}}
Recall the definition \eqref{defomega} of the set $\Omega_{\Lambda}$. To show Theorem \ref{1-thm1}, it is sufficient to show that for any $ \Lambda> 0$, $P( \Omega_\Lambda ) > 0 $. To do this, we first establish that it is sufficient to show the result for a finite number of terms in the initial data.  For $ u_0 = \displaystyle{ \sum_{n \in \N} c_n h_n }   \in  L^2 ( \R^3 ) $, we define for $ {K \in \N^\star}$,
\begin{equation*}
 [u_0]_K = \sum_{\lambda_n < K } c_n h_n, \qquad   {[u_0]}^K = \sum_{\lambda_n  \geq K} c_n h_n.
\end{equation*}
 
 By independence of the random variables, we obtain
\begin{align*}
 P( \Omega_\Lambda ) \geq  & \mu \bigg(  \|\, [ u_0]^K \|_{ \mathcal{H}^\sigma(\R^3) } \leq \frac{\Lambda}{2}  , \   \| \ (e^{itH} [u_0]^K ) ^2  \|_{  L^4( [-\pi,\pi] ; \mathcal{H}^s (\R^3)) } \leq  \frac{\Lambda^2}{2}  , 
\\ & \hspace*{1cm}  \|\,(e^{itH} [u_0]^K)^3\,\|_{ L^4( [-\pi,\pi] ;  \mathcal{H}^s (\R^3))  }  \leq \frac{\Lambda^3}{2}, \
\\  & \hspace*{2cm} \underset{ N }{ \bigcap }  \left\{  \| \Delta_N(e^{itH} [u_0]^K ) \|_{L^4( [-\pi,\pi] ; L^\infty (\R^3)) } \leq \frac{\Lambda  N^{-1/6} }{ 2 } \right\} , 
\\ & \hspace*{3cm} \underset{ N }{ \bigcap } \left\{  \| \Delta_N( e^{itH} [u_0]^K ) \|_{L^R( [-\pi,\pi] ; \W^{s,4} (\R^3)) } \leq \frac{\Lambda  N^{s-1/4} }{2 } \right\} \bigg) \cdot
\\   &\cdot  \mu \bigg(  \| \, [ u_0]_K \|_{ \mathcal{H}^\sigma(\R^3) } \leq \frac{\Lambda}{2}, \  \|\,(e^{itH} [u_0]_K )^2\,\|_{  L^4( [-\pi,\pi] ; \mathcal{H}^s (\R^3)) } \leq  \frac{\Lambda^2}{2},
\\ & \hspace*{1cm}  \|\,(e^{itH} [u_0]_K)^3\,\|_{ L^4( [-\pi,\pi] ; \mathcal{H}^s (\R^3))  } \leq \frac{\Lambda^3}{2} , 
\\ & \hspace*{2cm} \underset{ N }{ \bigcap }  \left\{  \| \Delta_N(e^{itH} [u_0]_K ) \|_{L^4( [-\pi,\pi]; L^\infty (\R^3)) } \leq \frac{\Lambda  N^{-1/6} }{2} \right\},
\\ &  \hspace*{3cm} \underset{ N }{ \bigcap } \left\{  \| \Delta_N( e^{itH} [u_0]_K ) \|_{L^R( [-\pi,\pi] ; \W^{s,4} (\R^3)) } \leq \frac{\Lambda  N^{s-1/4} }{2}  \right\}  \bigg) .
\end{align*}
Let us denote by $ P_{\Lambda,K} $ the first probabilistic term of this inequality. Then by Theorem~\ref{1-deviation}, for all $ \Lambda\geq 0 $ and $ K \in \N^\star$,  
\begin{equation*}
P_{\Lambda,K} \geq 1 - C \exp\big(  - \frac{c \Lambda ^2 } {\| \, [u_0]^K \, \|^2_{ \mathcal{H}^\sigma ( \R^3 ) }} \big),
\end{equation*}
with
\begin{equation*}
 \underset{K \rightarrow + \infty }{\lim} \| \, [u_0]^K \, \|^2_{ \mathcal{H}^\sigma ( \R^3 ) } = 0.
\end{equation*}
Therefore, there is an integer $ K \geq 1  $ such that $\dis  C \exp(  - \frac{c \Lambda ^2 } {\| \, [u_0]^K \, \|^2_{ \mathcal{H}^\sigma ( \R^3 ) }} ) \leq \alpha$, and we deduce that 
\begin{multline}\label{1-deviation2}
 \frac{P( \Omega_\Lambda )}{1-\alpha} \geq \\
\begin{aligned}
&      \mu \bigg(  u_0 \in \mathcal{H}^\sigma(\R^3 ) :     \|\, [u_0]_K\|_{ \mathcal{H}^\sigma(\R^3) } \leq \frac{\Lambda}{2} , \, \| \, (e^{itH} [u_0]_K )^2 \, \|_{  L^4( [-\pi,\pi] ; \mathcal{H}^s (\R^3)) } \leq  \frac{\Lambda^2}{2} ,\\
& \hspace*{5cm}   \| \, (e^{itH} [u_0]_K)^3 \, \|_{ L^4( [-\pi,\pi] ; \mathcal{H}^s (\R^3))  } \leq \frac{\Lambda^3}{2} ,  \\
  & \hspace*{5cm} \underset{ N }{ \bigcap } \left\{   \| \Delta_N(e^{itH} [u_0]_K ) \|_{L^4( [-\pi,\pi] ; L^\infty (\R^3)) } \leq   \frac{\Lambda N^{-1/6}}{2} \right\} , \\
 & \hspace*{4cm} \left. \underset{ N }{ \bigcap } \left\{   \| \Delta_N( e^{itH} [u_0]_K ) \|_{L^R( [-\pi,\pi] ; \W^{s,4} (\R^3)) } \leq  \frac{\Lambda N^{s-1/4}}{2} \right\} \right) .
\end{aligned}
\end{multline}

As a result, we are reduced to prove the result for a finite number of terms in the initial data. We notice that there are two constants $ C_1,C_2 > 0 $ such that $ C_1 n^{ \frac{1}{3} } \leq \lambda^2_n \leq C_2 n^{ \frac{1}{3} } $ then that $  | \lbrace \lambda_n \leq K  \rbrace | \leq C K^{ 6 }$. Thus by Proposition~\ref{1-dispersive} and by the Cauchy-Schwarz inequality

\begin{equation}\label{1-esti6}
 \bigg\{ \omega \in \Omega \;: \big\| \, [ u^\omega_0]_K \, \big\|_{ \mathcal{H}^\sigma(\R^3) } \leq \frac{\Lambda}{2} \bigg\} = \bigg\{  \omega \in \Omega \;: \sum_{\lambda_n \leq  K} \lambda_n^{2\sigma } | c_n |^2 |g_n(\omega) |  ^2 \leq \frac{\Lambda ^2}{C K^6 }  \bigg\}  \subset A_\Lambda 
\end{equation}
where
\begin{equation*}
\begin{aligned}
A_\Lambda = \bigcap & \ \ \bigg\{ \omega \in \Omega \;: \| \, (e^{itH} [ u^\omega_0]_K)^2 \, \|_{  L^4( [-\pi,\pi] ; \mathcal{H}^s (\R^3)) } \leq \frac{\Lambda^2}{2} \bigg\} ,
\\  
 \bigcap  & \ \ \bigg\{ \omega \in \Omega \;: \| \, (e^{itH} [ u^\omega_0]_K )^3 \, \|_{  L^4( [-\pi,\pi] ; \mathcal{H}^s (\R^3)) } \leq \frac{\Lambda^3}{2} \bigg\} ,
\\  
  \bigcap & \ \ \underset{N} \bigcap   \bigg\{   \omega \in \Omega \;:  \| \Delta_N(e^{itH} [ u^\omega_0]_K ) \|_{L^4( [-\pi,\pi] ; L^\infty (\R^3)) } \leq \frac{\Lambda}{2}  N^{-1/6} \bigg\} ,
\\  
  \bigcap  & \ \ \underset{N} \bigcap  \bigg\{  \omega \in \Omega \;:   \| \Delta_N(e^{itH} [ u^\omega_0]_K ) \|_{L^R( [-\pi,\pi] ; \W^{s,4} (\R^3)) } \leq \frac{\Lambda}{2}  N^{s-1/4}  \bigg\} .
\end{aligned}
 \end{equation*}

We are now able to complete the proof of  Theorem \ref{1-thm1}. By   \eqref{1-deviation2} and   \eqref{1-esti6}, it is sufficient to show that for any integer $ K \geq 1 $ and any $ \Lambda > 0 $,
\begin{equation*}
P \Big(  \omega \in \Omega \;: \sum_{\lambda_n \leq  K} \lambda_n^{2\sigma } | c_n |^2 |g_n(\omega) |  ^2 \leq \frac{ \Lambda^2}{C K^6 }  \Big) > 0.
\end{equation*}
But, by independence,
\begin{multline*}
P \Big(  \omega \in \Omega \;: \sum_{\lambda_n \leq  K} \lambda_n^{2\sigma } | c_n |^2 |g_n(\omega) |  ^2 \leq \frac{ \Lambda^2}{C K^6 }  \Big) \geq  \\
\begin{aligned}
& \geq  P \Big(  \underset{\lambda_n \leq K }{\bigcap}  \Big\{ \omega \in \Omega \;:  |g_n(\omega) |  ^2 \leq \frac{ \Lambda^2}{C K ^{12} \|u_0\|^2_{  \mathcal{H}^\sigma(\R^3)  }  } \Big\} \Big) \\
 &  =   \underset{\lambda_n \leq K }{\prod} \ P \Big( \omega \in \Omega \;:   |g_n(\omega) |  ^2 \leq \frac{ \Lambda^2}{C K ^{12} \|u_0\|^2_{  \mathcal{H}^\sigma(\R^3)  } } \Big)   > 0,
\end{aligned}
\end{multline*}
because for all $ R > 0 $ and all  $ n  \in \N$,   $P\big( \omega  \in \Omega \;: |g_n(\omega)| \leq R  \big) > 0$.
 
It remains to show \eqref{1-proba1}. By  Theorem \ref{1-existence} and Theorem \ref{1-scattering},  it is sufficient to establish that for any  $ \Lambda > 0 $,
\begin{equation*}
\lim _{  \eta \rightarrow 0 } \ P \left( \omega \in \Omega_\Lambda    \ : \ \|u^\omega_0\|_{\mathcal{H}^\sigma {(\R^3)}   }   \leq \eta   \right) = 1.
\end{equation*}
Actually, by adapting the proof of \cite[Appendix A.2]{burq7}, we can  obtain that 
\begin{equation*}
P \left( \omega \in \Omega^c_\Lambda   \   :\ \|u_0^\omega\|_{ \mathcal{H}^\sigma(\R^3) } \leq \eta \right) \leq C e ^{  -c \frac{\Lambda^2}{\eta^2} },
\end{equation*}
and \eqref{1-proba1} is proven.

\subsection{Proof of  \eqref{quanti}}
To prove  \eqref{quanti}, it is enough to show that for any $ \Lambda> 0 $ and any $ \alpha \in ]0,1] $,   $ P( \Omega_\Lambda ) \geq 1 - \alpha $. This result is clear from Theorem~\ref{1-deviation} since
\begin{equation*}
P( \Omega_\Lambda  ) \geq 1 -  C_1 \exp\big( -  \frac{C_2}{  \|u_0\|^2_{ \mathcal{H}^{\sigma} (\R^3)  }   }  \big) \geq 1 - \alpha 
\end{equation*}
if $\|u_0\|_{ \mathcal{H}^{\sigma} (\R^3)  } $ is small enough.

%%%%%%%%%%%%%%%%%%%%%%%%%%%%%%%%%%%%%%%%%%%%%%%%%%%%%%%%%%%%%%%%%%%%%%%%

\section{The fixed point argument for equation \eqref{2schrodinger}}\label{Sect8}

We introduce the following equation:
\begin{equation} \label{2schrodingerH} 
  \left\{
      \begin{aligned}
         & i \frac{ \partial u }{ \partial t } - H u = \kappa \cos (2 t)^{   \frac{d(p-1)}{2} -2 }   |u|^{p-1} u, \quad (t,x) \in \R \times \R^d,
       \\  &  u(0,x)  =u_0(x),
      \end{aligned}
    \right.
\end{equation}
where $ p \geq 5 $ denotes an odd integer and $ \kappa \in \lbrace -1, 1 \rbrace $. We will first show that this equation is almost surely locally well-posed, and then deduce global well-posedness of \eqref{2schrodinger} using the lens transform.

\subsection{Some nonlinear estimates and local well-posedness of equation \eqref{2schrodingerH}}

In this section, we establish estimates that will be useful to apply the Picard fixed point theorem. We start by showing two preliminary lemmas.
\begin{lem} \label{2injection}
Let  $ (q,r) \in [2,+ \infty [ \times [2,+ \infty ]$, $s,s_0 \geq 0 $  and assume that $ s-s_0 >  \frac{d}{2} - \frac{2}{q} - \frac{d}{r} $, then there are two constants $ \kappa,C > 0 $ such that for all $ T \geq 0 $ and $  u \in \overline{X}_T^s $,
 \begin{equation*}
\| u \|_{ L^q( [-T,T] ;  \W^{s_0,r}(\R^d)) }  \leq  C T ^\kappa \|u\|_{  \overline{X}_T^s }.
\end{equation*}
\end{lem}

\begin{proof}  Let  $ \epsilon > 0  $ then there exists  $ \kappa_\epsilon > 0  $  such that  
\begin{align*}
\| u \|_{ L^q( [-T,T] ;  \W^{s_0,r}(\R^d)) }  \leq T^{\kappa_\epsilon} \| u \|_{ L^{q+\epsilon}( [-T,T] ;  \W^{s_0,r}(\R^d)) }.
\end{align*}
Now the couple  $(q+\epsilon, \frac{2d(q+\epsilon)}{dq+d\epsilon-4})$ is admissible with 
\begin{equation*}
   \W^{s,\frac{2d(q+\epsilon)}{dq+d\epsilon-4}}(\R^d)  \hookrightarrow \W^{s_0,r}(\R^d)\quad  \mbox{ if } \quad s-s_0 \geq  \frac{d}{2} - \frac{2}{q+\epsilon} - \frac{d}{r}.
\end{equation*}
But, as  $ s-s_0 >  \frac{d}{2} - \frac{2}{q} - \frac{d}{r} $ then there is $ 0 < \epsilon \ll 1 $  such that  $ s-s_0 \geq  \frac{d}{2} - \frac{2}{q+\epsilon} - \frac{d}{r}$. 
\end{proof}

\begin{lem} \label{2gradient2} Let $ s \geq 0 $ then there exists a constant $ C > 0 $ such that for all functions $f$ and $g$ in~$ \mathcal{S}   (\R^d) $,
\begin{equation*}
\| (-\Delta)^{s/2}(  f g  ) \|_{L^2(\R^d)} \leq C   \big( \|  | g  (-\Delta)^{s/2}f       \|_{L^2(\R^d)} + \|  f     (-\Delta)^{s/2} g \|_{L^2(\R^d)} \big) .
\end{equation*}
\end{lem}
\begin{proof} Using the  Fourier transform   we obtain
\begin{eqnarray*}
\| (-\Delta)^{s/2}(  f g  ) \|_{L^2(\R^d)} & \leq & C   \|  |\xi|^s  \mathcal{F} (  f g   ) \|_{L^2(\R^d)} \\
 & =&  C   \|  |\xi|^s   \mathcal{F} (  f ) \star  \mathcal{F} ( g )    \|_{L^2(\R^d)}.
\end{eqnarray*}
But for all  $ \xi , \eta  \in  \R^d $, 
\begin{align*}
|\xi|^s \leq ( |\eta| + |\xi-\eta| )^s \leq C_s   ( |\eta|^s +  |\xi-\eta|^s  ),
\end{align*}
therefore 
\begin{eqnarray*}
\| (-\Delta)^{s/2}(  f g  ) \|_{L^2(\R^d)} & \leq &C_s   \big(  \|   ( |.|^s \mathcal{F} (  f ) ) \star   \mathcal{F} ( g )      \|_{L^2(\R^d)} + \|  \mathcal{F} (  f )  \star  ( |.|^s \mathcal{F} (  g ) )      \|_{L^2(\R^d)} \big) \\
  & = &C_s   \big(  \|   \mathcal{F} (  g(-\Delta)^{s/2} f     )      \|_{L^2(\R^d)} + \|  \mathcal{F} (f (-\Delta)^{s/2} g )      \|_{L^2(\R^d)} \big)  \\
 & = &C_s   \big(  \|  g  (-\Delta)^{s/2} f           \|_{L^2(\R^d)} + \|    f      (-\Delta)^{s/2} g       \|_{L^2(\R^d)}  \big) ,
\end{eqnarray*}  
which was the claim.
\end{proof}

Then, the expected estimates are established.

\begin{prop} \label{2estim1}
Let $ s > \frac{d}{2} - \frac{2}{p-1} $  then there exist two constants $ C > 0 $ and $ \kappa > 0 $ such that if we assume
\begin{equation*}
 \| e^{-itH} u_0 \|_{ L^p  (  [-\pi, \pi ] ;   L^ \infty ( \R^d )  )  } \leq \lambda 
\end{equation*}
for some $ \lambda >0$, then for any $ 0 < T \leq 1 $, $ v \in \overline{X}^s_T $ and $ f_j = v $ or $ f_j = e^{-itH} u_0 $,
\begin{equation*}
\|  (\sqrt{-\Delta} ^s v  )     \prod_{ j = 2 } ^p  f_j  \|_{  L^1( [-T,T]  ; L ^2(\R^d)) } \leq C  T ^\kappa(  \lambda ^p + \|v\|^p_{\overline{X}^s_T}  ),
\end{equation*}
and
\begin{equation*}
\| \langle x\rangle ^s     v      \prod_{ j = 2 } ^p  f_j  \|_{  L^1( [-T,T]  ; L ^2(\R^d)) } \leq C  T ^\kappa(  \lambda ^p + \|v\|^p_{\overline{X}^s_T}  ).
\end{equation*}
\end{prop}

\begin{proof} According to the H\"older inequality and   \eqref{1-comparaison},
\begin{multline*}
 \| ( \sqrt{-\Delta}^s v  )     \prod_{ j = 2 } ^p  f_j  \|_{  L^1( [-T,T]  ; L ^2(\R^d)) }   \leq  \\
\begin{aligned}
&\leq  \|  \sqrt{-\Delta}^s v \|_{  L^\infty ( [-T,T] ;  L^2(\R^d)) }   \prod_{ j = 2 } ^p \| f_j \|_{  L^{p-1} ( [-T,T] ;  L^\infty(\R^d)) } \\
& \leq  C \| v \|_{  L^\infty ( [-T,T] ;  \mathcal{H}^s (\R^d)) }   \prod_{ j = 2 } ^p \| f_j \|_{  L^{p-1} ( [-T,T] ;  L^\infty(\R^d)) },
\end{aligned}
\end{multline*}
and
\begin{eqnarray*}
  \| \langle x\rangle ^s    v       \prod_{ j = 2 } ^p  f_j  \|_{  L^1( [-T,T]  ; L ^2(\R^d)) }  &\leq &  \| \langle x\rangle ^s v   \|_{  L^\infty ( [-T,T] ;  L^2(\R^d)) }   \prod_{ j = 2 } ^p \| f_j \|_{  L^{p-1} ( [-T,T] ;  L^\infty(\R^d)) } \\
& \leq & \| v \|_{  L^\infty ( [-T,T] ;  \mathcal{H}^s (\R^d)) }   \prod_{ j = 2 } ^p \| f_j \|_{  L^{p-1} ( [-T,T] ;  L^\infty(\R^d)) }.
\end{eqnarray*}
If $ f_j = v $ then as  $ s > \frac{d}{2} - \frac{2}{p-1} $, we can use Lemma~\ref{2injection} to get
\begin{equation*}
\| v \|_{  L^{p-1} ( [-T,T] ;  L^\infty(\R^d)) }  \leq C T^{\kappa  } \| v \|_{  \overline{X}^s_T }.
\end{equation*}
If $ f_j = e^{-itH} u_0 $ then according to the H\"older inequality, 
\begin{eqnarray*}
\| e^{-itH} u_0  \|_{  L^{p-1} ( [-T,T] ;  L^\infty(\R^d)) } & \leq & T^{   \frac{1}{p-1} - \frac{1}{p} } \| e^{-itH} u_0 \|_{ L^p  (  [-\pi,  \pi ] ;   L^ \infty ( \R^d)  )  } \\ 
& \leq & T^{   \frac{1}{p(p-1)} }  \lambda,
\end{eqnarray*}
which completes the proof.
\end{proof}

\begin{lem} \label{2estim2}
If $ 0< s< \frac{d}{2}$ then there exist two constants $ C > 0 $ and $ \kappa > 0 $ such that if we assume that
\begin{equation*}
 \| e^{-itH} u_0 \|_{ L^p  (  [-\pi,  \pi ] ;    \W^{  \frac{1}{8} , \infty } ( \R^d )  )  } \leq \lambda, 
\end{equation*}
and
\begin{equation*}
 \| u_0 \|_{   \mathcal{H}^{  \frac{d-1}{2} } ( \R^d   )  } \leq \lambda 
\end{equation*}
for some $\lambda>0$, then for all $ 0 < T \leq 1 $,
\begin{equation*}
\| \langle x\rangle ^s    ( e^{-itH } u_0 )^p \|_{  L^1( [-T,T]  ; L ^2(\R^d)) } \leq C  T ^\kappa  \lambda ^p.
\end{equation*}
\end{lem}

\begin{proof}  According to the H\"older inequality and   \eqref{1-comparaison}, we obtain
\begin{multline*}
  \| \langle x\rangle ^s    ( e^{-itH } u_0 )^p \|_{  L^1( [-T,T]  ; L ^2(\R^d)) }  \leq \\
 \begin{aligned}
& \leq  \| \langle x\rangle ^{ \frac{d}{2}  }   ( e^{-itH } u_0 )^p \|_{  L^1( [-T,T]  ; L ^2(\R^d)) }\\
& \leq   \| \langle x\rangle ^{ \frac{d-1}{2}  }     e^{-itH } u_0  \|_{  L^\infty( [-T,T]  ; L ^2(\R^d)) }    \| \langle x\rangle ^{ \frac{1}{2(p-1)}  }   e^{-itH } u_0  \|_{  L^{p-1}( [-T,T]  ; L ^\infty(\R^d)) }  ^{p-1} \\
&  \leq  C T^{1/p}  \|  u_0  \|_{  \mathcal{H} ^{  \frac{d-1}{2}}(\R^d) }    \|   e^{-itH } u_0  \|_{  L^{p}( [-T,T]  ; \W^{ \frac{1}{8} , \infty}(\R^d))  }  ^{p-1}\\
& \leq  C  T^{1/p} \lambda ^p.  
 \end{aligned}
\end{multline*}
\end{proof}

\begin{lem} \label{2estim3}
There exist $ s \in  ]  \frac{d}{2}   - \frac{2}{p-1} ; \frac{d}{2} [ $  and constants $ C, \kappa > 0 $  such that if we assume that
\begin{align*}
 \| e^{-itH} u_0 \|_{ L^{2p}  (  [- \pi,  \pi ] ;   \W ^ { \frac{1}{7}  ;  \infty }  ( \R^d )  )  } & \leq \lambda,
 \end{align*}
and
 \begin{align*}
 \|u_0\|_{  \mathcal{H}^{ \frac{d-1}{2} } (\R^d)} & \leq \lambda  
\end{align*}
for some  $ \lambda>0 $, then for all $ 0 < T \leq 1 $, $ v \in \overline{X}^s_T $ and $ f_j = v $ or $ f_j = e^{-itH} u_0 $,
\begin{equation*}
\|  \sqrt{-\Delta}^s( e^{-itH} u_0   )     \prod_{j = 2 } ^p  f_j  \|_{  L^1( [-T,T]  ; L ^2(\R^d)) } \leq C  T ^\kappa(  \lambda ^p + \|v\|^p_{\overline{X}^s_T}  ).
\end{equation*}
\end{lem}

\begin{proof} For all $ \epsilon \in ] 0, \frac{1}{2} [ $, according to the H\"older inequality, we have 
\begin{multline*}
  \|  \sqrt{-\Delta}^s( e^{-itH} u_0   )     \prod_{ j = 2 } ^p  f_j  \|_{  L^1( [-T,T]  ; L ^2(\R^d)) }  \leq \\
    \leq  \Big\|   \frac{ \sqrt{-\Delta}^s( e^{-itH} u_0   ) }{ \langle x\rangle ^{1/2-\epsilon} }\Big\|_{  L^2 ( [-T,T] ; L^2 (\R^d)) } \prod_{ j = 2 } ^p \|  \langle x\rangle ^{\frac{1}{p-1} ( \frac{1}{2}  - \epsilon )} f_j \|_{  L^{2(p-1)} ( [-T,T] ;  L^{ \infty }(\R^d)) }.
\end{multline*}
Then, we choose $ s = \frac{d}{2} - 2\epsilon $ with $ \epsilon \ll 1 $ to obtain by using~\eqref{2effectregularisant2} that
\begin{multline*}
  \| \sqrt{-\Delta}^s( e^{-itH} u_0   )     \prod_{ j = 2 } ^p  f_j  \|_{  L^1( [-T,T]  ; L ^2(\R^d)) }  \leq \\
  \begin{aligned}
&\leq   \lambda   \prod_{ j = 2 } ^p \| \langle x\rangle ^{\frac{1}{p-1} ( \frac{1}{2}  - \epsilon )} f_j \|_{  L^{2(p-1)} ( [-T,T] ;  L^{  \infty }(\R^d)) } \\
&\leq   \lambda   \prod_{ j = 2 } ^p \|  f_j \|_{  L^{2(p-1)} ( [-T,T] ;   \W ^{ \frac{1}{p-1} ( \frac{1}{2}  - \epsilon ) + \epsilon , \frac{d}{\epsilon} +1 }(\R^d)) }.
\end{aligned}
\end{multline*}
If $ f_j = e^{-itH} u_0 $, by interpolation, we  there exists a constant $ \kappa > 0 $ such that
\begin{multline*}
  \| e^{-itH} u_0  \|_{   L^{2(p-1)} ( [-T,T] ;    \W ^{ \frac{1}{p-1} ( \frac{1}{2}  - \epsilon ) + \epsilon , \frac{d}{\epsilon} +1 }(\R^d))  }  \leq \\
   \leq  C   T ^\kappa    \| e^{-itH} u_0  \|^\theta_{   L^{2p} ( [-T,T] ;  \W ^{ s_0 , \infty }(\R^d))  }   \| e^{-itH} u_0  \|^{1-\theta}_{   L^{\infty} ( [-T,T] ;    \mathcal{H} ^{ \frac{d-1}{2} }(\R^d))  }
\end{multline*}
where $ \theta = \frac{d-\epsilon}{d+\epsilon} $ and $ s_0 = (  \frac{1-\theta}{\theta} )( \frac{d-1}{2}  )+ \frac{1}{\theta(p-1)} ( \frac{1}{2} - \epsilon  ) + \frac{\epsilon}{\theta} $. \medskip

Firstly,  $ \| e^{-itH} u_0  \|_{   L^{\infty} ( [-T,T] ;  \mathcal{H} ^{ \frac{d-1}{2}  }(\R^d))  } =\| u_0  \|_{   \mathcal{H} ^{ \frac{d-1}{2} }(\R^d)  } \leq \lambda $. Then as 
\begin{align*}
s_0 = \frac{1}{2(p-1)} + C \epsilon + o (\epsilon) \leq \frac{1}{7}
\end{align*}
we deduce $ \| e^{-itH} u_0  \|_{   L^{2p} ( [-T,T] ;   \W ^{ s_0 , \infty }(\R^d))  } \leq \lambda $ and   \\$ {\| e^{-itH} u_0  \|_{   L^{2(p-1)} ( [-T,T] ;   \W ^{ \frac{1}{p-1} ( \frac{1}{2}  - \epsilon ) + \epsilon , \frac{d}{\epsilon} +1 }(\R^d))  } \leq \lambda} $. \medskip

If  $ f_j = v $, since $ s - \frac{1}{p-1} (  \frac{1}{2} - \epsilon ) > \frac{d}{2} - \frac{1}{p-1} - \frac{d\epsilon}{d+\epsilon} $ (if $ \epsilon \ll \frac{1}{2(p-2)} $) then by Lemma~\ref{2injection}, we obtain 
$$ \| v  \|_{  L^{2(p-1)} ( [-T,T] ;   \W ^{ \frac{1}{p-1} ( \frac{1}{2}  - \epsilon ) + \epsilon , \frac{d}{\epsilon} +1 }(\R^d))  }  \leq C  T^\kappa  \|v\|_{\overline{X}^s_T },$$
hence the result.
\end{proof}

We therefore introduce the natural set of initial conditions. 
Let $ \lambda \geq 0 $ and define $ F_0( \lambda ) $ as the set of functions $ u_0 \in \mathcal{H}^{ \frac{d-1}{2} } (\R^d) $ which satify
\begin{equation}\label{defF}
\left\{
    \begin{array}{ll}
      \|u_0\|_{  \mathcal{H}^{  \frac{d-1}{2} } (\R^d) } & \leq \lambda,   
       \\  \|  e^{-itH} u_0 \|_{ L^{2p}( [-\pi,\pi] ; \W^{  \frac{1}{7} , \infty } (\R^d))  } & \leq \lambda.
    \end{array} 
\right.
\end{equation}

Then, we can establish the two main results of this section.
\begin{prop} \label{2pointfixe1} There exists $ s \in ] \frac{d}{2}   - \frac{2}{p-1}  ; \frac{d}{2} [ $ and constants  $ C, \kappa > 0 $ such that if $ u_0 \in F_0(\lambda ) $  for some $ \lambda > 0 $ then for all $ v, v_1,v_2 \in \overline{X}_T^s $ and $ 0 < T \leq 1 $,
 \begin{multline*}
\Big\|  \int_0^t   e^{-i(t-s)H}   \cos ( 2s )^{ \frac{d(p-1)}{2} -2  } |  e^{-isH} u_0 + v |^{p-1}   (   e^{-isH} u_0 + v ) ds   \Big\|_{  \overline{X}_T^s }   \leq \\
\leq C   T^\kappa   ( \lambda ^p +  \|v\|^p_{  \overline{X}_T^s } ),
\end{multline*} 
and 
 \begin{align*}
& \Big\|  \int_0^t   e^{-i(t-s)H}   \cos (2s)^{ \frac{d(p-1)}{2} -2  }  |  e^{-isH} u_0 + v_1 |^{p-1}   (  e^{-isH} u_0 + v_1 ) ds  
\\ & \hspace*{2cm} -  \int_0^t   e^{-i(t-s)H}  \cos (2s)^{ \frac{d(p-1)}{2} -2  }  |  e^{-isH} u_0 + v_2 |^{p-1}   (  e^{-isH} u_0 + v_2 ) ds  \Big\|_{  \overline{X}_T^s }  
\\ & \leq C  T ^\kappa   \|v_1-v_2\|_{ \overline{X}^s_T }   ( \lambda ^{p-1} +  \|v_1\|^{p-1}_{  \overline{X}_T^s } + \|v_2\|^{p-1}_{  \overline{X}_T^s } ).
\end{align*} 
\end{prop} 

\begin{proof} We only prove the first estimate. Using Proposition~\ref{1-Strichartz} and~\eqref{1-comparaison}, we obtain
\begin{multline*}
 \Big\|  \int_0^t   e^{-i(t-s)H}  \cos ( 2s )^{  \frac{d(p-1)}{2} -2  }  |  e^{-isH} u_0 + v |^{p-1}   (   e^{-isH} u_0 + v ) ds    \Big\|_{  \overline{X}_T^s }  \leq \\
\begin{aligned}
&  \leq   C  \|  \cos ( 2s )^{  \frac{d(p-1)}{2} -2  }    |  e^{-isH} u_0 + v |^{p-1}   (   e^{-isH} u_0 + v )   \|_{   L^1( [-T,T] ;  \mathcal{H}^s(\R^d) )   }  \\
& \leq  C  \| \,   |  e^{-isH} u_0 + v |^{p-1}   (   e^{-isH} u_0 + v )   \|_{   L^1( [-T,T] ;  \mathcal{H}^s(\R^d) )   }  \\
& \leq  C   \| \,  (-\Delta)^s  \big(   |  e^{-isH} u_0 + v |^{p-1}   (   e^{-isH} u_0 + v ) \big)   \|_{   L^1( [-T,T] ;  L^2 (\R^d) )   }  \\
  & \qquad \qquad \qquad+  C \| \langle x\rangle ^s     |  e^{-isH} u_0 + v |^{p-1}   (   e^{-isH} u_0 + v )   \|_{   L^1( [-T,T] ;  L^2 (\R^d) )   }.   
\end{aligned}
\end{multline*}
Then, using Lemma~\ref{2gradient2} and Proposition~\ref{2estim1}, Lemma~\ref{2estim2}, and Lemma~\ref{2estim3}, we can obtain a constant $ \kappa > 0 $ such that for all  $ u_0 \in F_0(\lambda) $, $ 0 < T \leq 1 $ and $ v  \in \overline{X}^s_T $,
\begin{align*}
\|   (-\Delta)^s  \big(  |  e^{-isH} u_0 + v |^{p-1}   (   e^{-isH} u_0 + v )  \big)   \|_{   L^1( [-T,T] ;  L^2 (\R^d) )   }  \leq C T ^\kappa ( \lambda^p +   \|v\|^p_{  \overline{X}_T^s }   ), 
\end{align*}
and
\begin{align*}
\|  \langle x\rangle ^s     |  e^{-isH} u_0 + v |^{p-1}   (   e^{-isH} u_0 + v )   \|_{   L^1( [-T,T] ;  L^2 (\R^d) )   }    \leq C T ^\kappa ( \lambda^p +   \|v\|^p_{  \overline{X}_T^s }   ),
\end{align*}
which was to prove.
\end{proof}

 We are now able to state the local well-posedness result for equation~\eqref{2schrodingerH}.

\begin{thm}  \label{2existencebis}
There exist $ s \in ] \frac{d}{2}   - \frac{2}{p-1} ; \frac{d}{2} [  $, $ C > 0 $ and $ \delta > 0 $ such that for any $ 0 < T \leq 1 $, if $ u_0 \in F_0(\lambda ) $ with $ \lambda < C     T ^{ - \delta  } $ then there is a unique solution to the equation~\eqref{2schrodingerH} on $ [-T,T] $ in the space $ e^{-itH}u_0 + B_{\overline{X}^s_T } (0,  \lambda   ) $.
\end{thm}

\begin{proof} Let us define 
\begin{equation*}
L ( v) = -i\kappa  \int_0^t   e^{-i(t-s)H}  \cos (2s)^{  \frac{d(p-1)}{2} -2  }  |   e^{-isH} u_0 + v(s) |^{p-1} (  e^{-isH} u_0 + v(s) ) ds,
\end{equation*}
and note that $ u = e^{-itH} u_0 + v $ is the unique solution of~\eqref{2schrodingerH} on $ [-T,T] $ in the space $ e^{-itH}u_0 + B_{\overline{X}^s_T}(0 , R ) $ if and only if $v$ is the unique fixed point of $L$ on $ B_{\overline{X}^s_T}(0 , R ) $. \medskip

According to Proposition~\ref{2pointfixe1}, there are two constants $ C > 0 $ and $ \kappa > 0 $ such that 
\begin{align*}
& \| L (v)  \|_{  \overline{ X } ^s_T  } \leq C T^\kappa ( \lambda ^p + \| v\|^p _{ \overline{ X } ^s_T }) 
\\ & \| L (v_1)-L(v_2)  \|_{  \overline{ X } ^s_T  } \leq C T^\kappa \| v_1-v_2\| _{ \overline{ X } ^s_T } ( \lambda ^{p-1} + \| v_1\|^{p-1}_{ \overline{ X } ^s_T }+ \| v_2\|^{p-1}_{ \overline{ X } ^s_T } ).
\end{align*}
Therefore if $ \lambda <   (  \frac{1}{8 C  T^\kappa }  )^{  \frac{1}{p-1} }$ then $L$ is a contraction of $ B_{\overline{X}^{s}_T } (0, \lambda   ) $ and the theorem follows.
\end{proof}

%%%%%%%%%%%%%%%%%%%%%%%%%%%%%%%%%%%%%%%%%%%%%%%%%%%%%%%%%%%%%%%%%%%%%%%%

\subsection{Global solutions and scattering for the equation \eqref{2schrodinger}}
Thanks to the lens transform and the results of the previous section, we are able obtain the existence of global solutions for the equation \eqref{2schrodinger}.

\begin{thm}  \label{2existence1}
There exist $  s \in ]  \frac{d}{2}   - \frac{2}{p-1} ; \frac{d}{2} [ $ and  $ C_1,C_2 > 0 $ such that if $ u_0 \in F_0(\lambda ) $ with $ \lambda < C_1 $ then there is a global solution to \eqref{2schrodinger} in the space $ e^{it\Delta}u_0 + B_{X^s } (0, C_2 ) $.
\end{thm}

\begin{proof} Let $u$ be given by Theorem \ref{2existencebis} with $ T = \frac{\pi}{4} $. We apply to $u$ the lens transformation defined in Section~\ref{1-214}  to obtain a function $ U $ which, according to Proposition~\ref{tvb1}, satisfies the conditions of the theorem.
\end{proof}

\begin{thm}  \label{2existence2}
There exist  $ s \in ]  \frac{d}{2}   - \frac{2}{p-1} ; \frac{d}{2} [ $, $ C_1,C_2 > 0 $ and $ \delta > 0 $ such that for any  $ 0 < T \leq 1 $, if $ u_0 \in F_0(\lambda ) $ with  $ \lambda < C_1 ( \arctan 2 T ) ^ {- \delta  }  $ then there is a solution to~\eqref{2schrodinger} on $ [-T,T] $ in the space $ e^{it\Delta}u_0 + B_{X_T^s } (0, C_2 \lambda ^p  ) $.
\end{thm}
\begin{proof} Let $u$ be given by Theorem \ref{2existencebis} at $T$ replaced by $ \frac{1}{2} \arctan 2T $. Then, as for the previous proof, we apply to $u$ the lens transformation, which yields the result. 
\end{proof}

 We then prove the uniqueness of the constructed solutions.
\begin{prop} \label{2unicite}
Let $  s \in ]  \frac{d}{2}   - \frac{2}{p-1} ; \frac{d}{2} [ $, $ u_0 \in F_0(\lambda) $ and $ T \in ]0,1 ] $. Let  $ U_1 $ and $ U_2 $ be two  solutions of~\eqref{2schrodinger} on $[-T,T]$ on the space $  e^{it\Delta} u_0 + X_T^s $ then,
\begin{equation*}
 U_1(t) = U_2(t) \ \mbox{in} \ L^2(\R^d), \ \forall t \in [-T,T].
\end{equation*} 
\end{prop}

\begin{proof} By reversibility of the equation,  is enough to consider the case  $ t \in [0,T] $. For all $ t \in \R $, we have 
\begin{multline*}
\frac{d}{dt} \| U_1(t) - U_2(t) \|^2_{ L^2(\R^d) } =\\
\begin{aligned}
&=  2 \Re ( \langle   \partial_t(U_1(t)-U_2(t)) , U_1(t)-U_2(t)  \rangle _{L^2(\R^d)   L^2(\R^d) } ) \\
&=   2 |   \langle   | U_1(t)| ^{p-1} U_1(t) - |U_2(t)|^{p-1}  U_2(t)  , U_1(t)-U_2(t)  \rangle _{L^2(\R^d)   L^2(\R^d) } | \\
& \leq  2 \| U_1(t) - U_2(t) \|_{L^2(\R^d)}   \|  |U_1(t)| ^{p-1} U_1(t) - |U_2(t)|^{p-1} U_2(t) \|_{L^2(\R^d)} \\
& \leq  C \|U_1(t) - U_2(t) \|^2_{L^2(\R^d)}   \big(  \|  U_1(t) \|^{p-1}_{L^\infty(\R^d)} +   \|  U_2(t) \|^{p-1}_{L^\infty(\R^d)}  \big).
\end{aligned}
\end{multline*}
Then, by the Gr\"onwall Lemma, the result is proved if $ \| U_j(t) \|^{p-1}_{L^\infty(\R^d)}   \in L^1_{loc} $, since $ \| U_1(0) - U_2(0) \|^2_{ L^2(\R^d) } = 0 $.  By using Proposition~\ref{tvb1}, we get
\begin{eqnarray*}
\| U_j \|_{ L^{p-1}([0,T]), L^\infty(\R^d)  } & \leq   & \| e^{it\Delta} u_0 \|_{ L^{p-1}([0,T] ; L^\infty(\R^d) )  } +\| V_j \|_{ L^{p-1}([0,T];  L^\infty(\R^d) ) }\\
& \leq & C_T  (     \| e^{-itH} u_0 \|_{ L^{p-1}([- \pi, \pi ];  L^\infty(\R^d) )  } + \| V_j \|_{ X_T^s } ) \\
 & \leq & C_T  (    \lambda + \|  V_j \|_{ X_T^s } ),
\end{eqnarray*}
hence the result.
\end{proof}

Finally, we prove that the global solutions constructed scatter  at $ + \infty $ and at $ - \infty $, namely \eqref{2scat11} and \eqref{2scat12}. The proof is the same as the proof of Theorem~\ref{1-scattering} and  is left here.

%%%%%%%%%%%%%%%%%%%%%%%%%%%%%%%%%%%%%%%%%%%%%%%%%%%%%%%%%%%%%%%%%%%%%%%%

\section{Estimation of the regularity of the initial random data  and proof of  Theorem~\ref{2thm1}}\label{Sect9}

\subsection{Estimation of the regularity of the initial random data}
For  $ \Lambda > 0 $, recall the definition~\eqref{defF} of the set $F_0(\Lambda)$, and define
\begin{equation*}
\Omega_\Lambda = \big\{ \omega \in \Omega \;:  u_0 ^\omega  \in F_0( \Lambda )  \big\}.
\end{equation*}

The purpose of this part is to establish the following theorem:

\begin{thm} \label{2cond2} There exist constants $  C, c > 0 $ such that for any $ \Lambda > 0 $,
\begin{equation*} 
P(  \Omega^c_\Lambda  )\leq C \exp \Big(  - c  \frac{  \Lambda^2}{   \|u_0\|^2_{     \mathcal{H}^{ (d-1)/2  } (\R^d)  } }  \Big).
\end{equation*}
\end{thm}

\begin{proof} By the triangle inequality, we can write
\begin{multline} \label{2deuxtermes}
P(  \Omega^c_\Lambda  )  \leq \\
  P \big( \omega \in \Omega \ : \ \|u^\omega_0\|_{     \mathcal{H}^{ (d-1)/2  }(\R^d) } \geq \Lambda\big)   +  P  \big( \omega \in \Omega \ : \ \|  e^{-itH} u_0 \|_{ L^{2p}( [-\pi,\pi] ;  \W^{  \frac{1}{7} , \infty } (\R^d))  }  \geq \Lambda  \big)
\end{multline}
and it suffices to bound each of the previous  two terms.  
\medskip
 
We first estimate the first term of \eqref{2deuxtermes}. It is sufficient to establish the estimate for $  \Lambda \geq C \|u_0\|_{  \mathcal{H}^{ (d-1)/2 } (\R^d)    }  $. Let $ q \geq 1$ then according to the Markov inequality, we get
\begin{align*}
P \Big(  \omega \in \Omega \ : \   \|  u_0^\omega  \|_{   \mathcal{H}^{ (d-1)/2 } (\R^d)    }  \geq \Lambda \Big) & = P   \Big(  \omega \in \Omega \ : \  \sum_{n \in \N} \lambda_n^{d-1} |c_n|^2 |g_n(\omega)|^2  \geq \Lambda^2 \Big)
\\ & \leq   \Lambda^{-2 q}     \Big\| \sum_{n \in \N} \lambda_n^{d-1} |c_n|^2 |g_n|^2   \Big\|_{L^q( \Omega)}^q 
\\ & \leq   \Lambda^{-2 q}   \Big(  \sum_{n \in \N} \lambda_n^{d-1} |c_n|^2 \| g_n \|^2_{L^{2q}( \Omega)}   \Big)  ^q 
\\ & \leq     \Big(  C  \big( \frac{q}{ c } \big) ^{ \frac{1}{2}  }   \frac{ \|u_0\|_{ \mathcal{H}^{ (d-1)/2 } ( \R^d ) }     }{\Lambda}     \Big) ^{2q} .
\end{align*}
Then we can choose $ q = c  \Big( \frac{ \Lambda }{2C\|u_0\|_{ \mathcal{H}^{ (d-1)/2 } ( \R^d ) } } \Big) ^2 \geq  1  $ to get 
\begin{equation*}
P \Big(  \omega \in \Omega \ : \   \|  u_0^\omega  \|_{   \mathcal{H}^{ (d-1)/2 } (\R^d)    }  \geq \Lambda \Big)  \leq  \frac{1}{2^{2q}} =  e^{ -2\ln(2)q  } \leq   
  \exp \Big( -\frac{c \Lambda^2}{\|u_0\|^2_{  \mathcal{H}^{ (d-1)/2 } (\R^d)    } } \Big) .
\end{equation*}

From then on, the second term of \eqref{2deuxtermes} remains to be estimated. For this, let us recall the estimates of the eigenfunctions of the harmonic oscillator whose proof can be found in \cite[Corollary~3.2]{koch}. 
For all $ p \in [4 , + \infty ]  $, there exists a constant $ C > 0 $ such that for all $ n \in \N $,
\begin{equation} \label{2propre1}
 \\   \| h_n \|_{L^p(\R^d)}  \leq C \lambda_n^{-1 + \frac{d}{2} } \quad \mbox{ if }   d \geq 2.
\end{equation}
  Since $\W^{  \frac{1}{6} , r } ( \R^d )  \hookrightarrow \W^{  \frac{1}{7} , \infty  } (\R^d)$  if    $\frac{1}{6} - \frac{1}{7} > \frac{d}{r}$, we have to prove that there exist two constants $ C,c > 0 $ such that for all $ \Lambda \geq 0 $ and $ r \geq 2 $,
\begin{equation}  \label{2terme2bis}
 P \left(  \omega \in \Omega \ :  \| e^{-itH} u_0^\omega  \|_{  L^{2p}( [-\pi,  \pi ]  ;  \W^{ \frac{1}{6}  , r  }(\R^d))   }  \geq \Lambda \right) \leq C  \exp \Big( -  c \frac{ \Lambda^2 }{\|u_0\|^2_{  \mathcal{H}^{ (d-1)/2 } (\R^d)    } }   \Big).
\end{equation} 
It is sufficient to show the estimate for $ t \geq C \|u_0\|_{ \mathcal{H}^\sigma(\R^d) } $. According to the Markov and Minkowski inequalities, we obtain for $ q \geq  \max( 2p,r , 2 )$,
\begin{multline*}
P \left(  \omega \in \Omega \ : \  \| e^{-itH} u_0^\omega  \|_{  L^{2p}( [- \pi,  \pi ]  ;  \W^{ \frac{1}{6}  , r  }(\R^d))   }  \geq \Lambda \right) \leq \\
\begin{aligned}
&  \leq  \Lambda^{-q }  \| e^{-itH} u_0^\omega  \|^q_{ L^q( \Omega, L^{2p}( [- \pi,  \pi ]  ;  \W^{ \frac{1}{6}  , r  }(\R^d)) )   } \\
& \leq  \Lambda^{-q }  \| H^{\frac{1}{12} } e^{-itH} u_0^\omega  \|^q_{ L^{2p}( [- \pi,  \pi ]  ; L^ r  (\R^d , L^q( \Omega)) )   }.
\end{aligned}
\end{multline*}
Then, thanks to \eqref{zyg1}, we obtain
\begin{multline*}
\| H^{\frac{1}{12} } e^{-itH} u_0^\omega  \|_{ L^q( \Omega )   }  = \Big\| \sum_{n \in \N} \lambda_n^{\frac{1}{6} } c_n e^{-it\lambda_n^2} h_n(x) g_n(\omega )  \Big\|_{ L^q( \Omega )   } \leq \\
\leq C q ^{   \frac{1}{2} }   \sqrt{ \sum_{n \in \N} \lambda_n^{\frac{1}{3} } |c_n|^2 | h_n(x)|^2   }.
\end{multline*}
And finally, by the triangle inequality and   \eqref{2propre1}, we have 
\begin{multline*}
 P \left(  \omega \in \Omega \ :  \  \| e^{-itH} u_0^\omega  \|_{  L^{2p}( [- \pi,  \pi ]  ;  \W^{ \frac{1}{6}  , r  }(\R^d))   }  \geq \Lambda \right)  \leq  \\
\begin{aligned}
&\leq   \Big(  \frac{Cq^{ \frac{1}{2} }}{\Lambda}  \Big)^q  \Big\| \sqrt{ \sum_{n \in \N} \lambda_n^{\frac{1}{3} } |c_n|^2 | h_n(x)|^2   } \Big\|^q_{L^r(\R^d)} \\
& \leq  \Big(  \frac{Cq^{ \frac{1}{2} }}{\Lambda}  \Big)^q  \Big\|  \sum_{n \in \N} \lambda_n^{\frac{1}{3} } |c_n|^2 | h_n(x)|^2    \Big\| ^{q/2}_{L^{r/2}(\R^d)} \\
& \leq  \Big(  \frac{Cq^{ \frac{1}{2} }}{\Lambda}  \Big)^q   \Big( \sum_{n \in \N} \lambda_n^{\frac{1}{3} } |c_n|^2  \|   h_n(x)   \|^2_{L^{r}(\R^d)} \Big)^{q/2} \\
& \leq  \Big(  \frac{Cq^{ \frac{1}{2} }  \|u_0\|_{  \mathcal{H}^{(d-1)/2}(\R^d) } }{\Lambda}  \Big)^q.
\end{aligned}
\end{multline*}
Thus, it is enough to choose $ q = \Big( \frac{\Lambda}{2C\|u_0\|_{  \mathcal{H}^{(d-1)/2}(\R^d) } } \Big)^{2} $ to get \eqref{2terme2bis}.
\end{proof}

 %%%%%%%%%%%%%%%%%%%%%%%%%%%%%%%%%%%%%%%%%%%%%%%%%%%%%%%%%%%%%%%%%%%%%%%%
 
\subsection{Proof of Theorem \ref{2thm1}}
To complete the proof of this result, it is sufficient to establish that for all $ \Lambda > 0 $,
\begin{equation} \label{2cond1}
P( \Omega_\Lambda  ) > 0.
\end{equation}
For $ u_0 =   \displaystyle  \sum_{n \in \N} c_n h_n  \in  L^2 ( \R^d ) $, we define for $ N \in \N^\star$,
\begin{equation*}
 [u_0]_N = \sum_{\lambda_n < N } c_n h_n, \quad [u_0]^N = \sum_{\lambda_n   \geq N} c_n h_n.
\end{equation*}
\begin{lem} There exists $ N \in \N^ * $ such that 
\begin{multline*}
  P(  \Omega _\Lambda )  \geq  \\
   \frac{1}{2}    P \big(\,  \omega \in \Omega \; :  \;  \| \, [u^\omega_0]_N \|_{  \mathcal{H}^{(d-1)/2}  (\R^d) }  \leq \frac{\Lambda}{2}  , \  \|  e^{-itH} [u^\omega_0]_N \|_{ L^{2p}( [-\pi,\pi] ;  \W^{  \frac{1}{7} , \infty } (\R^d))   }  \leq \frac{\Lambda}{2} \big) . 
\end{multline*}
\end{lem}

\begin{proof} By independence, using Theorem \ref{2cond2}, we obtain
\begin{multline*}
  P(  \Omega _\Lambda  )  =   P \Big(  \omega \in \Omega \ :  \  \| \, [u^\omega_0]_N + [u^\omega_0]^N \|_{  \mathcal{H}^{(d-1)/2}  (\R^d) }  \leq \Lambda ,
\\  \hspace*{4cm}  \   \| e^{-itH} [u^\omega_0]_N + e^{-itH} [u^\omega_0]^N  \|_{  L^{2p}( [-\pi,\pi] ;  \W^{ \frac{1}{7} , \infty  } (\R^d)) } \leq  \Lambda \Big)  \\
\begin{aligned}
& \geq  P \Big( \omega \in \Omega \ :  \  \| \, [u^\omega_0]_N \|_{   \mathcal{H}^{(d-1)/2} (\R^d)  } \leq \frac{\Lambda}{2}   , \  \| e^{-itH} [u^\omega_0]_N \|_{  L^{2p}( [-\pi,\pi] ;  \W^{ \frac{1}{7} , \infty  } (\R^d)) } \leq  \frac{\Lambda}{2} \Big)  \cdot 
\\ &  \hspace*{.4cm} \cdot  P \Big(  \omega \in \Omega \ :  \  \| \, [u^\omega_0]^N \|_{  \mathcal{H}^{(d-1)/2}   (  \R^d ) } \leq \frac{\Lambda}{2}    ,  \   \| e^{-itH} [u^\omega_0]^N \|_{  L^{2p}( [-\pi,\pi] ;  \W^{ \frac{1}{7} , \infty  } (\R^d)) } \leq  \frac{\Lambda}{2} \Big)  \\
& \geq P \Big(  \omega \in \Omega \ :  \  \| \, [u^\omega_0]_N  \|_{  \mathcal{H}^{(d-1)/2}   (  \R^d ) } \leq \frac{\Lambda}{2}     , \   \| e^{-itH} [u^\omega_0]_N \|_{  L^{2p}( [-\pi,\pi] ;  \W^{ \frac{1}{7} , \infty  } (\R^d)) } \leq  \frac{\Lambda}{2}  \Big) \cdot
\\ & \hspace*{7cm}   \cdot \Big( 1-  C \exp \big(  - \frac{c \Lambda ^{2}}{    \underset{\lambda_n \geq N  }{\sum}   \lambda_n ^{d-1}  |c_n |^2   }   \big)  \Big) .
\end{aligned}
\end{multline*}
But we have $\dis  \lim_{N \longrightarrow + \infty } 1-  C \exp \Big(  - \frac{c \Lambda ^2}{     \underset{\lambda_n \geq N  }{\sum}   \lambda_n ^{d-1}  |c_n |^2    }   \Big) = 1$,
thus there is $ N \in \N^\star $ such that  $ \dis 1-  C \exp \Big(  - \frac{c \Lambda ^2}{    \underset{\lambda_n \geq N  }{\sum}   \lambda_n ^{d-1}  |c_n |^2   }   \Big) \geq 1/2$. 
\end{proof}

Therefore, to prove \eqref{2cond1}, it is sufficient to prove the following proposition:
\begin{lem}  
For all  $ \Lambda > 0 $ and $ N \in \N ^\star $,
\begin{align*} 
P \Big(  \omega \in \Omega \ :  \  \big\| \, [u^\omega_0]_N \big\|_{  \mathcal{H}^{ (d-1)/2 }  (\R^d) }  \leq \Lambda     , \   \big\| e^{-itH} [u^\omega_0]_N  \big\|_{  L^{2p}( [-\pi,\pi] ;  \W^{ \frac{1}{7} , \infty  } (\R^d)) } \leq  \Lambda \Big) > 0 .
\end{align*}
\end{lem}

\begin{proof} We get
\begin{multline*}
 P \Big(  \omega \in \Omega \ :  \  \big\| \, [u^\omega_0]_N \big\|_{  \mathcal{H}^{ (d-1)/2 }  (\R^d) }  \leq \Lambda    , \   \big\| e^{-itH} [u^\omega_0]_N  \big\|_{  L^{2p}( [-\pi,\pi] ;  \W^{ \frac{1}{7} , \infty  } (\R^d)) } \leq  \Lambda \Big) \geq \\
 \begin{aligned}
 & \geq   P \Big( \, \omega \in \Omega \ :  \   \sum_{ \lambda_n  < N   } \lambda_n^{d-1} |c_n|^2  | g_n(  \omega  )|^2  \leq   \frac{ C \Lambda ^2 }{N^{2d}} \Big) \\
&  \geq  P \Big( \, \omega \in \Omega \ :  \   \sum_{ \lambda_n < N   }  | g_n(  \omega  )|^2  \leq   \frac{ C \Lambda^2  }{ N^{2d}  \| u _0 \|^2_{ \mathcal{H}^{d-1}(\R^d)   }  } \Big) \\
&  \geq  P   \Big(  \underset{\lambda_n < N }{\bigcap}  \Big\{ \, \omega \in \Omega \ : \  | g_n(  \omega  )|^2  \leq   \frac{ C  \Lambda^2 }{N^{4d}  \| u _0 \|^2_{ \mathcal{H}^{d-1}(\R^d)   }  } \Big\} \Big) \\
&  =  \prod_{\lambda_n < N} P  \Big( \, \omega \in \Omega \ :  \  | g_n(  \omega  )|^2  \leq \frac{ C  \Lambda^2 }{N^{4d}  \| u _0 \|^2_{ \mathcal{H}^{d-1}(\R^d)   }  } \Big)  > 0. & 
\end{aligned}
\end{multline*}
\end{proof}

 %%%%%%%%%%%%%%%%%%%%%%%%%%%%%%%%%%%%%%%%%%%%%%%%%%%%%%%%%%%%%%%%%%%%%%%%%%%%%%

\appendix

\section{Tools of pseudo-differential operators and applications to eigenfunctions}
 For $ m \in \R$, we define $ T^m $ as the vector space of symbols $ q(x,\xi) \in C^\infty( \R^d \times \R^d ) $ such that  for all $ \alpha \in \N^d $ and $ \beta \in \N^d $, there exists a  constant $ C_{\alpha, \beta} >0$ such that for all $ ( x , \xi ) \in \R^d \times \R^d $, we have 
\begin{equation*}
  | \partial ^\alpha_x  \partial ^\beta_\xi q(x, \xi ) | \leq C_{\alpha, \beta}  (1+|x|+|\xi|)^{m-\beta}.
 \end{equation*} 
 
 Similarly, let $S^m $ be the vector space of symbols satisfying
\begin{equation*}
  | \partial ^\alpha_x  \partial ^\beta_\xi q(x, \xi ) | \leq C_{\alpha, \beta}  (1+|\xi|)^{m-\beta}.
 \end{equation*} 
 
For $ q \in S^m \cup T^m $ and $ h> 0 $, let $ Op_h(q) $ be the operator defined by
\begin{eqnarray*}
Op_h (q) f(x) &=& (2 \pi h )^{-d} \int_{\R^d \times \R^d} e^{i(x-y)\cdot  \xi / h } q(x , \xi ) f(y) dy d \xi \\
&=& (2 \pi h )^{-d} \int_{\R^d } e^{ix \cdot \xi/h } q(x , \xi ) \hat{f}(\xi/h) d \xi\\
 & =& (2 \pi )^{-d} \int_{\R^d } e^{ix \cdot \xi } q(x , h\xi ) \hat{f}(\xi) d \xi.
\end{eqnarray*}
 
 Let $ p(x, \xi) = |\xi|^2 + |x|^2 -1 $ and define $ H_h = Op_h(p) \in Op_h( T^2 ) $.

 \subsection{Microlocal analysis for the harmonic oscillator}
In \cite{Martinez}, we can then obtain the following two results:

 \begin{thm}  \label{1-compose}
If $ q_1 \in S^{m_1} $ (respectively $ T^{m_1} $) and $ q_2 \in S ^{m_2} $ (respectively $ T^{m_2} $) then there exists a symbol $ q \in S^{m_1+m_2} $ (respectively $ T^{m_1+m_2} $) such that
\begin{equation*}
 Op_h(q_1) \circ Op_h(q_2) = Op_h(q) 
 \end{equation*}
with
 \begin{equation*}
q = \sum _{| \alpha | \leq N } \frac{ h^{| \alpha | } } { i^{| \alpha| }} \ \partial^\alpha_\xi q_1 \partial^\alpha_x q_2 + h^{N+1} r_N  
\end{equation*}
where $r_N \in S^{m_1+m_2-(N+1) }$  (respectively  $T^{m_1+m_2-(N+1) }$).
\end{thm} 

 \begin{thm} \label{1-operator}
If $ q(x,\xi) \in S^0 $ then for any $ s \in \R $, there exists a constant $ C > 0 $ such that for any $ h \in ]0,1] $ and any  $ u \in H^s(\R^d) $,
\begin{equation*}
\|   Op_h( q ) u \|_{H^s(\R^d)} \leq C \|u\|_{H^s(\R^d)}.
\end{equation*}
\end{thm}
We then state the following property which will allow us to invert the harmonic oscillator modulo a very regular remainder term.

Let $ \delta > 0 $ and define a function $ \eta \in C^\infty ( \R^d ) $ such that  
\begin{equation*}
\eta (x) = \left\{
    \begin{aligned}
  &       0 \ \mbox{ if } \ |x| \leq 1 + \delta,  \\
   &     1 \ \mbox{ if } \ |x| \geq 1 + 2 \delta.
    \end{aligned}
\right.
 \end{equation*}

 \begin{prop}  \label{1-inverse}
For any $ N \in \N^\star$, there exist two pseudo-differential operators $ E_N \in Op_h( T^{-2} ) $ and $ R_N \in Op_h( T^{-(N+1)}) $ such that 
\begin{equation*}
E_N \circ H_h = \eta + h^{N+1 } R_N.
\end{equation*}
\end{prop}

\begin{proof} Recall that  $ p(x, \xi) = |\xi|^2 + |x|^2 -1 $ and set 
\begin{equation*}
f_0  =  \frac{\eta}{p} \in T^{-2}.
\end{equation*}
For $ n \geq 1 $, we define $ f_n $ by induction in the following way: 
\begin{equation*}
f_n  = -   \frac{1}{p} \sum_{|\alpha|+j =n , j \neq n } \frac{1}{i^{|\alpha|}} \partial^\alpha_\xi f_j \partial^\alpha_x p \ \ \in T^{-2-n}.
\end{equation*}
Finally we set
\begin{equation*}
 E_N = Op_h \Big( \sum _{0 \leq j \leq N } h^j f_j \Big).
\end{equation*}
Then, by Proposition~\ref{1-compose},
\begin{eqnarray*}
E _N   \circ  H_h & =& Op_h \Big( \sum_{0 \leq j \leq N } h^j f_j \Big) \circ Op_h( p ) \\
 & =& Op_h \Big( \sum_{ |\alpha| + j \leq N } \frac{h^{|\alpha|+j}}{i^{|\alpha|}}  \partial^\alpha_\xi  f_{j}  \partial^\alpha_x p  + h^{N+1} r_N  \Big) \\
 & =& Op_h   \Big( \sum_{ |\alpha| + j \leq N } \frac{h^{|\alpha|+j}}{i^{|\alpha|}}  \partial^\alpha_\xi  f_{j}  \partial^\alpha_x  p  \Big)  + h^{N+1} R_N  
\end{eqnarray*}
with  $ R_N = Op_h(r_N) \in Op_h \left( T^{-(N+1)}  \right)$. Therefore
\begin{eqnarray*}
\sum_{ |\alpha| + j \leq N } \frac{h^{|\alpha|+j}}{i^{\ell\alpha|}}  \partial^\alpha_\xi  f_{j}  \partial^\alpha_x  p  & = &f_0 p + \sum_{  1 \leq \ell   \leq N } \sum_{|\alpha|+j = \ell }\frac{h^{\ell}}{i^{|\alpha|}} \ \partial^\alpha_\xi  f_j  \partial^\alpha_x  p 
\\ & =& \eta + \sum_{  1 \leq \ell   \leq N} h^{\ell} \Big(  \sum_{|\alpha|+j = \ell , j \neq  \ell }\frac{1}{i^{|\alpha|}} \ \partial^\alpha_\xi  f_j  \partial^\alpha_x  p  +  f_{\ell}.p \Big)
\\ & =& \eta,
\end{eqnarray*}
which was the claim.
 \end{proof}

 \subsection{A localisation property of the eigenfunctions of the harmonic oscillator in dimension $d \geq1$}
We can now establish a localisation property of the eigenfunctions of the harmonic oscillator.
  \begin{prop} \label{1-propre}
 Let $d\geq 1$.  For all integers $ K , N \geq 1$, all $ c > 1 $ and all $ 1 \leq p \leq + \infty $, there exists a constant $ C > 0 $ such that for all $ n \geq 0 $,
\begin{equation*}
\| \langle x \rangle ^{K} h_n \|_{L^p (|x| \geq c \lambda_n) } \leq C  \lambda_n ^
{-N}.
\end{equation*}
\end{prop}

\begin{proof}
By definition, we have $ ( -\Delta + |x|^2 -\lambda_n^2) h_n = 0 $. Therefore, if we set   $ h =  \frac{1}{\lambda_n^2} $ and $ \Phi ( x) = h_n( \lambda_n x ) $ then  
$$ ( -h^2 \Delta + |x|^2 -1 )  \Phi = 0.$$

Let  $ \delta \ll 1 $ and set $ \chi \in  C^\infty ( \R^d ) $ such that 
\begin{equation*}
\chi(x) = \left\{
    \begin{aligned}
      &  0 \ \mbox{ if } \ |x| \leq 1, \\
     &   1 \ \mbox{ if } \ |x| \geq 1 + \delta,
    \end{aligned}
\right.
\end{equation*}
and set also $ \widetilde{\chi} \in  C^\infty ( \R^d)  $ such that
\begin{equation*}
\widetilde{\chi}(x) = \left\{
    \begin{aligned}
   &     0 \ \mbox{ if  } \ |x| \leq 1 + 2 \delta, \\
   &     1 \ \mbox{ if } \ |x| \geq 1 + 3\delta.
    \end{aligned}
\right.
\end{equation*}
Then
\begin{equation*}
H_h ( \chi \Phi ) = - h^2 \Phi \Delta  \chi   -2h^2  \nabla \chi \cdot  \nabla \Phi .
\end{equation*}
Next, thanks to Proposition~\ref{1-inverse}, we obtain 
\begin{equation*}
 \eta \chi \Phi = - E_N ( h^2 \Phi \Delta \chi   + 2 h^2 \nabla \chi \cdot \nabla \Phi )- h^{N+1} R_N ( \chi \Phi ) .
  \end{equation*} 
And finally
\begin{equation*}
\langle x\rangle ^K \widetilde{\chi} \eta \chi \Phi = - \langle x\rangle ^K\widetilde{\chi} E_N ( h^2 \Phi \Delta \chi  + 2 h^2 \nabla \chi \cdot \nabla \Phi )- h^{N+1} \langle x\rangle ^K \widetilde{\chi} R_N ( \chi \Phi ) .
\end{equation*}
\medskip 

$\bullet$ Estimate of $ \langle x\rangle ^K \widetilde{\chi} E_N (\Phi \Delta \chi) $: We have
\begin{equation*}
\widetilde{\chi}(x) E_N (  \Phi \Delta \chi  ) (x) = \frac{ \widetilde{\chi} (x)  }{ (2 \pi h )^d} \int_{\xi,1 \leq  |y| \leq 1+ \delta } e^{i(x-y)\cdot  \xi / h } E_N(x, \xi ) (\Delta \chi \Phi )(y) dy d \xi. 
\end{equation*}
Since $ |x| > 1 + 2 \delta $, then $ |x-y | > \delta $.   Next, observe that we have
$$ \int_{|x-y| >  \delta } \leq \int_{|x_1-y_1| > \delta }+ \int_{|x_2-y_2| >  \delta }+ \cdots + \int_{|x_d-y_d| >  \delta },$$
so that we are  reduced to treat the term where $ |x_1-y_1 | >  \delta $. \medskip

From $ \frac{h^M}{i^M(x_1-y_1)^M}  \partial^M_{\xi_1}  e^{i(x-y)\cdot  \xi / h  } = e^{i(x-y) \cdot \xi / h  }  $ and from an integration by parts, we deduce
\begin{multline*}
 \langle x\rangle ^K \widetilde{\chi}(x) E_N ( \Phi  \Delta \chi  ) (x) =\\
=  (-1)^M\frac{ \overline{ \chi  }(x) }{(2 \pi h )^d}   \int_{\substack{ \xi,1 \leq  |y| \leq 1+ \delta \\    |x_1-y_1 | >  \delta     }} \frac{h^M   e^{i(x-y)\cdot \xi / h } \langle x\rangle ^K \partial ^M_{\xi_1} E_N(x, \xi ) ( \Phi\Delta \chi )(y)}{i^M(x_1-y_1)^M}  dy d \xi.
  \end{multline*}
Therefore, as $ E_N \in T^{-2} $, we obtain
\begin{equation*}
\big|  \langle  x \rangle  ^K  \widetilde{\chi}(x)   E_N  (\Phi \Delta \chi  ) (x) \big| \leq C   h^{M-d}  | \widetilde{\chi}(x) |  \int_{\xi,1 \leq  |y| \leq 1+ \delta } \frac{ | ( \Phi\Delta \chi)(y)| }{(1+|x|+| \xi| )^{2+M-K}} d \xi dy.
\end{equation*}
We can then assume that $M\geq 1$ is such that $M> 2d+K-1 $, and therefore 
\begin{multline*}
\big|  \langle  x \rangle  ^K  \widetilde{\chi}(x)   E_N(\Phi  \Delta \chi  ) (x)   \big| \leq \\
\begin{aligned}
 & \leq C   h^{M-d}   \| \Delta \chi  \Phi \|_{L^1(\R^d)}   \frac{| \widetilde{\chi}(x) | }{(1+|x|)^{1+M-K-d}}  \int_{\R^d } \frac{d \xi }{(1+ | \xi|)^{d+1}} \\
&\leq C   h^{M-d}   \|   \Phi \|_{L^2(\R^d)}   \frac{1 }{(1+|x|)^{1+M-K-d}} .
\end{aligned}
\end{multline*}
And finally, for all integers $M, K \geq 1$, there exists a constant $ C > 0 $ such that for all $ h \in ]0,1] $,
\begin{equation*}
\|\langle  x \rangle  ^K \widetilde{\chi} E_N (\Phi  \Delta \chi  ) \|_{L^p(\R^d)} \leq C   h^{M-d}  \|   \Phi \|_{L^2(\R^d)}.
\end{equation*}
\medskip

$\bullet$ Estimate of $  \langle x\rangle ^K \widetilde{\chi} E_N ( \nabla \chi \cdot \nabla \Phi ) $: Since $ \Phi $ satisfies $ (-h^2 \Delta + |x|^2) \Phi = \Phi $, then $ h \| \nabla \Phi \|_{L^2(\R^d) } \leq \| \Phi \|_{L^2(\R^d)}$. As a consequence we can proceed as the first term to get the same kind of estimate. \medskip
 
$\bullet$ Estimate of $ h^{N+1} \langle x\rangle ^K  \widetilde{\chi} R_N ( \chi \Phi ) $: We have
\begin{equation*}
h^{N+1} \langle x\rangle ^K \widetilde{\chi}(x)  R_N ( \chi \Phi )(x) =
 h^{N+1}   \frac{1}{(2 \pi )^d}  \int_{\R^d } e^{i(x-y)\cdot \xi  } \langle x\rangle ^K \widetilde{\chi}(x) r_N(x,h\xi)   \mathcal{F} (  \chi  \Phi )(\xi) d \xi  
\end{equation*}
with
\begin{equation*}
 \langle x\rangle ^K \widetilde{\chi}(x) r_N (x, \xi ) \in T^{-N-1+K}  \subset T^0 \subset S^0 \quad \mbox{for}\quad  N \geq K-1.
 \end{equation*}
Using Theorem \ref{1-operator} and the Sobolev embeddings, we obtain 
\begin{eqnarray*}
 \| h^{N+1} \langle x\rangle ^K \widetilde{\chi}  R_N ( \chi \Phi ) \|_{L^p(\R^d)}
& \leq  &h^{N+1} \|  \langle x\rangle ^K \widetilde{\chi}  R_N ( \chi \Phi ) \|_{H^{d/2+1}(\R^d)}\\
& \leq  & C  h^{N+1}   \|  \chi \Phi \|_{H^{d/2+1}(\R^d)}\\
& \leq  & C   h^{N+1}  \|  \Phi \|_{H^{d/2+1}(\R^d)}.
\end{eqnarray*}
Finally, we obtain that for all integers $K, N \geq 1$, for all $ p \in [1,+ \infty] $ and $ c > 1 $, that there exists a  constant $ C> 0 $ such that for all $ 0 < h \leq 1$, we have  
\begin{equation*}
 \| \langle x\rangle ^K \Phi \|_{L^p( |x| \geq c ) } \leq C    h^N  \|  \Phi \|_{H^{d/2+1}(\R^d)} .
\end{equation*}
Then, returning to the initial variable, we get that  for all integers $K, N \geq 1$, for all $ p \in [1,+ \infty] $ and $ c > 1 $, there exists a constant $ C > 0 $ such that for all $ 0 < h \leq 1$ and $ n \in \N $ with $ h = \frac{1}{\lambda_n^2} $, we have
\begin{eqnarray*}
 \| \langle  \sqrt{h} x \rangle ^K h_n \|_{L^p( |x| \geq c \lambda_n ) } & \leq & C    h^{N-d/(2p)-d/4-1/2}  \|  h_n \|_{H^{d/2+1}(\R^d)} \\
 & \leq  &C    h^{N-d/(2p)-d/4-1/2}  \|  h_n \|_{\mathcal{H}^{d/2+1}(\R^d)} \\
 & \leq & C    h^{N-d/(2p)-d/2-1}  \|  h_n \|_{L^2(\R^d)}.
\end{eqnarray*}
But
\begin{eqnarray*}
\| \langle  x \rangle ^K h_n \|_{L^p( |x| \geq c \lambda_n ) } & \leq &\| h_n \|_{L^p( |x| \geq c \lambda_n ) } + h^{-K/2} \| \langle  \sqrt{h} x \rangle ^K h_n \|_{L^p( |x| \geq c \lambda_n ) } \\
 & \leq & C    h^{N-d/(2p)-d/2-1-K/2}  \|  h_n \|_{L^2(\R^d)}\\
  & \leq& C    \lambda_n^{-2N+d/p+d+K+2}  \|  h_n \|_{L^2(\R^d)}
\end{eqnarray*}
which concludes the proof. 
\end{proof}

   \subsection{Functional and pseudo-differential calculus}
The following  property of functional calculus  explains that some operators can be approximated by pseudo-differential operators.
 
 \begin{prop}  \label{1-approx} Let $ \Phi \in C^\infty_0 ( \R ) $ and $ \chi_2 \in C^\infty_0 ( \R^d ) $ with $ \chi_2 (x) = 1 $ for $ x \in B(0,1^+ )$. Then for all $ N \in \N^\star$ and  $ s \geq 0 $, there is a constant $ C_{N,s } > 0  $ such that for all $ h \in ]0,1 ] $ and $ u \in L^2(\R^d) $,
\begin{equation*}
 \| \Phi \big( |x|^2 + |h\nabla|^2 \big)u - \sum_{j=0}^{N-1} h^j  Op_h(  \Psi_j ) \chi_2  u \|_{H^s(\R^d)} \leq C_{N,s} h^{N-s} \|u\|_{L ^2(\R^d)},  
 \end{equation*} 
where $ \Psi_0(x, \xi ) = \Phi ( |x|^2 +  |\xi| ^2  ) $, $ Supp ( \Psi_j ) \subset \big\{ (x, \xi ) : |x|^2 + |\xi|^2 \in Supp ( \Phi ) \big\} $ and $ \Psi_j \in T^{-j} \subset S^0 $.
\end{prop}

\begin{proof} We use  \cite[Proposition 2.1]{burq2} (notice that {\em stricto sensu} \cite[Proposition~2.1]{burq2} is only for a semi-classical Laplace operator, but the proof given applies also here). If  $ \chi_1 \in C^\infty_0 ( \R^d ) $ is such that $ \chi _1 \chi_2 =  \chi_1 $  then
\begin{equation*}
\big\| \Phi \big( |x|^2 + |h\nabla|^2 \big) \chi_1 u - \sum_{j=0}^{N-1} h^j Op_h( \Psi_j(x , \xi ) ) \chi_2 u \big\|_{H^s(\R^d)} \leq C_{N,s} h^{N-s} \|u\|_{L ^2(\R^d)}, 
\end{equation*} 
with $\Psi_0(x, \xi ) = \Phi ( |x|^2 + |\xi|^2  ) $, $ Supp ( \Psi_j ) \subset \big\{ (x, \xi ) : |x|^2 + |\xi|^2 \in Supp ( \Phi ) \big\}$ and $ \Psi_j \in T^{-j} $. Then, it is enough to show that  
\begin{equation*}
 \big\| \Phi ( |x|^2 + |h\nabla| ^2 )(1- \chi_1 ) u \big\|_{H^s(\R^d)} = \mathcal{O}(h^\infty )  \|u\|_{L ^2(\R^d)}.
 \end{equation*}
We choose  $ \chi _1 =1 $ on $ B(0,R), R>1$, and the result then follows from Proposition~\ref{1-propre}.
\end{proof}

  \subsection{A continuity result}

Thanks to the functional calculus, it is clear that $\dis \| \chi \big(  \frac{H}{N^2}  \big) u \|_{\mathcal{H}^s(  \R^d )} \leq C \|u\|_{\mathcal{H}^s(\R^d)}$, however the continuity result stated in the next proposition is not straightforward and relies on Proposition~\ref{1-approx}.
   \begin{prop} \label{1-continuous}
For any $ s \geq 0 $, there exists a constant $ C > 0 $ such that for any $ N \in \N ^\star$ and any function $ u \in H^s(\R^d) $,
 \begin{equation*}
\| \chi \Big(  \frac{H}{N^2}  \Big) u \|_{H^s(  \R^d )} \leq C \|u\|_{H^s(\R^d)}.
\end{equation*}
\end{prop}

\begin{proof}   Using \eqref{1-semi-classique1} and \eqref{1-semi-classique2}, it is sufficient to show that: $ \forall s \geq 0 , \ \exists \, C > 0 $ and $ h_0 $ such that  $ \forall \ 0 < h \leq h_0  , \ \forall u \in H^s(\R^d) $
 \begin{equation} \label{1-continuitebis}
\|  \chi (  |x|^2 + |h\nabla|^2   )  u \|_{H^s(  \R^d )} \leq C \|u\|_{H^s(\R^d)}.
\end{equation}
Indeed, recalling that  $\dis \mathfrak{u}(x) =u( \frac{x}{     \sqrt{ h }  } ) $, then 
\begin{multline*}
 \|  \chi \Big(   \frac{H}{N^2}  \Big) u \|_{H^s(  \R ^d )}  \leq \\
\begin{aligned}
& \leq   \|  [ \chi (  |x|^2 + |h\nabla|^2   )  ] \mathfrak{u}  ( \sqrt{h} .  ) \|_{H^s(  \R^d )}\\
& \leq   \|   [ \chi (  |x|^2 + |h\nabla|^2   )  ] \mathfrak{u}  ( \sqrt{h} .  ) \|_{L^2(  \R^d )} + \|  \sqrt{-\Delta}^s  [ \chi (  |x|^2 + |h\nabla|^2   )  ] \mathfrak{u}  ( \sqrt{h} .  ) \|_{L^2(  \R^d )}\\
& \leq   h^{-d/4} \| \mathfrak{u} \|_{L^2( \R^d)} + h^{s/2 -d/4} \|  [ \chi (  |x|^2 + |h\nabla|^2   )  ] \mathfrak{u}   \|_{H^s(  \R^d )}\\
& \leq   \|u\|_{L^2(\R^d)}+ h^{s/2 -d/4} \| \mathfrak{u}  \|_{H^s(   \R^d  ) }\\
& \leq   \|u\|_{  H^s( \R^d ) }.
\end{aligned}
\end{multline*}
By interpolation, we can limit the proof to the case where $ s $ is an integer. Thanks to Proposition~\ref{1-approx} (with $ N=s $), there exists a constant $ C > 0 $ such that for any $ h \in ]0,1] $ and any function $ u \in L^2(\R^d) $, we have 
 \begin{equation*} 
\| \chi (  |x|^2 + |h\nabla|^2    ) u  - \sum_{j=0}^N h^j Op_h (     \Psi_j   ) \chi_2 u \|_{H^s(  \R^d )} \leq C \|u\|_{L^2(\R^d)},
\end{equation*}
with $ Supp ( \Psi_j(x,\xi) ) \subset \big\{ (x,\xi) \;:\, |x|^2 + |\xi|^2 \in Supp ( \chi ) \big\}  $. \medskip

 Thus, to obtain \eqref{1-continuitebis}, it suffices to obtain that for any $ s \geq 0 $, there exist two constants $ C> 0 $ and $ h_0 \geq 1 $ such that for any $ h \in ]0,h_0] $ and any function $ u \in H^s(\R^d) $,
 \begin{equation} \label{1-continuitetris}
\| Op_h (     \Psi_j   ) u \|_{H^s(  \R^d )} \leq C \|u\|_{H^s(\R^d)}.
\end{equation}
Finally, to establish \eqref{1-continuitetris}, it is enough to use Theorem \ref{1-operator} and to notice that
\begin{equation*}
( x , \xi  ) \longrightarrow \chi( |x|^2 +  |\xi| ^2 ) \in S^0 \; \mbox{ and } \;  ( x , \xi  ) \longrightarrow \Psi_j ( x , \xi ) \in S^0.  
\end{equation*}
\end{proof}

%%%%%%%%%%%%%%%%%%%%%%%%%%%%%%%%%%%%%%%%%%%%%%%%%%%%%%%%%%%%%%%%%%%%%%%%

\end{document}